\documentclass[a4paper,11pt]{article}
\usepackage[left=3cm, right=3cm, top=3cm, bottom=3cm]{geometry}
\usepackage{graphicx}
\usepackage{tikz}
\usepackage{cancel}
\usepackage{amsmath,amsthm,amsfonts,amssymb}
\usepackage{mathrsfs} 
\usepackage{bm}
\usepackage{fancyhdr}
\usepackage{textcomp}
\usepackage{verbatim}
\usepackage{float}
\usepackage[squaren]{SIunits}
\usepackage[affil-it]{authblk}
\usepackage{xcolor,colortbl}
\definecolor{Gray}{gray}{0.9}
\definecolor{GrayS}{gray}{0.85}
\usepackage[labelformat=simple]{subcaption}

\allowdisplaybreaks

\newtheorem{proposition}{\bf Proposition}
\newtheorem{lemma}{\bf Lemma}
\newtheorem{remark}{\bf Remark}

\def\enK{h_k}
\def\am{c_k}
\def\enS{h_s}
\def\llK{\ell_k}
\def\llS{\ell_s}

\def\Ubf{\bm U}
\def\bX{\bm{X}}
\def\bx{\bm{x}}
\def\be{\bm{e}}
\def\bPhi{\bm{\Phi}}
\def\R{\mathbb{R}}

\def\hsfun{{\sf h}_s}
\def\xiE{\xi_{\tiny E}}
\def\puE{p_{u_{\tiny{E}}}}
\def\xiEstar{\xi_{\tiny E}^*}
\def\xistar{{\xi^*}}
\def\ensstar{h_s^*}
\def\enBar{\bar{h}_s}
\def\enHat{\widehat{h}_s}
\def\rxi{r_{\xi}}

\def\hkoP{h_k^{{o,1}}}
\def\hsiP{h_s^{{i,1}}}
\def\hsiS{h_s^{{i,2}}}
\def\lkoP{\ell_k^{\ {o,1}}}
\def\lsiP{\ell_s^{\ {i,1}}}
\def\lsiS{\ell_s^{\ {i,2}}}
\def\uiP{u_{i,1}}
\def\viP{v_{i,1}}
\def\pviP{p_{v_{i,1}}}
\def\puiP{p_{u_{i,1}}}
\def\uoP{u_{o,1}}
\def\uiS{u_{i,2}}
\def\viS{v_{i,2}}

\def\puiS{p_{u_{i,2}}}

\def\domsigma{{\cal D}}
\def\bUps{{\bm{\upsilon}}}
\def\fpD{\widehat{{\bm{\upsilon}}}_{1}}
\def\fpS{\widehat{{\bm{\upsilon}}}_{2}}

\title{\bf On the Sun-shadow dynamics}
\author{Irene Cavallari, Giovanni F. Gronchi, Giulio Ba\`u}
\affil{Dipartimento di Matematica, Universit\`a di Pisa}
\begin{document}

\maketitle
\begin{abstract}
  We investigate the planar motion of a mass particle in a force field
  defined by patching Kepler's and Stark's dynamics. This model is
  called {\em Sun-shadow} dynamics, referring to the motion of an
  Earth satellite perturbed by the solar radiation pressure and
  considering the Earth shadow effect.  The existence of periodic
  orbits of brake type is proved, and the Sun-shadow dynamics is
  investigated by means of a Poincar\'e-like map defined by a quantity
  that is not conserved along the flow.  We also present the results
  of our numerical investigations on some properties of the
  map. Moreover, we construct the invariant manifolds of the
  hyperbolic fixed points related to the periodic orbits of brake
  type. The global picture of the map shows evidence of regular and
  chaotic behaviour.
\end{abstract}

\section{Introduction}

This paper deals with the study of a mass particle moving under the
alternated action of two different force fields: one of
Kepler's problem, the other of Stark's problem, where a
constant force is added to the central one.  This is a basic model to
study the short-period evolution of an Earth satellite which
alternately spends some time in the Earth shadow and some time in the
region where the solar radiation pressure acts.  We call this model
{\em Sun-shadow} dynamics.
  
The solar radiation pressure (srp) can become the main perturbation to
be added to the monopole term of the Earth, when the {\em area to
  mass} ratio of the satellite is large enough.  The importance of the
srp for accurately predicting the motion of artificial satellites of
the Earth was first shown in \cite{Musen_1960} and
\cite{PJShapiro_1960}.  Musen developed an analytic theory that was
applied to the Vanguard I satellite. In \cite{PJShapiro_1960} the
motions of the Echo balloon and the Beacon satellite were numerically
propagated. In both works it was found that the srp can seriously
affect the lifetime of an Earth satellite, especially when a particular
resonance
condition is satisfied. Since these pioneering studies, the effects of
srp have been investigated by many authors. For example, the long and
short-period variations of the orbital elements are discussed in
detail in \cite{MiNoFa_1987}.
  
A relevant aspect related to the srp perturbation is the passage
through the Earth shadow. Kozai \cite{Kozai_1963} was among the
first authors to treat the eclipse perturbation and understand its
effects. He developed a semi-analytic method to obtain the
first-order variations of the orbital elements when the srp is
switched off inside the Earth shadow. It is worth noting that, in
general, the accumulation of these short-period effects produces a
long-period drift of the semi-major axis, see \cite{MiNoFa_1987}.
Similarly, in \cite{Lidov_1969} Lidov described a semi-analytic
method to compute the secular variation of the osculating
parameters. His study showed that they all oscillate periodically,
except in two limiting cases in which either the argument of
periapsis or both the semi-major axis and the eccentricity vary
monotonically. Also Ferraz-Mello \cite{FMello_1972} studied the
possible secular effects induced by srp with the Earth shadow, by
developing an analytic theory in the Hamiltonian formalism; he
found that the angular drifts of the longitude, perigee, and
node are quite small.  A more recent paper, by Hubaux and Lema\^itre
\cite{HubLem_2013}, showed that successive crossings of the shadow
for long time spans (in the order of 1000 years) cause significant
oscillations of the orbital elements, with amplitudes and
frequencies that depend on the area-to-mass ratio. Moreover,
numerical experiments carried out in \cite{Hub_etal_2013} indicate
that the passage through the shadow is a source of instability for
space debris with high area-to-mass ratio at geostationary
altitudes.
  
The idea which inspired this paper emerges in
\cite{Beletsky_2001}. Beletsky proposed to apply Kepler's dynamics
inside the Earth shadow, and Stark's dynamics outside of it. His
qualitative analysis of the problem pointed out that the orbital
energy has leaps each time the shadow is crossed and that the argument
of periapsis, the semi-major axis and the eccentricity have
long-period oscillations. More precisely, the semi-major axis
decreases and the eccentricity increases while the apse line moves
away from the direction of the srp force. In this work we try to go
deeper into the subject. We consider the two-dimensional case where
the srp force lies in the plane of motion of the satellites. After
reviewing Stark's dynamics following \cite{Beletsky_1964} and
\cite{Beletsky_2001}, we describe some features occurring when we
alternate it with the dynamics of Kepler's problem, using separable
variables for the Hamilton-Jacobi equations of both problems. In
particular, we prove the existence of periodic orbits of brake type, which are close to the unstable brake periodic orbits of Stark's dynamics.
For a further description we introduce a Poincar\'e-like map
$\mathfrak{S}$, that we call Sun-shadow map. For this purpose, we
choose a section $\Sigma$ in the boundary of the shadow region, and
fix a quantity that has the same value when the section is crossed
with the right orientation (but is not conserved along the
flow). Then, we present the results of our numerical investigations:
we describe the domain of $\mathfrak{S}$, prove that the fixed points
related to the periodic orbits of brake type are hyperbolic, and
compute the stable and unstable manifolds of these points. These
manifolds are made of several connected components because there are
orbits that either collide with the Earth or go to infinity, therefore
they do not go back to $\Sigma$. Finally, a global picture of the map
is drawn, showing evidence of regular and chaotic behaviour.
    
The paper is organized as follows. In Section~\ref{s:KS} we
introduce Kepler's and Stark's dynamics using separable
coordinates for the Hamilton-Jacobi equations of the two problems. In
Section~\ref{s:Stark} there is the review of Stark's problem. We
investigate the alternation of the two different dynamics in
Section~\ref{s:sunshadow}, proving the existence of a family of periodic
orbits. In Section~\ref{s:ssmap} we define the Sun-shadow map and 
describe our numerical investigations.

\section{Kepler's and Stark's dynamics}
\label{s:KS}

We consider Kepler's dynamics, defined by
\begin{equation}
\ddot\bx = -\frac{\mu\bx}{|\bx|^3}, 
\label{shadow}
\end{equation}
with $\mu>0$ the Earth gravitation parameter and $\bx=(x,y)\in\R^2$.
Moreover, we take into account Stark's dynamics, given by
\begin{equation}
\ddot\bx = -\frac{\mu\bx}{|\bx|^3} + f\be_1, \qquad f>0.
\label{sun}
\end{equation}
Both equations \eqref{shadow} and \eqref{sun} can be written in Hamiltonian form, with
Hamilton's functions
\begin{align}
H_k &= \frac{1}{2}(p_x^2+p_y^2) -
\frac{\mu}{\sqrt{x^2+y^2}}, \label{HKep}\\
H_s &= \frac{1}{2}(p_x^2+p_y^2) -
\frac{\mu}{\sqrt{x^2+y^2}} - fx, \label{HSta}
\end{align}
where $p_x,p_y$ are the moments conjugated to $x,y$. Hereafter, the labels $k,s$ will stand for {\em Kepler} and {\em Stark}, respectively.

Besides $H_k$, the angular momentum
\begin{equation}
  C_k = p_y x - p_x y
  \label{am}
\end{equation}
and the Laplace-Lenz vector
\[
\bm{A}_k = \Bigl(-p_y(p_x y - p_y x) - \frac{\mu x}{\sqrt{x^2+y^2}}, p_x(p_xy
- p_yx) - \frac{\mu y}{\sqrt{x^2+y^2}}\Bigr)
\]
are first integrals of Kepler's dynamics.
Note that relation
\[
|\bm{A}_k|^2 = \mu^2 +  2H_kC_k^2
\]
holds among these integrals.
We denote by $L_k$ the opposite of the $x$-component of $\bm{A}_k$.

On the other hand, besides $H_s$, Stark's dynamics has the first integral
\begin{equation}
L_s = p_y(p_x y - p_y x) + \frac{\mu x}{\sqrt{x^2+y^2}} - \frac{f}{2}y^2, 
\label{llS}
\end{equation}
which is a generalisation of $L_k$ (see \cite{Redmond_1964}), but there are no other integrals independent from $L_s$ and $H_s$.

\subsection{Hamilton-Jacobi equations and separation of variables}
\label{hamiltonjacobi}
The coordinate change $(x,y) \mapsto (u,v)$ defined by
\begin{equation}
  x = \frac{u^2-v^2}{2}, \hskip 1.5cm y = uv
\label{uv}
\end{equation}
separates the variables in the Hamilton-Jacobi equations of both
Kepler's and Stark's problems.
Relations \eqref{uv} can be completed to a canonical transformation
leading to new variables $(p_u,p_v,u,v)$:
\begin{equation}
\begin{split}
  &u = \pm\sqrt{x+\sqrt{x^2+y^2}}, \qquad v = y/u,\cr
  &p_u = up_x + vp_y, \qquad p_v = -vp_x + up_y.\cr
  \end{split}
\label{pupv}
\end{equation}
Also a transformation of the time variable $t$ can be performed by
introducing the {\em fictitious time} $\tau$ through the differential relation
\[
\frac{d\tau}{dt}= \frac{1}{u^2+v^2}.
\]
Hereafter, we shall denote with a {\em prime} the derivative with respect to $\tau$.
Hamilton's functions for the two dynamics in these coordinates are
\begin{align}
\mathscr{H}_k &= \frac{p_u^2+p_v^2}{2(u^2+v^2)} - \frac{2\mu}{u^2+v^2}, \label{HKuv} \\ 
\mathscr{H}_s &= \frac{p_u^2+p_v^2}{2(u^2+v^2)} - \frac{2\mu}{u^2+v^2} - \frac{f}{2}(u^2-v^2). \label{HSuv}
\end{align}
Set $\Ubf = (p_u,p_v,u,v)^T$, where $T$ stands for vector transposition. Stark's and Kepler's dynamical systems can be written as 
\begin{equation}
{\Ubf}' = \bX_k({\Ubf}) = \left(
2\enK u, 2\enK v, p_u, p_v \right)^T,
\label{UDynSystK}
\end{equation}
\begin{equation}
{\Ubf}' = \bX_s({\Ubf}) = \left(
2\enS u+2fu^3, 2\enS v-2fv^3, p_u, p_v
\right)^T,
\label{UDynSystS}
\end{equation}
where $\enK$ and $\enS$ are the values of $\mathscr{H}_k $ and $\mathscr{H}_s$ for some given initial conditions. The angular momentum in the new coordinates is
\[
\mathscr{C}_k = \frac{1}{2}(p_v u - p_u v),
\]
while the integrals $L_k,L_s$ become
\begin{align}
 \mathscr{L}_k &= \frac{p_u^2v^2 - p_v^2u^2}{2(u^2+v^2)} +
 \mu\frac{u^2-v^2}{u^2+v^2}, \label{Lkuv} \\
 \mathscr{L}_s &= \frac{p_u^2v^2 - p_v^2u^2}{2(u^2+v^2)} +
 \mu\frac{u^2-v^2}{u^2+v^2} - \frac{f}{2}u^2v^2. \label{Lsuv}
\end{align}

Hamilton-Jacobi equations for the two problems are 
\begin{eqnarray}
 &&\biggl(\frac{\partial W_k}{\partial u}\biggr)^2+\biggl(\frac{\partial W_k}{\partial v}\biggr)^2 = 2\bigl(\enK(u^2+v^2) + 2\mu\bigr),\label{HJK}\\
&&\biggl(\frac{\partial W_s}{\partial u}\biggr)^2+\biggl(\frac{\partial W_s}{\partial v}\biggr)^2 = 2\bigl(\enS(u^2+v^2) + 2\mu\bigr) + f(u^4-v^4),\label{HJS}
\end{eqnarray}
where 
$W_k,W_s$ are the unknown generating functions.
In \eqref{HJK} the variables are separated, so that we obtain 
\begin{equation}
\left\{
\begin{split}
&p_u^2 - 2(\enK u^2 + \mu) = \alpha_k,\cr
&p_v^2 - 2(\enK v^2 + \mu) = -\alpha_k,\cr
\end{split}
\right.
\label{kepl_sepvariables}
\end{equation}
where $\alpha_k$ is an integration constant.
Let  $\llK$ be the value of the integral $\mathscr{L}_k$.
Substituting $p_u^2$, $p_v^2$ given by \eqref{kepl_sepvariables} into $\mathscr{L}_k(p_u,p_v,u,v) = \ell_k$
and simplifying we get
\[
\begin{split}
\alpha_k = 2\llK. 
\end{split}
\]
In a similar way, from equation \eqref{HJS} we get
\begin{equation}
\left\{
\begin{split}
&p_u^2 - \bigl(2(\enS u^2 + \mu) + fu^4\bigr) = \alpha_s,\cr
&p_v^2 - \bigl(2(\enS v^2 + \mu) - fv^4\bigr) = -\alpha_s,\cr
\end{split}
\right.
\label{HJStark}
\end{equation}
with
\[
\alpha_s = 2\llS,
\]
where $\llS$ is the value of the integral $\mathscr{L}_s$.


\section{Trajectories in Stark's dynamics}
\label{s:Stark}

As explained in \cite{Beletsky_1964}, all the possible
trajectories of Stark's dynamics can be divided into four categories,
depending on the values $\llS,\enS$ of $\mathscr{L}_s,\mathscr{H}_s$.

\noindent Relations
\begin{equation}
\begin{split}
&p_u = \frac{du}{d\tau}, \hskip 1cm p_v = \frac{dv}{d\tau}\cr
\end{split}
\label{pupvtau}
\end{equation}
yield
\begin{equation}
	\tau + A_1 = \int \frac{du}{\sqrt{U(u)}},
	\label{timeDef_u}
\end{equation}

\begin{equation}
\tau + A_2 = \int \frac{dv}{\sqrt{V(v)}},
\label{timeDef_v}
\end{equation}
where $A_1$ and $A_2$ are integration constants and
\begin{equation}
  U(u) = fu^4+2\enS u^2+2(\mu+\llS),
  \label{U}
\end{equation}
\begin{equation}
  V(v) = -fv^4+2\enS v^2+2(\mu-\llS)
  \label{V}
\end{equation}
correspond to the expressions of $p_u^2$, $p_v^2$ in
\eqref{HJStark}.

\noindent Conditions
\[
U(u)\geq 0, \qquad V(v)\geq 0
\]
restrict the possible configurations on the basis of the values $\llS$, $\enS$. Let us set
\[
	 \xi = u^2, \hskip 1.5cm \eta = v^2.
\]
The polynomials $U(u)$ and $V(v)$ can be written as 
\[
U(u)= f(u^2-\xi_1)(u^2-\xi_2),\qquad V(v)= -f(v^2-\eta_1)(v^2-\eta_2),
\]
where
\begin{equation}
\begin{split}
&\xi_1 = -\frac{\enS}{f}+\sqrt{\frac{\enS^2}{f^2}-\frac{2(\mu+\llS)}{f}}, \hskip 1cm \xi_2 = -\frac{\enS}{f}-\sqrt{\frac{\enS^2}{f^2}-\frac{2(\mu+\llS)}{f}},\cr
&\eta_1 = \frac{\enS}{f}+\sqrt{\frac{\enS^2}{f^2}+\frac{2(\mu-\llS)}{f}}, \hskip 1cm \eta_2 = \frac{\enS}{f}-\sqrt{\frac{\enS^2}{f^2}+\frac{2(\mu-\llS)}{f}}.\cr
\end{split}
\label{xi12eta12}
\end{equation}
Setting
\begin{equation}
	 {u}_1 = \sqrt{\xi_1}, \hskip 0.8cm {u}_2 = \sqrt{\xi_2},\hskip 0.8cm
	 {v}_1 = \sqrt{\eta_1}, \hskip 0.8cm {v}_2 = \sqrt{\eta_2},
\end{equation}
the roots of $U(u)$ are $\pm {u}_1 , \pm {u}_2$, and those of
$V(v)$ are $\pm {v}_1, \pm {v}_2$.

\noindent It is convenient to study the problem in the $(\llS,
\enS/\sqrt{f})$ plane. Beletsky showed that this plane can be divided
into four regions as shown in Figure~\ref{fig:belRegions}. These
regions do not cover completely the $(\llS, \enS/\sqrt{f})$ plane: in
the remaining part (the brighter one in the figure) the motion is not
possible. Each region is characterised by different types of
trajectories listed in Table \ref{BelRegions}, see
\cite{Beletsky_1964}.  The admissible subsets of the configuration
space are shown in Figure~\ref{fig:admisreg}: these are delimited by
straight lines in the $(u,v)$ plane, and by parabolas in the $(x,y)$
plane. For completeness we added in Appendix~\ref{s:app} Tables
\ref{BelRegB} and \ref{BelRegBP}, describing the features of the
trajectories at boundaries of the regions.

\begin{figure}[h!]
	\centering 
   \includegraphics[clip=true,width=0.9\textwidth]{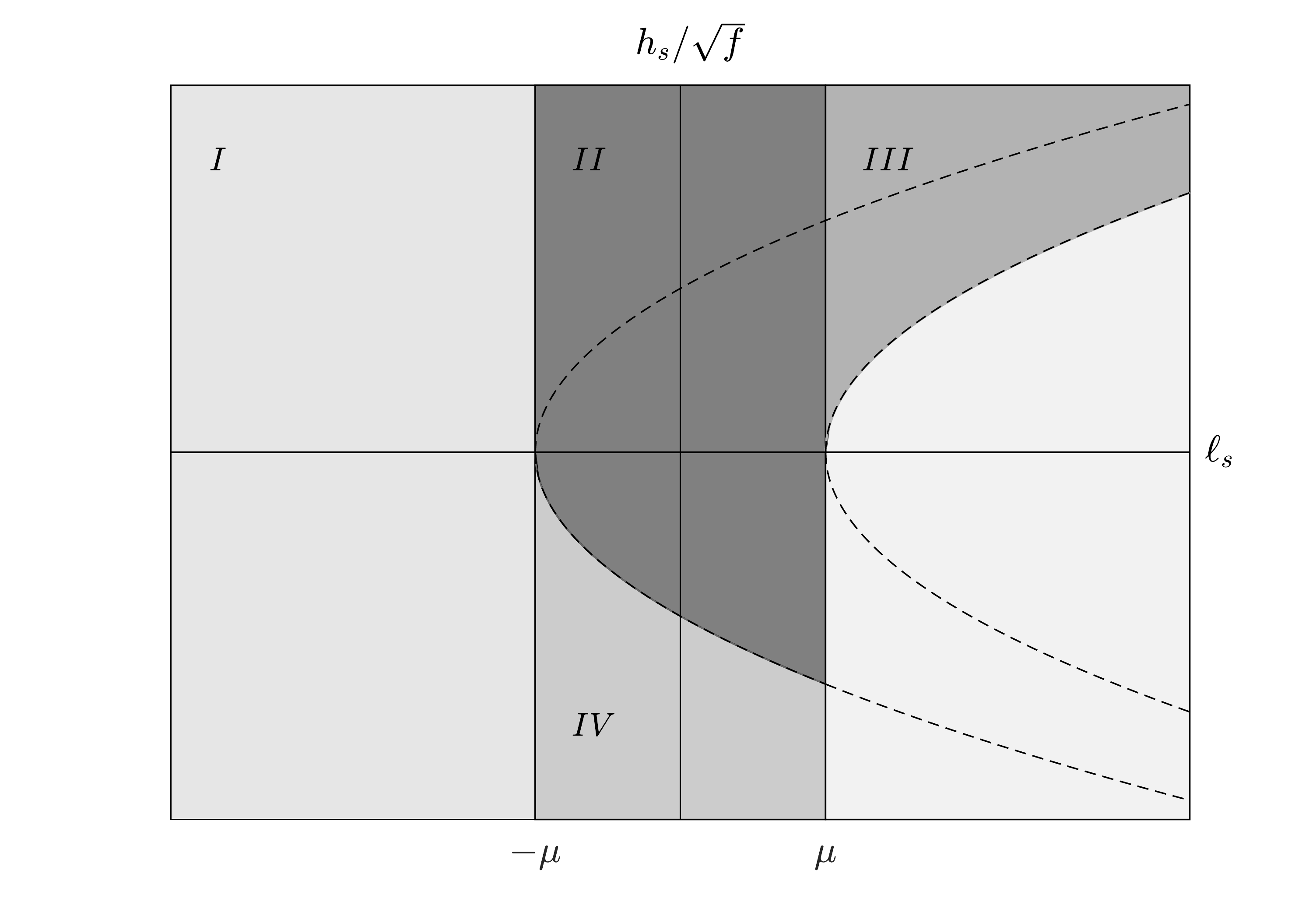}
  \caption{Stark's problem: the four regions in the
    $(\llS, h_s/\sqrt{f})$ plane.}
  \label{fig:belRegions}
\end{figure}

\begin{table}[h!]
  \caption{
    Qualitative description of the trajectories in the $(x,y)$ plane
    for Stark's problem. We denote by $i$ the imaginary unit.}
  \centering
	\begin{tabular}{c!{\color{GrayS}\vrule}c}
		
		\rowcolor{Gray}
		{\textbf{Region $I$}} &
		{$\llS \in (-\infty,-\mu)$; $\enS/\sqrt{f} \in (-\infty,+\infty)$}\\
		\hline
		$V(v)$, $U(u)$ roots & 
		${v}_1 >0$, ${v}_2  \in i\mathbb{R}$, ${u}_1 >0$, ${u}_2 \in i\mathbb{R}$ \\
		$v$, $u$ variable &
		$v\in[-{v}_1,{v}_1]$, $u\in(-\infty,-{u}_1]\cup[{u}_1,+\infty)$ \\
		{trajectories type}&
		{unbounded, self-intersecting, not encircling the origin}\\[3pt]	
                
		\rowcolor{Gray}
		{\textbf{Region $II$}} & 
		{$\llS \in (-\mu,\mu)$; $\enS/\sqrt{f} \in (-\sqrt{2(\mu+\llS)},+\infty)$}\\
		\hline
		$V(v)$, $U(u)$ roots &
		${v}_1 >0$, ${v}_2 \in i\R$, ${u}_1,{u}_2 \in \mathbb{C}\setminus\R$ \\
		$v$, $u$ variable &
		$v\in[-{v}_1,{v}_1]$, $u\in(-\infty,+\infty)$ \\
		{trajectories type} & {unbounded, self-intersecting, encircling the origin}\\[3pt]		
                
		\rowcolor{Gray}
		{\textbf{Region $III$}} & 
		{$\llS \in (\mu,+\infty)$; 	$\enS/\sqrt{f} \in (\sqrt{-2(\mu-\llS)},+\infty)$}\\
		\hline
		$V(v)$, $U(u)$ roots & 
		${v}_1 > {v}_2 >0$, ${u}_1,{u}_2 \in \mathbb{C}\setminus\R$ \\
		$v$, $u$ variable &
		$v\in[-{v}_1,-{v}_2]\cup[{v}_2,{v}_1]$,  $u\in(-\infty,+\infty)$ \\
		{trajectories type}&
		{unbounded, not self-intersecting}\\[3pt]
                
		\rowcolor{Gray}
		{\textbf{Region $IV$}} &
		{	$\llS \in (-\mu,\mu)$; $\enS/\sqrt{f} \in (-\infty,-\sqrt{2(\mu+\llS)}\,)$ }\\
		\hline
		$V(v)$, $U(u)$ roots &
		${v}_1 >0$, ${v}_2 \in i\mathbb{R}$, ${u}_1>{u}_2>0$  \\
		$v$, $u$ variable &
		$v\in[-{v}_1,{v}_1]$, $u\in(-\infty,-{u}_1]\cup[-{u}_2,{u}_2]\cup[{u}_1,+\infty)$ \\
		{trajectories type}& 
		two types: bounded; unbounded,\\ & self-intersecting, not encircling the origin\\
	\end{tabular}
	\label{BelRegions}
\end{table}


\begin{figure}[h!]
	\begin{center}
		\begin{subfigure}{0.43\textwidth}
			\includegraphics[width=1.1\textwidth]{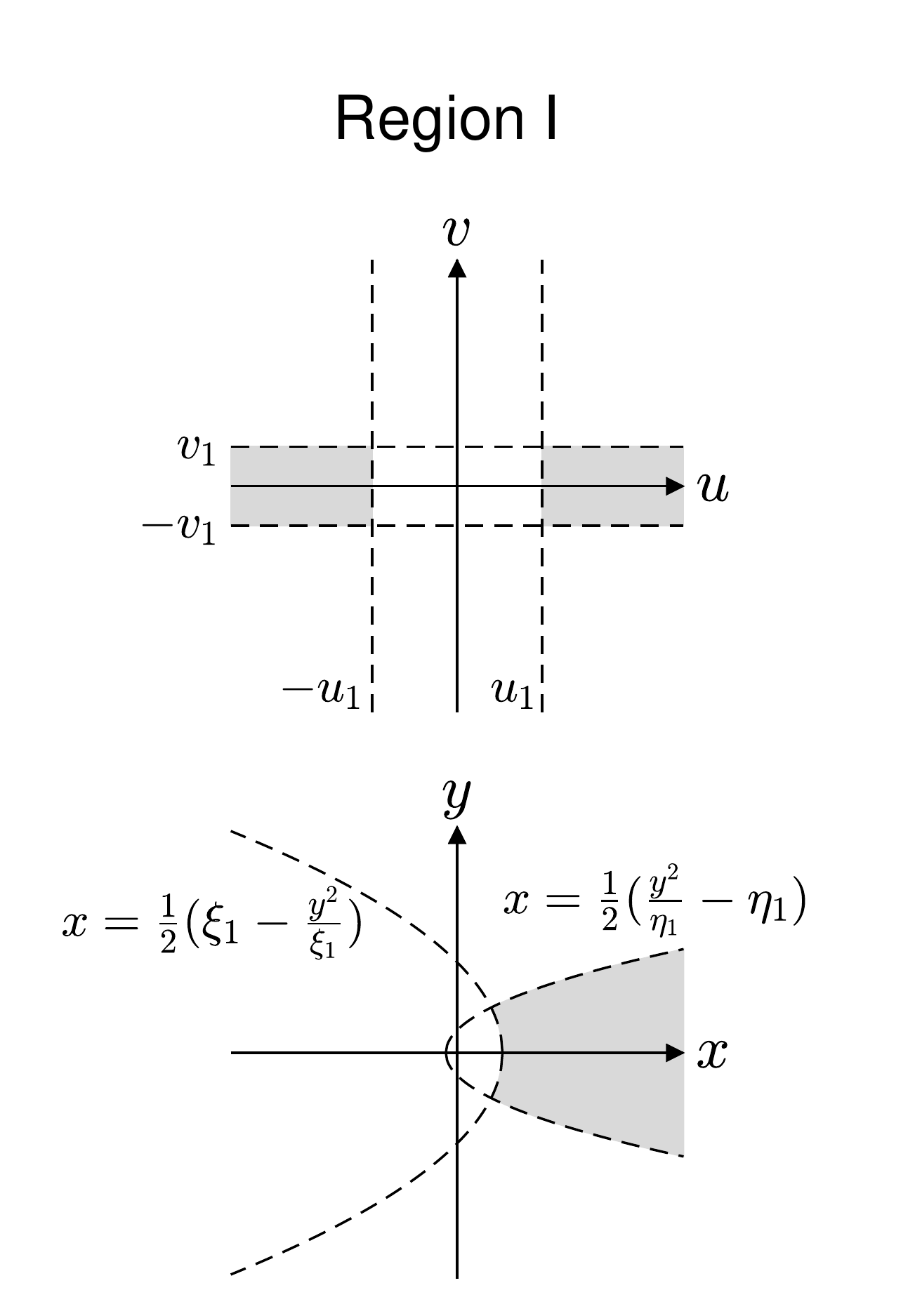}
		\end{subfigure}
		\begin{subfigure}{0.43\textwidth}
			\includegraphics[width=1.1\textwidth]{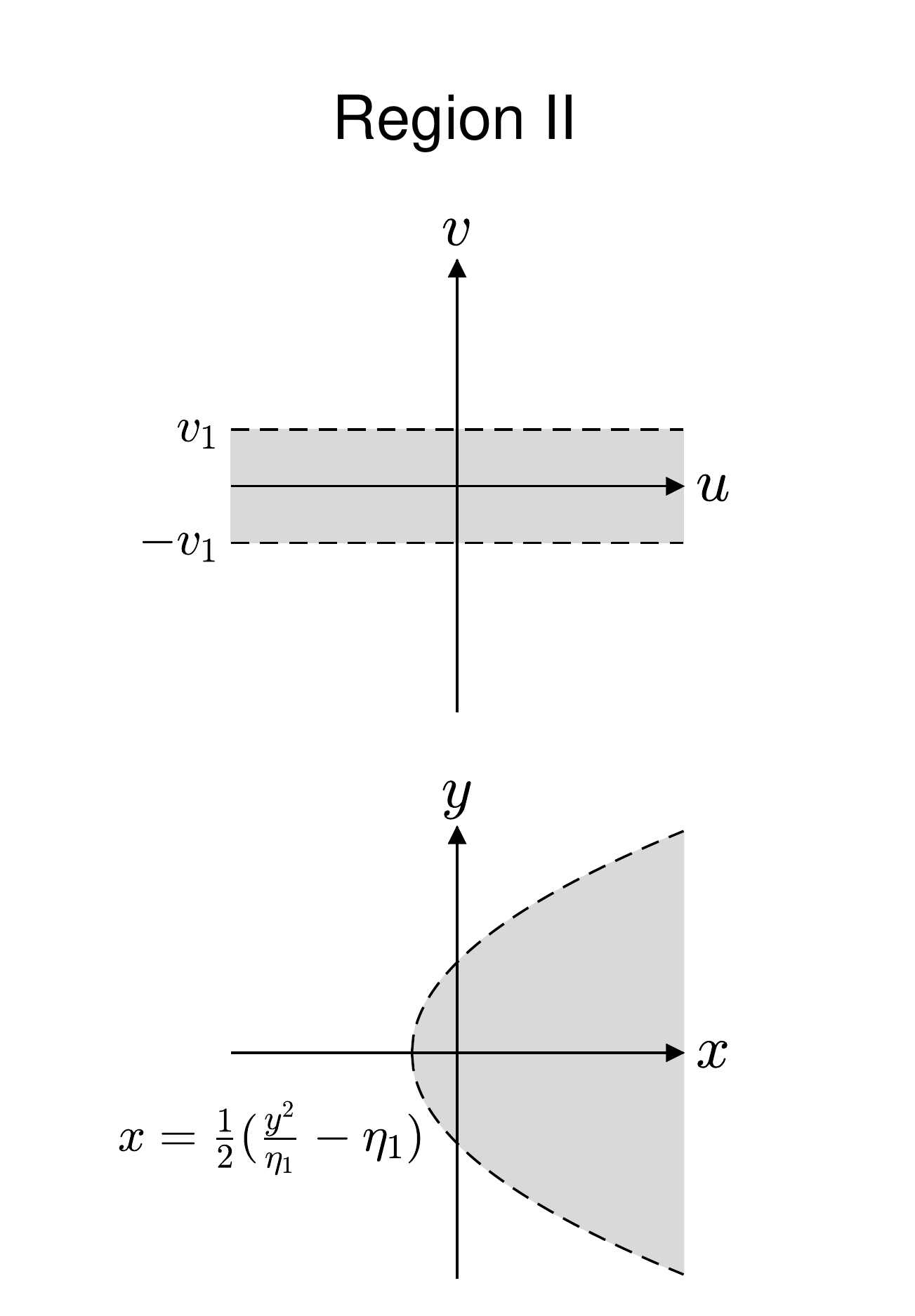}
		\end{subfigure}
		\begin{subfigure}{0.43\textwidth}
			\includegraphics[width=1.1\textwidth]{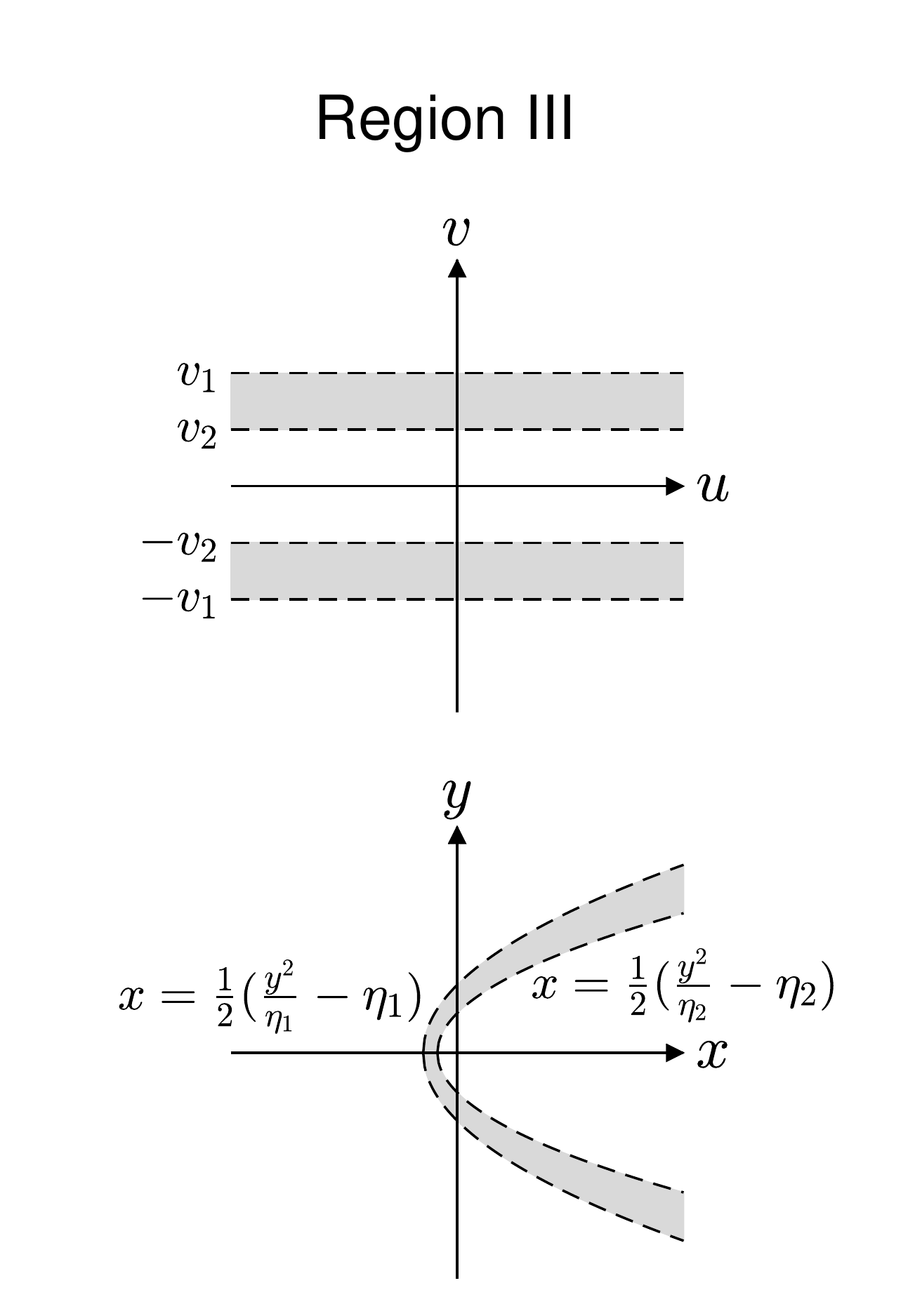}
		\end{subfigure}
		\begin{subfigure}{0.43\textwidth}
			\includegraphics[width=1.1\textwidth]{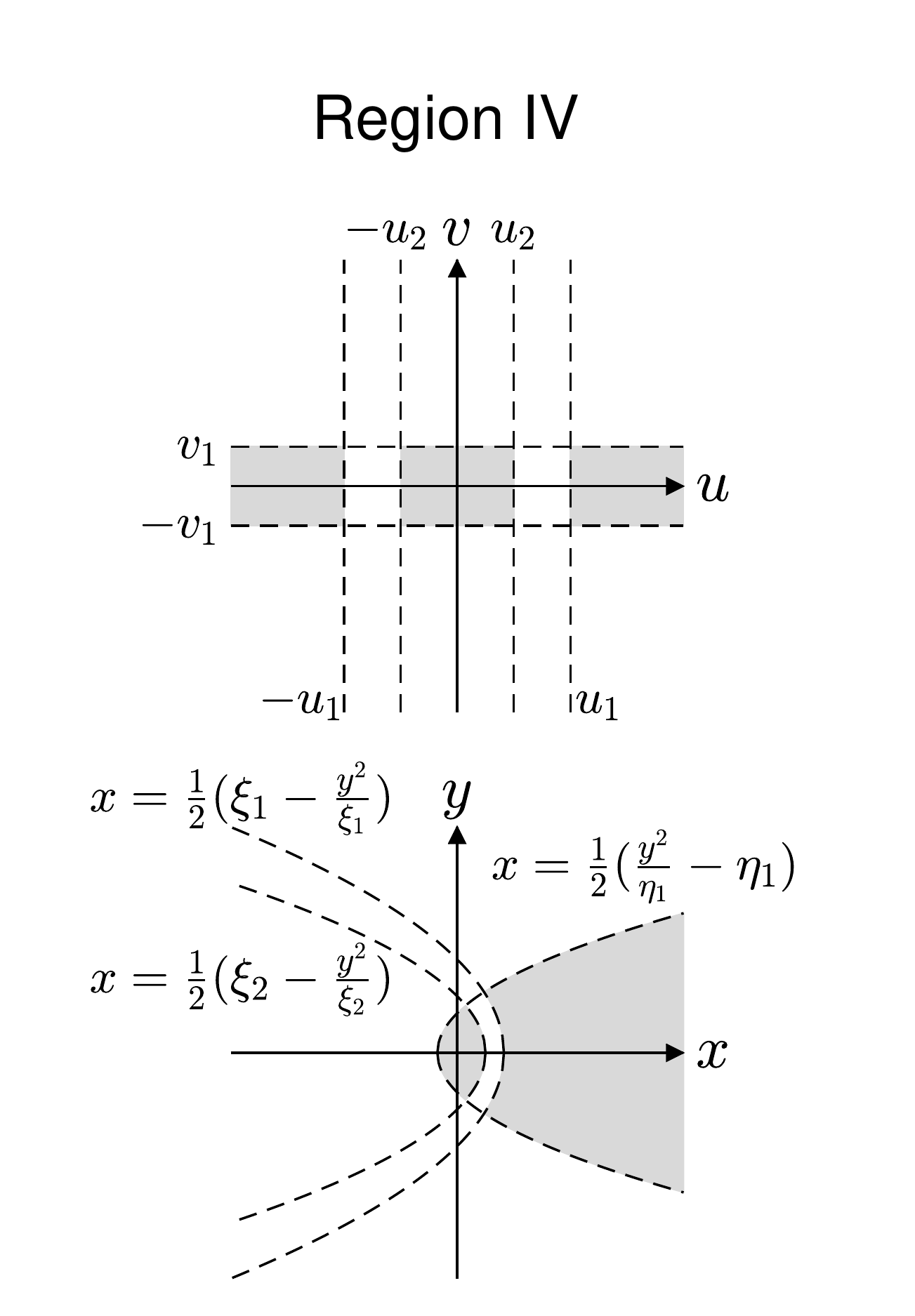}
		\end{subfigure}
		\caption{Admissible subsets of the configuration space in $(x,y)$ and $(u,v)$ planes, depending on the region. These subsets are represented by the grey areas; the dashed lines correspond to their boundaries.}
		\label{fig:admisreg}
	\end{center}	
\end{figure}

\begin{remark}
We can have zero velocity points only for some values $\llS, \enS$ of the integrals $\mathscr{L}_s,\mathscr{H}_s$.
In region $I$ we have the two points
\[
(x,y) \in \{(\xi_1/2-\eta_1/2, \pm\sqrt{\xi_1\eta_1})\},
\]
and in region $IV$ the four points
\[
(x,y) \in \{(\xi_1/2-\eta_1/2, \pm\sqrt{\xi_1\eta_1}),
(\xi_2/2-\eta_1/2, \pm\sqrt{\xi_2\eta_1})\}.
\]
In regions $II$ and $III$ there cannot be zero velocity points.
\label{r:zerovel}
\end{remark}

\begin{remark}
	\label{remark_rIV}
For each region of the $(\llS, \enS/\sqrt{f})$ plane, the $v$-component of an orbit is periodic (and bounded). The $u$-component is periodic (and bounded) only if $(\llS,\enS/\sqrt{f})$ belongs to region $IV$ and $u^2\leq \xi_2$. On the contrary, the $u$-component is unbounded if $(\llS,\enS/\sqrt{f})$ belongs to one among regions $I$, $II$, $III$, or it belongs to region $IV$, and $u^2\geq\xi_1$. 
\end{remark} 

\begin{proposition}
In case $u(\tau)$ and $v(\tau)$ are periodic solutions of \eqref{UDynSystS}, relation
\begin{equation}
\frac{T_v}{T_u}< 1
\label{perrap}
\end{equation}
holds for their periods $T_u, T_v$.
\end{proposition}
\begin{proof}
The periods $T_u$ and $T_v$ of the $u$ and $v$ variables can be written as elliptic integrals:
\[
T_u = \frac{4}{\sqrt{f}}\int_0^{\pi/2}\frac{d\varphi}{\sqrt{a_u^2\cos^2\varphi + b_u^2\sin^2\varphi}},
\qquad
T_v = \frac{4}{\sqrt{f}}\int_0^{\pi/2}\frac{d\varphi}{\sqrt{a_v^2\cos^2\varphi + b_v^2\sin^2\varphi}},
\]
where
\begin{align*}
&a_u^2 = \xi_1 = -\frac{\enS}{f} + \sqrt{\frac{\enS^2}{f^2} - \frac{2(\mu+\llS)}{f}},\qquad
  b_u^2 = \xi_1-\xi_2=2\sqrt{\frac{\enS^2}{f^2} - \frac{2(\mu+\llS)}{f}},\\
&a_v^2 = \eta_1-\eta_2= 2\sqrt{\frac{\enS^2}{f^2} + \frac{2(\mu-\llS)}{f}}, \qquad
  b_v^2 = -\eta_2 = -\frac{\enS}{f} + \sqrt{\frac{\enS^2}{f^2} + \frac{2(\mu-\llS)}{f}}.
\end{align*}
Let $a_u,b_u,a_v,b_v$ be the positive square roots of the previous quantities.
We note that
\begin{equation}
b_u < a_u < b_v < a_v. 
\label{ordered}
\end{equation}
Denoting by $M(a,b)$ the arithmetic-geometric mean of two real numbers $a,b$ (see \cite{Cox_1984}), we have
\[
T_u = \frac{2\pi}{\sqrt{f} M(a_u,b_u)},\qquad T_v = \frac{2\pi}{\sqrt{f} M(a_v,b_v)}.
\]
Since $b_u\leq M(a_u,b_u)\leq a_u$, and $b_v\leq M(a_v,b_v)\leq a_v$,
from \eqref{ordered} we obtain
\[
M(a_u,b_u) < M(a_v,b_v),
\]
that corresponds to \eqref{perrap}.

\end{proof}

\subsection{Unstable periodic orbits of {\em brake} type}
\label{s:brakestark}

There exists a family of unstable periodic orbits of brake type, ${\bx^*} = {\bx^*}(t;\llS)$,
parametrised by $\llS\in (-\mu, \mu)$. 

It is possible to analyse the behaviour of the $u$ and $v$-components
of the trajectory in the reduced phase spaces with coordinates $(u, p_u)$ and
$(v, p_v)$. For this purpose we can take into account the two Hamiltonian dynamics defined by
\[
\mathsf{H_{s_u}} = \frac{p_u^2}{2 u^2} - \frac{2(\mu+\llS)+ fu^4 }{2u^2},
\qquad
\mathsf{H_{s_v}} = \frac{p_v^2}{2v^2} - \frac{2(\mu-\llS)-fv^4}{2v^2},
\]  
obtained from system \eqref{HJStark}. 
$\mathsf{H_{s_u}}$ has the two critical points $(p_u^*,\pm u^*)$, where 
\[
p_u^* = 0, \qquad u^* = \left(\frac{2(\mu+\llS)}{f}\right)^{\frac{1}{4}}.
\]
We can show that they are two unstable equilibrium points for the reduced
dynamics in the $(u,p_u)$ plane. In fact, the Jacobian of the
Hamiltonian vector field induced by $\mathsf{H_{s_u}}$, evaluated in
both critical points, is
\[
D\mathsf{H_{s_u}} = \begin{bmatrix}
0 & \frac{1}{\xistar}\\ 4f & 0 \\
\end{bmatrix},
\]
where 
\begin{equation}
	\xistar={u^*}^2.
	\label{csi_star}
\end{equation}
At these critical points, the value of $\mathsf{H_{s_u}}$ is
\begin{equation}
	\ensstar = -\sqrt{2f(\mu+\llS)},
		\label{hs_star}
\end{equation}
 so that
\begin{equation}
\xistar = -\frac{\ensstar}{f}
\label{xistar2}
\end{equation}
and $\det D\mathsf{H_{s_u}}<0$ at $(p_u^*,\pm u^*)$. 

The level set $\mathsf{H_{s_v}} = h_s^*$ is a closed curve in the
$(v,p_v)$ plane (see Figure~\ref{fig:levelCurves}). Thus, the
$v$-component is periodic. This implies that, if $\llS \in (-\mu,\mu)$
and Stark's Hamiltonian $\mathscr{H}_s$ has the value $\ensstar$, we
have two unstable periodic orbits in the $(u,v)$ plane. They have a
constant value of the $u$-component, equal to $\pm u^*$, and a constant
value of $p_u$, equal to zero. Moreover, they are of brake type
because each of them develops between two zero velocity points. These
are given by $(u^*,\pm v_1^*)$, with $v_1^*=v_1(\ensstar)$, in one
case, and by $(-u^*, \pm v_1^*)$ in the other. In the $(x,y)$ plane,
they correspond to the same periodic orbit $\bx^*$, which is a
parabolic arc and develops between the zero velocity points
$(\xistar/2-\eta_1^*/2, \pm\sqrt{\xistar\eta_1^*})$, where
$\eta_1^*=v_1^{*2}$. Its trajectory is shown in
Figure~\ref{fig:sBOrbit}.

The family of brake orbits $\bx^*$ exists for values $\llS, \enS$
corresponding to the boundary between regions $II$ and $IV$.
On this boundary, the values of $\xi_1$
and $\xi_2$ coincide and are equal to $\xistar$.

\begin{remark}
For $\enS<\ensstar$ (region $IV$), 
we have $\xi_2<\xistar<\xi_1$.
\end{remark}

\begin{figure} [H]
	\begin{center}
		\includegraphics[width=1\textwidth]{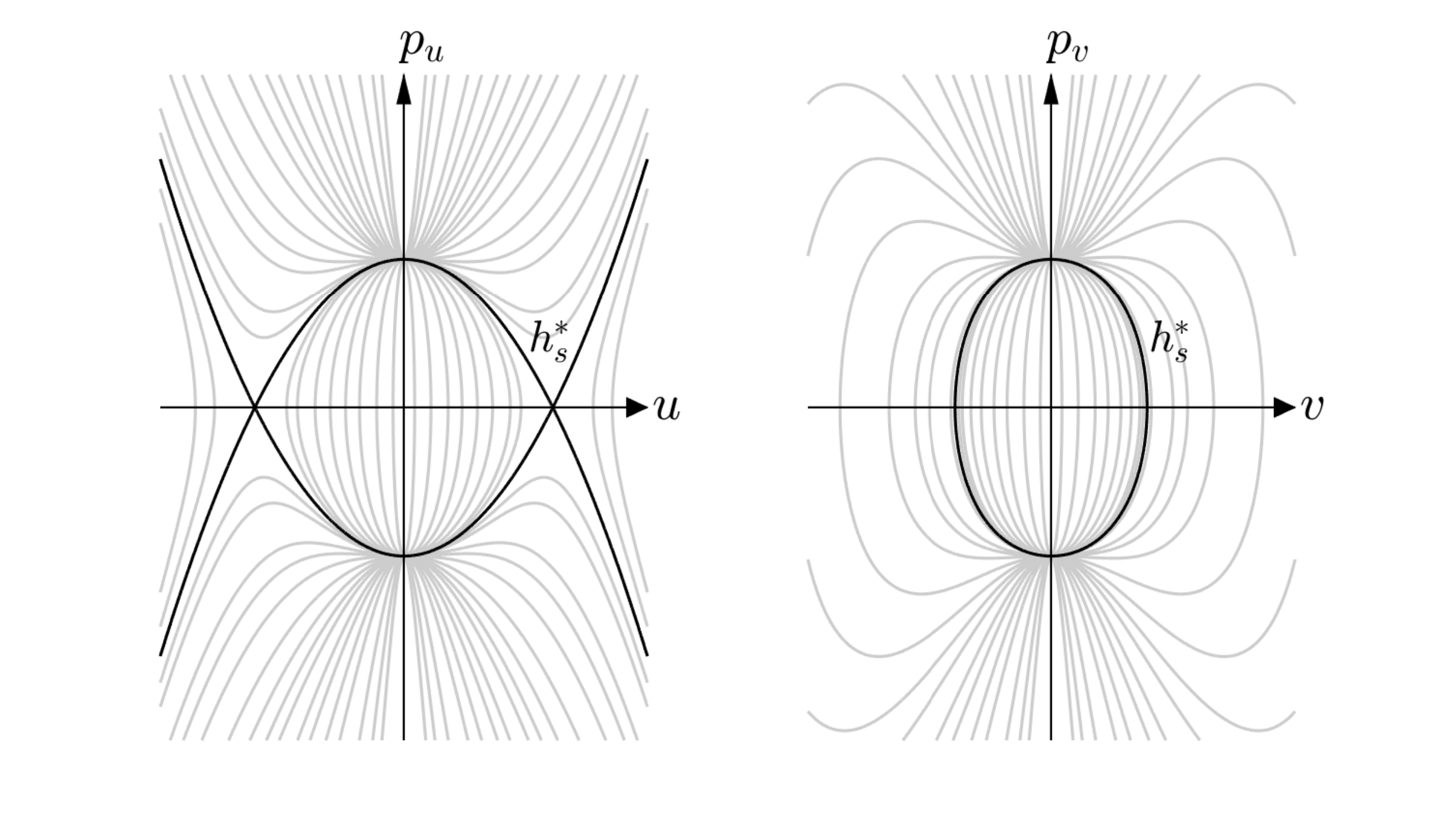}
		\caption{Stark's problem: phase portraits in the $(p_u,u)$ and $(p_v,v)$ planes, for $\llS = 119580$ $\kilo\cubic\meter\per\squaren\second$, $f=9.12\times 10^{-9}$ $\kilo\meter\per\squaren\second$.}
		\label{fig:levelCurves}
	\end{center}	
\end{figure}

\begin{figure} [h!]
	\begin{center}
		\includegraphics[width=0.6\textwidth]{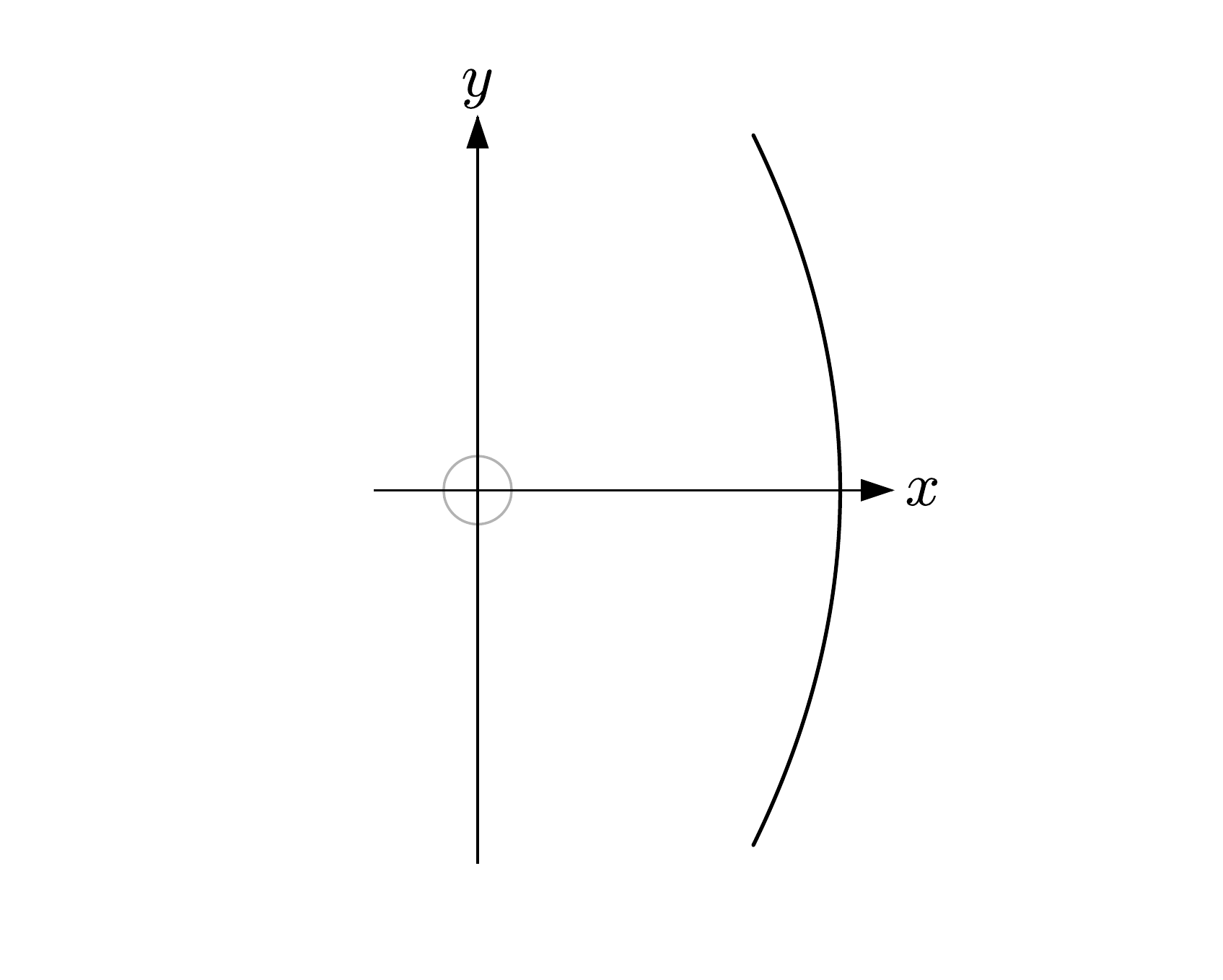}
		\caption{Stark's problem: unstable periodic orbit of brake type.}
		\label{fig:sBOrbit}
	\end{center}	
\end{figure}

\begin{figure}[h!]
	\begin{center}
		\begin{subfigure}{0.45\textwidth}
			\includegraphics[width=1.1\textwidth]{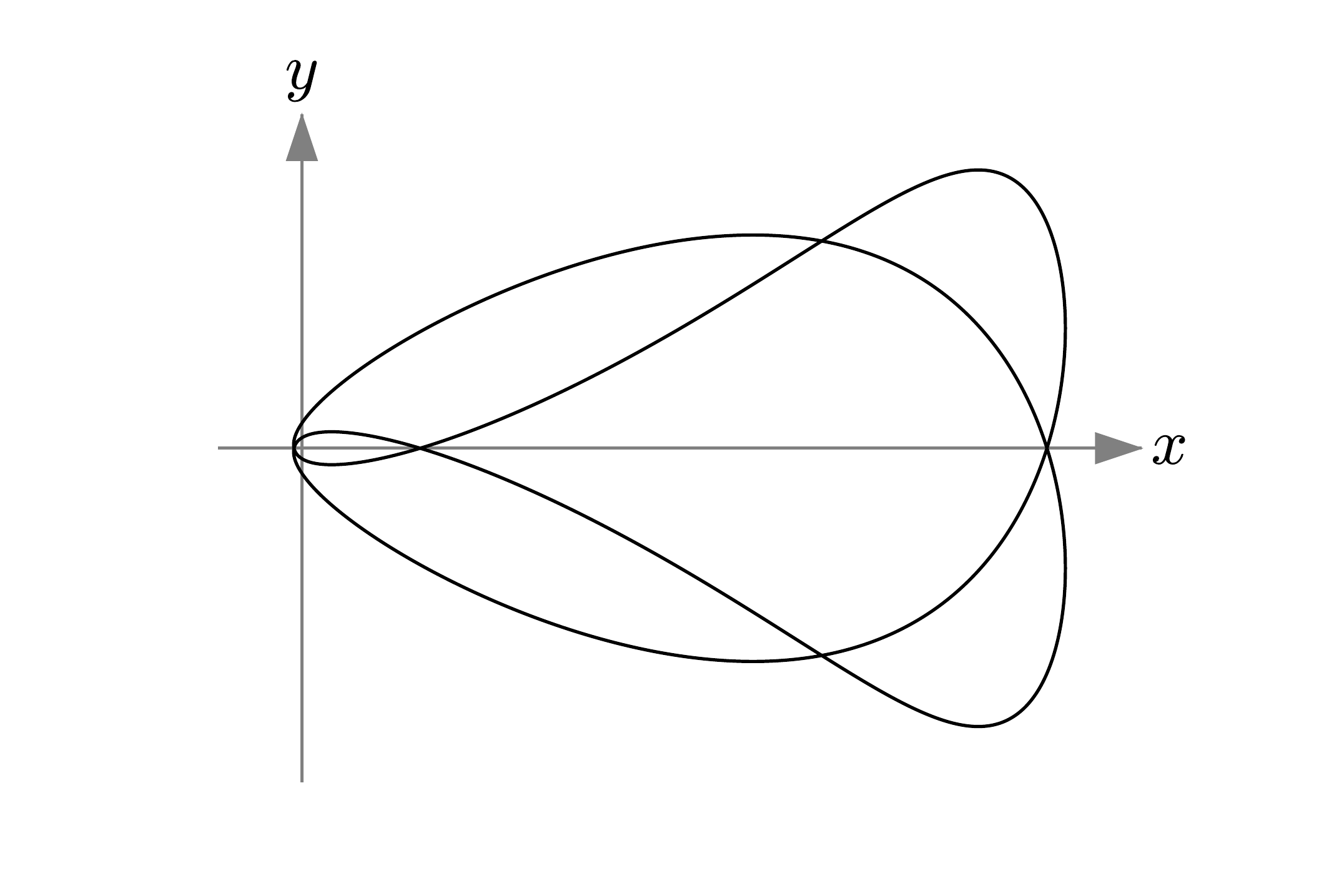}
			\caption{$T_u=2T_v$}
			\label{fig:pSO1}
		\end{subfigure}
		\begin{subfigure}{0.45\textwidth}
			\includegraphics[width=1.1\textwidth]{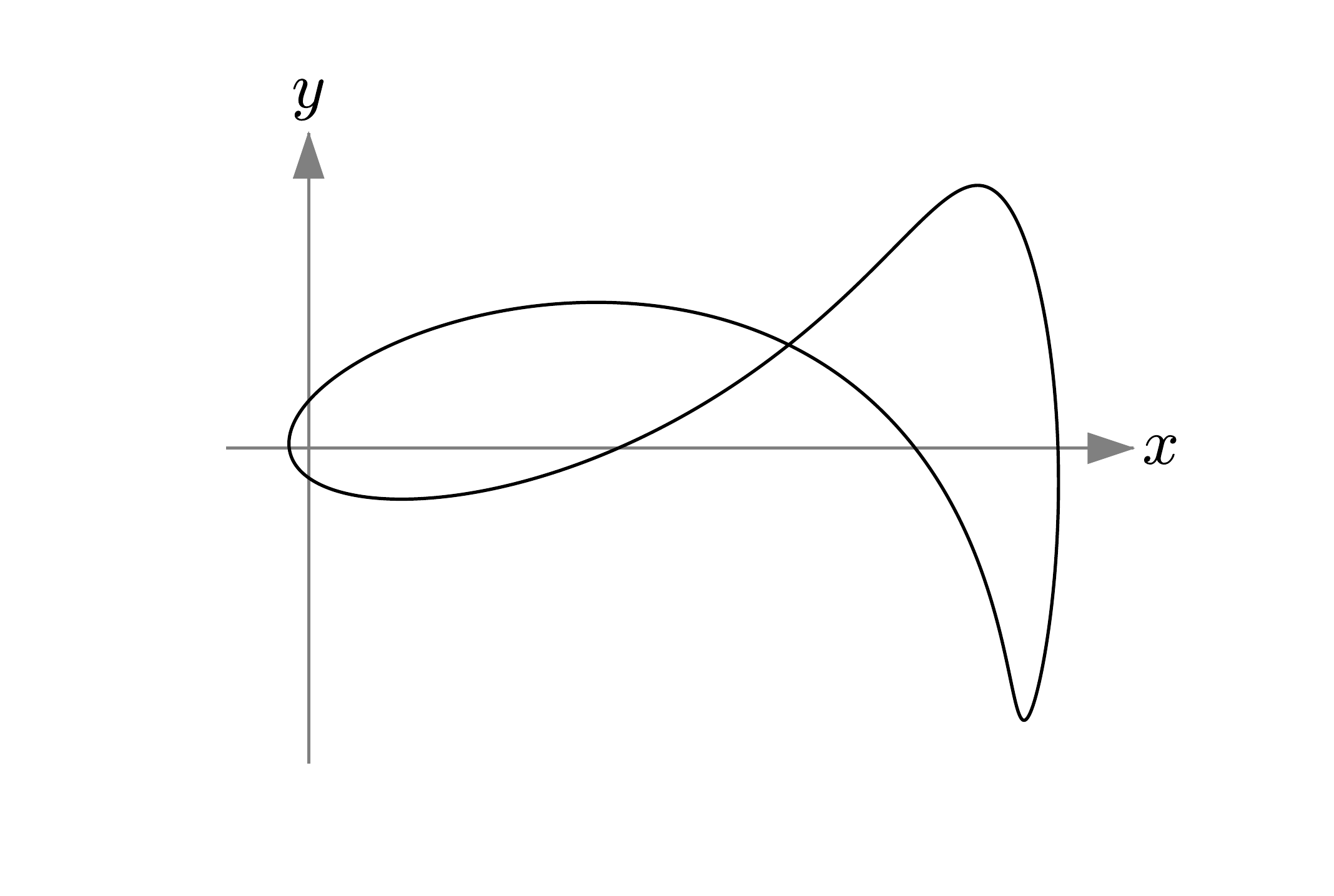}
			\caption{$T_u=3T_v$}
			\label{fig:pSO2}
		\end{subfigure}
		\begin{subfigure}{0.45\textwidth}
			\includegraphics[width=1.1\textwidth]{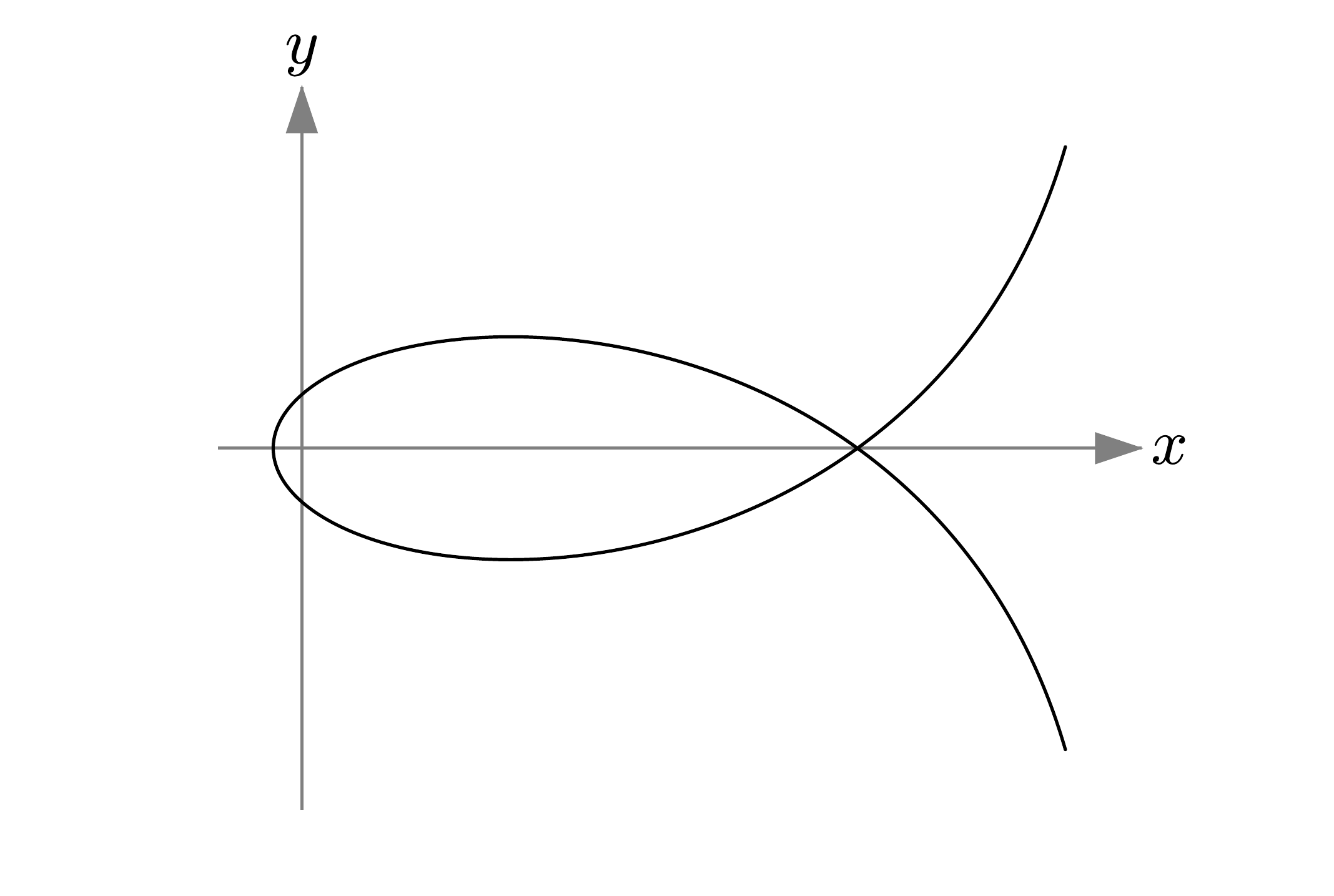}
			\caption{$T_u=2T_v$}
			\label{fig:pSO3}
		\end{subfigure}
		\begin{subfigure}{0.45\textwidth}
			\includegraphics[width=1.1\textwidth]{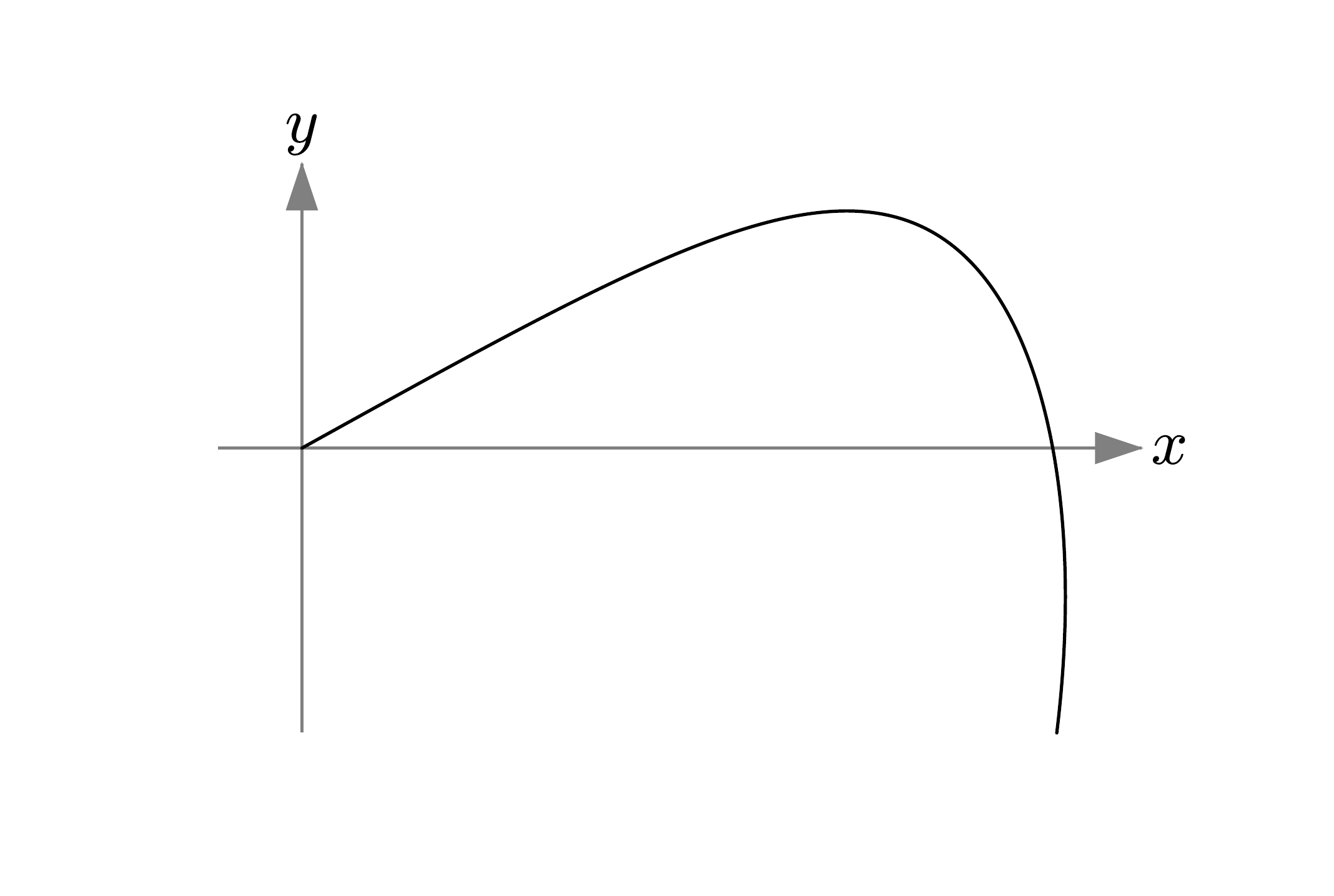}
			\caption{$T_u=3T_v$}
			\label{fig:pSO4}
		\end{subfigure}
		\caption{Stark's problem: some periodic orbits, with $\llS = 348600$ $\kilo\cubic\meter\per\squaren\second$, $f=9.12\times 10^{-9}$ $\kilo\meter\per\squaren\second$.}
		\label{fig:sPerOrb}
	\end{center}	
\end{figure}

\subsection{Other periodic orbits}
There exists another family of periodic orbits of brake type at the boundary of region $I$, i.e. for $\llS = -\mu$: they pass through the origin and have a constant value of the $u$-coordinate, equal to zero, in the $(u,v)$ plane, therefore they lie on the $y$ axis in the $(x,y)$ plane. Moreover, there are periodic orbits of brake type in correspondence of $\llS=\mu$ and $\enS<-2\sqrt{f\mu}$. They also pass through the origin and lie on the $y$ axis in the $(x,y)$ plane, but in the $(u,v)$ plane they have a constant value of the $v$-coordinate, equal to zero. If $\enS=-2\sqrt{f\mu}$, we obtain the two fixed points $(u,v)=(\pm \sqrt{|\enS|/f},0)$ for Stark's dynamics, corresponding to a single fixed point in the $(x,y)$ plane.\\ 
Additional periodic orbits can exist for $(\llS,\enS/\sqrt{f})$ belonging to region $IV$, as a consequence
of Remark~\ref{remark_rIV}. Their peculiarity is that the periods
$T_u$ and $T_v$ are commensurable, that is their quotient $T_v/T_u$ is
rational. In Figure~\ref{fig:sPerOrb} some examples are shown. Note
that the orbits in Figures~\ref{fig:pSO3} and \ref{fig:pSO4} are of
brake type. In this case, if $T_u$ is an odd integer multiple of
$T_v$, the trajectory passes through the origin as shown in
Figure~\ref{fig:pSO4}.

\begin{figure}[h!]
	\begin{center}
		\includegraphics[width=\textwidth]{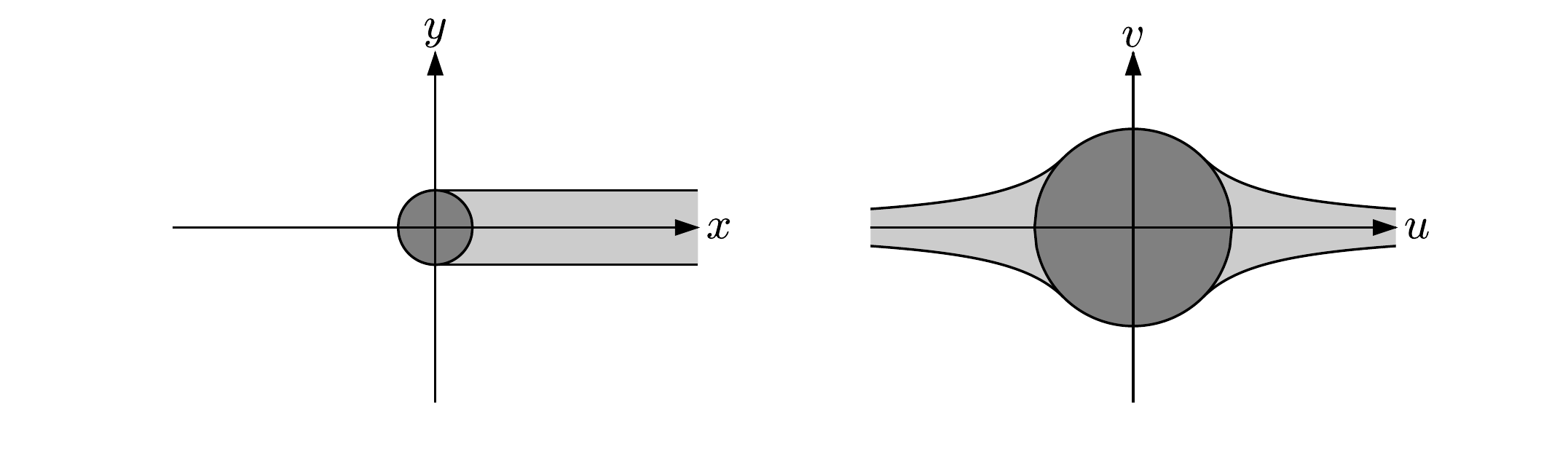}
		\caption{Sun-shadow regions in the $(x,y)$ plane
			(left), $(u,v)$ plane (right).}
		\label{fig:shadowUV}
	\end{center}
\end{figure}

\section{The Sun-shadow dynamics}
\label{s:sunshadow}

The Sun-shadow problem arises by switching dynamics each time the
satellite passes through the boundary of the Earth shadow. The flow
develops by alternating Kepler's regime, corresponding to the shadow
region, and Stark's regime, in the out-of-shadow region. As shown in
Figure~\ref{fig:shadowUV}, the shadow region is defined as the set
$\{(x,y): \ x\geq0,\ -R\leq y \leq R\}$ in the $(x,y)$ plane. In the
$(u,v)$ plane this region has two components: it is the set $\{(u,v): \ -R/u\leq v \leq
R/u,\ |u|\geq\sqrt{R}\}$.

At the time $t=0$ of entrance into the shadow, let us consider initial conditions belonging to the set
\begin{equation}
\bigl\{ (p_u,p_v,u,v): \ u>\sqrt{R},\ uv=-R,\ p_vu + p_uv  >0 \bigr\}.
\label{inicond}
\end{equation}
The last condition is needed to have the velocity vector pointing
inside the shadow region.
We search for
the solutions $(p_u,p_v,u,v)$ of the polynomial system
\begin{equation}
  \left\{
  \begin{split}
    &p_u^2 = 2\enK u^2 + 2(\mu + \llK)\cr
    &p_v^2 = 2\enK v^2 + 2 (\mu -\llK)\cr
    &2\am=p_vu - p_uv \cr
    &uv=R\cr
  \end{split}
  \right.
    \label{keplerpoly}
\end{equation}
with
\[
u>\sqrt{R}, \qquad v> 0,
\]
where $\am$ is the value of the angular momentum
$\mathscr{C}_k$. 
From system \eqref{keplerpoly}, it is possible to obtain the state at
the exit point of the shadow region. By eliminating the variables
$p_u,p_v,v$ we obtain an eight degree polynomial equation in $u$:
\begin{equation}
  \begin{split}
  &(\mu-\llK)^2 u^8 - 4(\mu-\llK)c_k^2 u^6 + 2\bigl(R^2(\llK^2 - \mu^2 - 4c_k^2\enK) + 2c_k^4\bigr)u^4 \cr
  &-4(\mu+\llK)R^2c_k^2 u^2+4(\mu+\llK)^2R^4=0.\cr
\end{split}
  \label{poly8}
\end{equation} 
The roots of \eqref{poly8} come in pairs $\pm u$, which give the same
values of $x$. 
We can select the right value of $x$ using the $y$-component of the Laplace-Lenz integral $\bm{A}_k$. Let us call $\Ubf_i= (p_{u_i},p_{v_i},u_i,v_i)^T$ the selected solution, corresponding to the state at the entrance point in Stark's regime. 
To find the exit point from this regime $\Ubf_o = (p_{u_o},p_{v_o},u_o,v_o)^T$,
we match the time intervals of $u$ and $v$ to go from $\Ubf_i$ to $\Ubf_o$, that can be computed from equations \eqref{pupvtau}.

In this regime, the angular momentum can change not only in value, but
even in sign. Indeed, the satellite can re-enter Kepler's regime
either in the first or third quadrant of the $(u,v)$ plane.

\begin{proposition}
	\label{HsLsprop}
	Each time the satellite crosses the boundary of the shadow
        region, we have a leap in energy from $\enS$ to $\enK$, or
        vice versa: the variation is equal to $\pm
        f(\bar{u}^2-R^2/\bar{u}^2)/2$, where $\bar{u}$ is the value
        taken by $u$ at the crossing point.  A similar leap occurs from
        $\llS$ to $\llK$, or vice versa. In this case, the variation
        is equal to $\pm fR^2/2$. When the satellite goes back to
        Stark's regime, the value of $\mathscr{L}_s$ is the same as
        before crossing the shadow; on the other hand, the energy
        usually changes unless the orbit is symmetric with respect to the
        $u$ axis.
\end{proposition}

\begin{proof}
  Assume that the body enters Stark's regime at the
  point $(\puiP, \pviP, \uiP, \viP)$. Let $\hsiP$ be
  the value of the energy and $\lsiP$ the value of the Laplace-Lenz
  integral. 
  When the body returns to the shadow region, 
  the integrals vary in the following way: 
  \begin{equation*}
  	\hkoP = \hsiP + \frac{f}{2}\biggl(\uoP^2-\frac{R^2}{\uoP^2}\biggr),
  \end{equation*}
  \begin{equation*}
    \lkoP = \lsiP + \frac{f}{2}R^2,
  \end{equation*}
  with $\uoP$ the $u$ coordinate of the point on the shadow boundary
  where the satellite exits from Stark's regime. When the body enters
  Stark's regime again, by passing through the point
  $(\puiS,\puiS,\uiS,\viS)$, similar variations occur:
  \begin{equation*}
    \hsiS = \hkoP - \frac{f}{2}\biggl(\uiS^2-\frac{R^2}{\uiS^2}\biggr),
  \end{equation*}
  \begin{equation*}
    \lsiS = \lkoP - \frac{f}{2}R^2.
  \end{equation*}
  Thus, we get
  \begin{equation*}
    \hsiS = \hsiP + \frac{f}{2}\biggl(1+\frac{R^2}{\uoP^2 \uiS^2}\biggr)\bigl(\uoP^2-\uiS^2\bigr),
  \end{equation*}
  \begin{equation*}
    \lsiS = \lsiP,
  \end{equation*}
  where $	\hsiS = \hsiP$ only if $\uoP=\uiS$.	
\end{proof}

\subsection{Periodic orbits of brake type}
\label{proofPBOexists}
We prove the existence of a family of periodic orbits of brake type,
$\widehat{\bx}=\widehat{\bx}(t;\llS)$, parametrised by $\llS$, which
are close to the brake periodic orbits $\bx^*=\bx^*(t;\llS)$ of
Stark's problem, described in Section~\ref{s:brakestark}. For this
purpose, we consider values $\llS$ in $(-\mu,
\mu)$, for which the periodic orbits $\bx^*$ exist. In the following
we shall restrict the interval $(-\mu,\mu)$ for reasons related to the
proof.

The idea of the proof is to search for an initial point $(x,y) =
(x_0,0)$ with $x_0>R$, i.e. in Kepler's regime, allowing to arrive at
a zero velocity point $(x_B,y_B)$ in Stark's regime after passing
through an exit point $(x_E,R)$ from the shadow region. We look for an
orbit that is symmetric with respect to the $x$ axis, like
$\bx^*$. Because of the symmetry, $\widehat{\bx}$ oscillates between
the points $(x_B,y_B)$ and $(x_B,-y_B)$.
We search for an initial value $x_0$ fulfilling
\begin{equation}
x_0 > \xistar/2,
\label{x0rel}
\end{equation}
where $(\xistar/2,0)$ belongs to $\bx^*$, see
Figure~\ref{fig:proofStrategyXY}.  The idea behind this choice is
that, passing through the shadow, the pushing effect of the solar
radiation pressure is lacking. Moreover, we require that $x_0$ is such
that at the exit point $(x_E,R)$ the energy fulfils
\begin{equation}
\enS< 	\ensstar,
\label{xErel}
\end{equation}
with $\ensstar$ given in \eqref{hs_star}. If $\enS > \enS^*$ we
cannot have zero velocity points, see Remark~\ref{r:zerovel}.

For the proof, we use the variables $u,v$. The initial point $(x_0,0)$
corresponds to two possible points $(\pm \sqrt{2x_0},0)$ in the
$(u,v)$ plane. Similarly, the point $(\xistar/2,0)$ corresponds to
$(\pm \sqrt{\xistar},0)$. By symmetry, we can focus only on the
$\{u>0\}$ half-plane of the $(u,v)$ plane, see
Figure~\ref{fig:proofStrategyUV}.

\begin{figure} [h!]
  \begin{subfigure}{0.45\textwidth}
    \centering
    \includegraphics[width=1.1\textwidth]{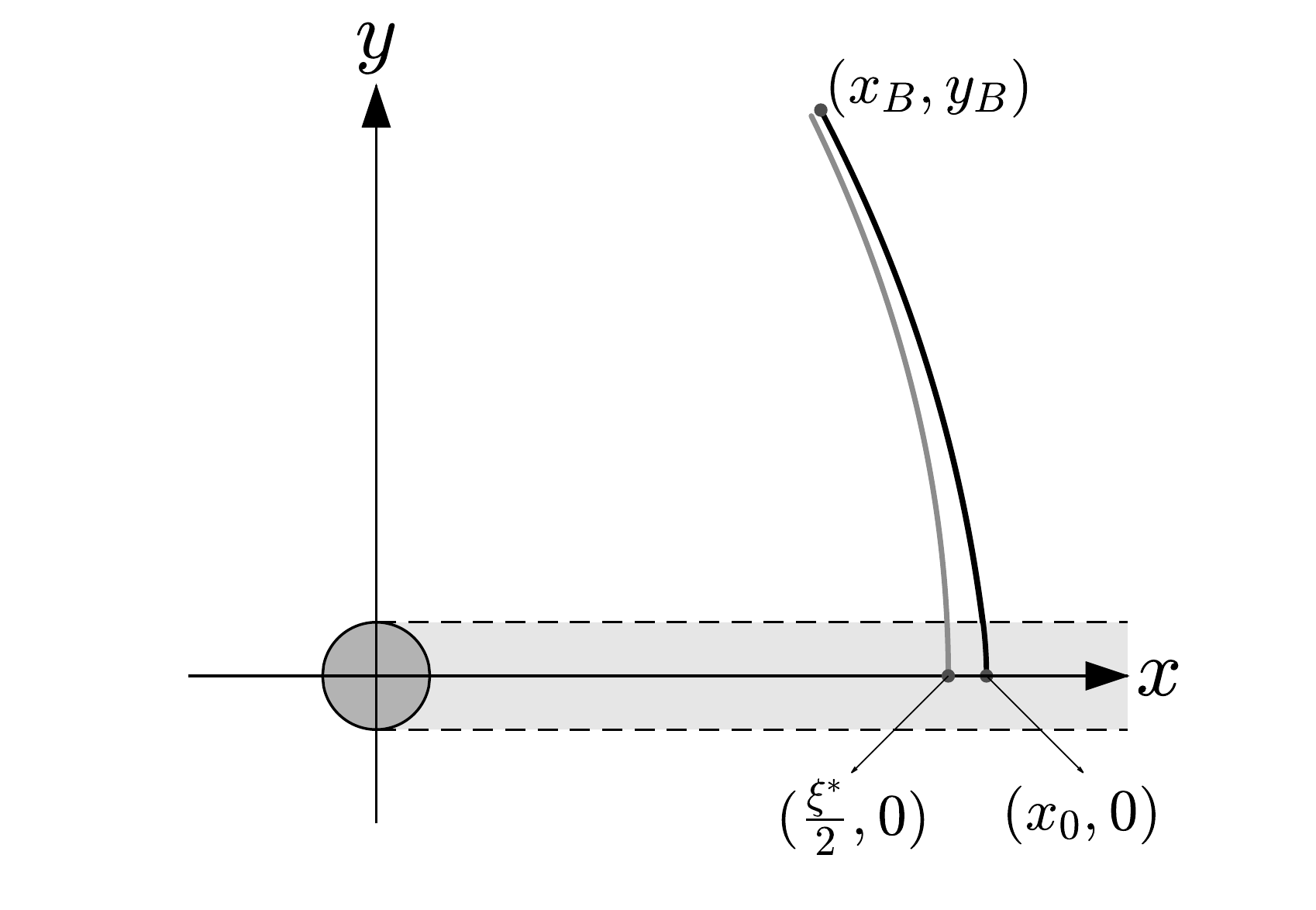}
    \subcaption{$(x,y)$ plane}
    \label{fig:proofStrategyXY}
  \end{subfigure}
  \begin{subfigure}{0.45\textwidth}
    \centering
    \includegraphics[width=1.1\textwidth]{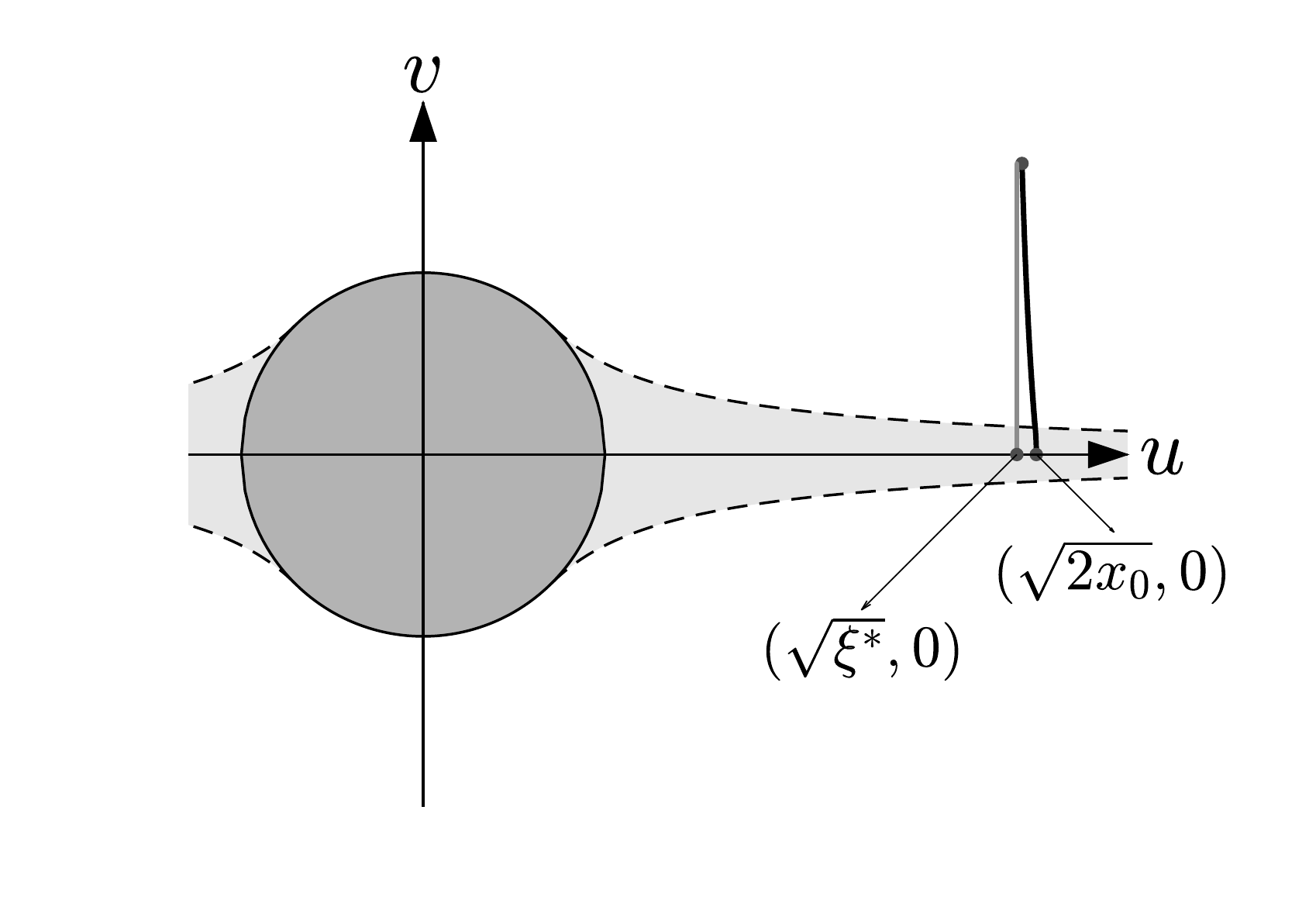}
    \subcaption{$(u,v)$ plane}
    \label{fig:proofStrategyUV}
  \end{subfigure}
  \caption{A portion of the brake orbit ${\bm{\widehat{x}}}(t;\llS)$
    between the horizontal axis and a zero velocity point is
    represented by the black curve in the $(x,y)$ and $(u,v)$
    coordinates. The analogous portion of the brake orbit
    ${\bm{{x}}}^*(t;\llS)$ of Stark's problem is drawn in grey.}
\end{figure}

\begin{proposition}
If $x_0$ is selected so that relations \eqref{x0rel}, \eqref{xErel}
hold, then the exit point $(x_E,R)$ from the shadow region, with 
\begin{equation}
	x_E
	= \frac{\xiE}{2}-\frac{R^2}{2\xiE},
	\label{xE}
\end{equation}
belongs to the unbounded component of the possible configurations
corresponding to region IV of Stark's regime.
\label{p:regionIV}
\end{proposition}

\begin{proof}

Because of the symmetry of the orbit with respect to the
$x$ axis, in Kepler's regime the initial state has the form
\[
(p_u,p_v,u,v)=(0,p_{v_0},\sqrt{2x_0},0).
\]
By system \eqref{keplerpoly} we have
\begin{equation}
  p_{v_0} = \sqrt{2(\mu-\llK)}
  \label{pv0}
\end{equation}
and
\[
\enK = -\frac{\mu+\llK}{2x_0},
\]
where
\begin{equation}
  \llK = \llS + \frac{f}{2}R^2.
  \label{llsllk}
\end{equation}
Note from relations \eqref{pv0}, \eqref{llsllk} that we need to
restrict the interval of $\llS$ to $(-\mu,\mu-\frac{f}{2}R^2)$.

The exit point from the shadow is obtained by solving
\begin{equation}
\left\{
\begin{split}
	&p_u^2 = 2\enK u^2 + 2(\mu + \llK)\cr
	&p_v^2 = 2\enK v^2 + 2(\mu - \llK)\cr
	&uv=R\cr
	&p_u p_v=2uv\enK, \cr
\end{split}
\right.
\label{exitpoint}
\end{equation}
where $p_v>0$.

\noindent With the last equation we set to zero the $y$-component of
the Laplace-Lenz vector, which is a necessary condition for the
symmetry with respect to the $x$ axis. There are four possible real
solutions of \eqref{exitpoint}, and two of them have positive values
of $u$. Between these two, only one corresponds to exiting from the
shadow, i.e. fulfils
\[
p_y = \frac{up_v+vp_u}{u^2+v^2} > 0.
\]
This solution gives the $u$ coordinate at the exit point, whose square is
\begin{equation}
	\xiE = x_0+x_T,
	\label{xiEDef}
\end{equation}
with
\begin{equation}
  x_T = x_T(x_0) = \sqrt{x_0^2 -a_k R^2}, \qquad a_k = \frac{\mu+\llK}{\mu-\llK}.
\label{xT}
\end{equation}
From $\xiE$, we obtain $x_E$ through relation \eqref{xE}, applying the coordinates change \eqref{uv}.
Using relation \eqref{x0rel}, we can write
\[
x_0 = \xistar/2 + \Delta x .
\]
We have real values of $\xiE$ if and only if
\[
\Delta x \in \Bigl(-\infty, - \frac{\xistar}{2} - \sqrt{a_k}R \Bigr]
                \cup \Bigl[- \frac{\xistar}{2} + \sqrt{a_k}R,
                  +\infty\Bigr).
\]
The last condition is fulfilled for each choice of $\Delta x\geq 0$ if
\begin{equation}
-\frac{\xistar}{2}+\sqrt{a_k}R\leq 0,
\label{xistarCond}
\end{equation}
which holds if we further restrict the interval of $\llS$ to
$[{\llS}^-,{\llS}^+]$, with
\[
  {\llS}^\pm = -\frac{5}{4}fR^2 \pm \sqrt{\mu^2+\frac{9}{16}f^2R^4-\frac{5}{2}fR^2\mu}.
\]

\begin{remark}
  Since $f\ll 1$, the new range $[{\llS}^-,{\llS}^+]$ is slightly
  smaller than $(-\mu,\mu)$. 
\end{remark}

The energy in Stark's regime is
\begin{equation}
\hsfun(x_0) = -\frac{\mu+\llK}{2x_0} -\frac{f}{2}\left(x_0+x_T\right) + \frac{fR^2}{2(x_0+x_T)}.
\label{hsx0}
\end{equation}
For the proof we need this result:

\begin{lemma}
  The energy $\hsfun$
	is a decreasing function of $x_0$ in the interval $[\frac{\xistar}{2},\infty)$. Moreover, we have
	
	\[
	\hsfun(\xistar/2)>\ensstar.
	\]
	\label{lemmaEnergy}
\end{lemma}
\begin{proof}
	
  From equation \eqref{hsx0}, the derivative of $\hsfun$ with respect to $x_0$ is	
  \[
  \frac{d\hsfun}{dx_0}= \frac{(\mu+\llK)(x_0+x_T)x_T-f(x_0+x_T)^2x_0^2-fR^2x^2_0}{2(x_0+x_T)x^2_0x_T}.
  \]
  The denominator is always positive, being $x_0,x_T>0$. We prove that
  the numerator is negative. Because of relations \eqref{hs_star}, \eqref{xistar2} and
  \eqref{llsllk}, this corresponds to showing that
  \[
  (x_0+x_T)\biggl(\frac{{\xistar}^2}{2}x_T - (x_0+x_T)x_0^2\biggr)
  +R^2\biggl(\frac{1}{2}(x_0+x_T)x_T - x_0^2\biggr) < 0.
  \]	
  This follows from \eqref{x0rel} and  $x_0>x_T$.
  We conclude that $\frac{d\hsfun}{dx_0}<0$. 

  \smallbreak
  Next we prove the second statement of the lemma. We have
  \[
  \hsfun\left(\frac{\xistar}{2}\right)-\ensstar =
  - \frac{\mu+\llK}{\xistar} - \frac{f}{4}\Bigl(\xistar
  + \sqrt{\xistar^2 - 4 a_k R^2}\,\Bigr)
  + \frac{fR^2}{\xistar + \sqrt{{\xistar}^2 - 4 a_k R^2}}
  - \ensstar.
  \]
  Using \eqref{hs_star}, \eqref{xistar2}, \eqref{llsllk} we obtain
  \[
-\frac{\mu+\ell_k}{\xi^*} = \frac{h_s^*}{2} + \frac{f^2R^2}{2 h_s^*}
  \]
  and we get
  \[
  \hsfun\left(\frac{\xistar}{2}\right)-\ensstar =
  \Bigl(\ensstar + \sqrt{{\ensstar}^2-4a_k f^2 R^2}\,\Bigr)
  \Biggl(
  -\frac{1}{4}  +
  \frac{f^2R^2} {2\ensstar
    \Bigl(-\ensstar + \sqrt{{\ensstar}^2-4
      a_k f^2R^2}\,\Bigr)}\Biggr). 
  \]
  From relation
  \[
  \ensstar + \sqrt{{\ensstar}^2-4 a_k f^2R^2} <0,
  \]
  we conclude that $\hsfun(\xistar/2)-\ensstar>0$.
	
\end{proof}


Using Lemma~\ref{lemmaEnergy}, we only need to find $x_0^* >
\xistar/2$ such that $\hsfun(x_0^*)=\ensstar$ to prove that at the
exit time the values of the integrals $(\llS,\hsfun(x_0))$, with
$x_0>\xistar/2$, belong to region IV. From \eqref{hs_star}, \eqref{xistar2},
\eqref{llsllk}, \eqref{xT} and \eqref{hsx0}, solving equation $\hsfun(x_0)=\ensstar$ corresponds to
searching for the roots of
\[
g(x_0) = -4x_0^3+ 4\xistar x_0^2 +
\bigl(R^2-\xistar^2+2 a_k R^2\bigr)x_0-
\bigl((2x_0-\xistar)^2+R^2\bigr)x_T.
\]
From relations \eqref{x0rel} and \eqref{xT}, it holds
\begin{equation}
Cx_0<x_T<x_0, \qquad C = \sqrt{1 - 4 a_k\frac{R^2}{\xistar^2}}.
\end{equation}
Thus, we have 
\[
g_-(x_0) \leq g(x_0) \leq g_+(x_0),
\]
where
\[
\begin{split}
  &g_-(x_0)=-2x_0\bigl( 4x_0^2-4\xistar x_0 + \xistar^2 - a_k R^2
  \bigr),\cr &g_+(x_0)=-(1+C)x_0\left( 4x_0^2-4\xistar x_0 +
  \xistar^2 - \frac{1-C}{1+C}R^2-\frac{2}{1+C} a_k R^2\right).\cr
\end{split}
\]
The polynomials $g_+$ and $g_-$ have three roots, but only one is
larger than $\xistar/2$. Denoting the latter with $x_0^+$ and $x_0^-$
respectively, we get
\[
x_0^- = \frac{\xistar}{2} + C_1,\qquad
x_0^+ = \frac{\xistar + C_2}{2}, 
\]
with
\[
  C_1 = \frac{R}{2}\sqrt{ a_k},\qquad C_2 =
  R\,\sqrt{\frac{1-C+2 a_k}{1+C}}.
\]
This shows the existence of $x_0^*>\xi^*/2$, solution of
$\hsfun(x_0)=\ensstar$.

\smallbreak
Next we show that the exit point belongs to the unbounded component of
the configuration set, i.e. that $\xi_E > \xi_1$ holds for each $x_0>x_0^*$.  Since $x_0^-<x_0^*<x_0^+$ and $\xiE$, given in
\eqref{xiEDef}, is an increasing function of $x_0$, we have 
\begin{equation}
	\xiE(x_0^-)<\xiE(x_0^*)<\xiE(x_0^+)
	\label{xiEestimate}
\end{equation}
where
\[
\begin{split}
&\xiE(x_0^-) = \frac{\xistar}{2}+C_1+\sqrt{\left(\frac{\xistar}{2}+C_1\right)^2- a_k R^2},\cr
&\xiE(x_0^+) = \frac{\xistar+C_2}{2}+\sqrt{\frac{\left(\xistar+C_2\right)^2}{4}- a_k R^2}.\cr
\end{split}
\]
Relation \eqref{xistarCond} implies $\xiE(x_0^-)>\xistar$. Thus
$\xiE(x_0^*)>\xistar$.  Since in region IV we have $\xi_2< \xistar
< \xi_1$, and $\xiE< \xi_2$ or $\xiE> \xi_1$, the latter
relation holds.

\end{proof}

\begin{remark}
Proposition~\ref{p:regionIV} yields
\[
2x_0>\xiE > \xi_1 > \xistar > \xi_2 > 0, 
\]
and
	\[
\xistar + C_1 <
\xiEstar < \xistar + C_2,
\]
where $\xiEstar = \xiE(x_0^*)$.
\label{remarkp:regionIV}
\end{remark}

To search for a zero velocity point $(x_B,y_B)$ we use the coordinates $u,v$ and the fictitious time $\tau$. The maps
$u(\tau)$, $v(\tau)$ become stationary at $\tau=\tau_u, \tau_v$
respectively, where
\begin{align*}
\tau_u &= \int_{u_1}^{\sqrt{\xiE}} \frac{du}{\sqrt{fu^4 + 2\enS u^2  + 2(\mu+\llS)}},
\\
\tau_v &= \int_{\frac{R}{\sqrt{\xiE}}}^{v_1} \frac{dv}{\sqrt{- fv^4 + 2\enS v^2 + 2(\mu-\llS)}}.
\end{align*}
We search for value of $x_0$ such that
\begin{equation}
\tau_u = \tau_v,
\label{sametau}
\end{equation}
which corresponds to reach a zero velocity point.

\begin{figure}
  \centering
  \begin{subfigure}{0.32\textwidth}
    \includegraphics[width=1.04\textwidth]{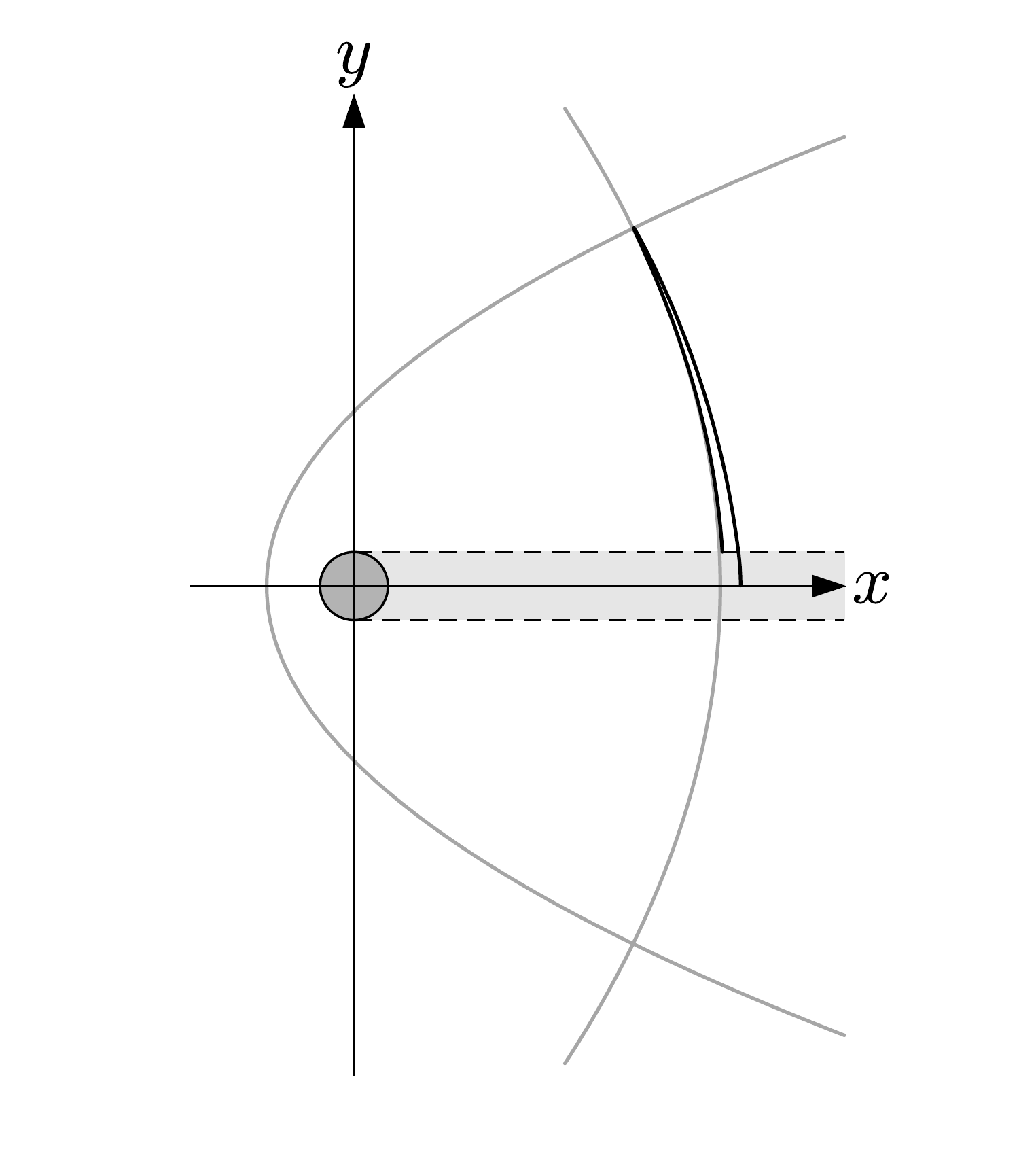}
    \subcaption{$\tau_u<\tau_v$}
  \end{subfigure}
  \begin{subfigure}{0.32\textwidth}
    \includegraphics[width=1.04\textwidth]{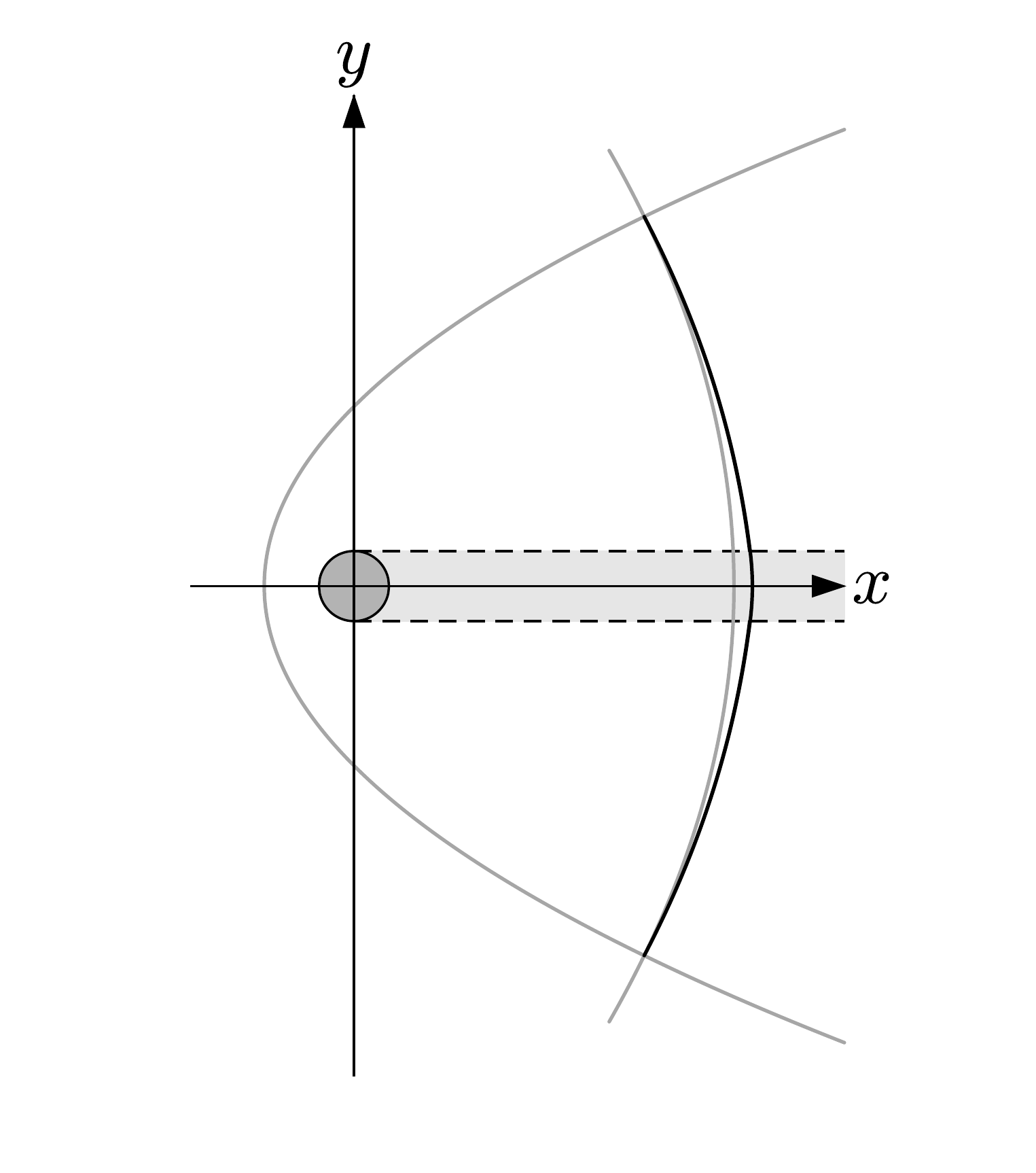}
    \subcaption{$\tau_u=\tau_v$}
  \end{subfigure}
  \begin{subfigure}{0.32\textwidth}
    \includegraphics[width=1.04\textwidth]{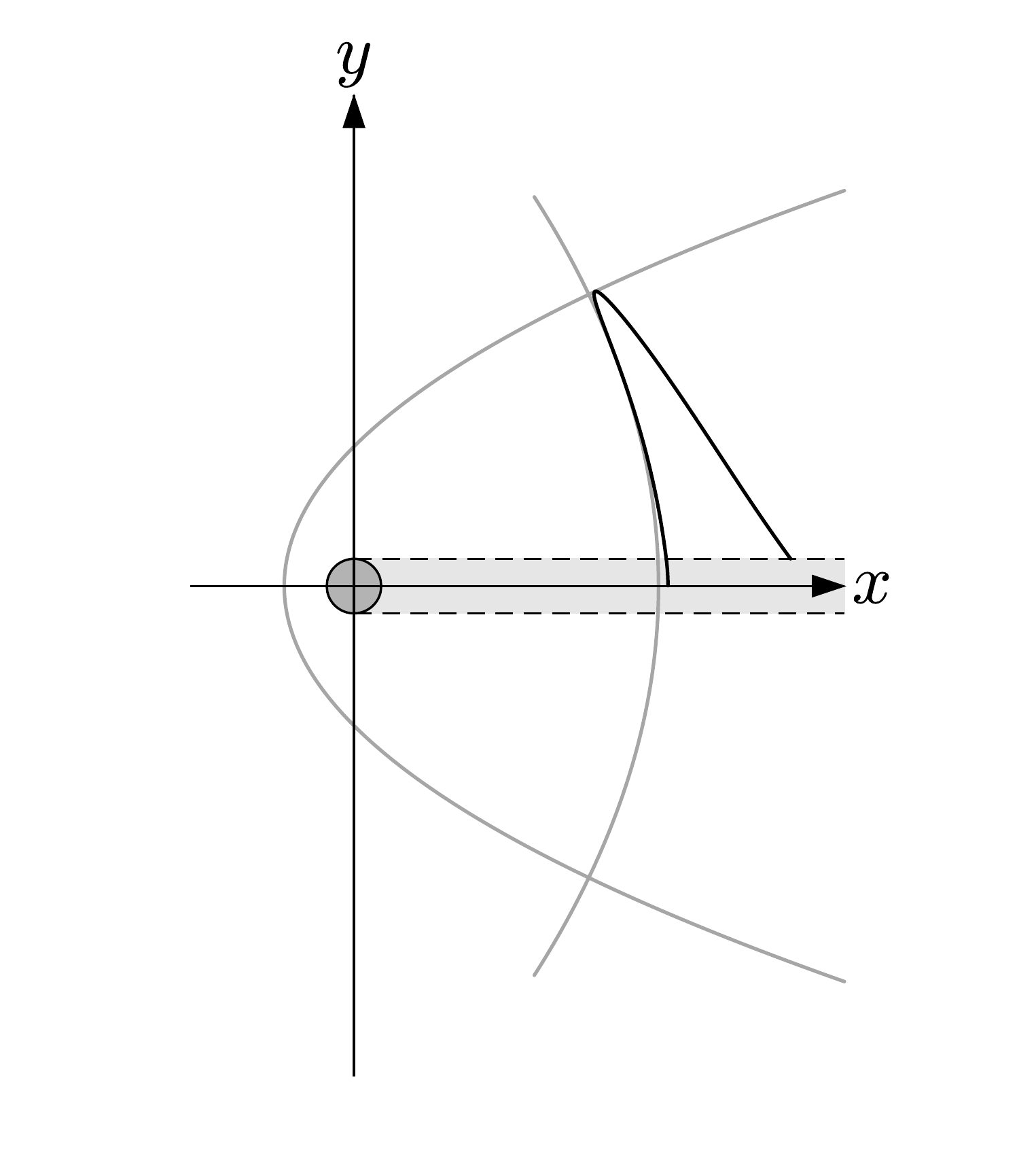}
    \subcaption{$\tau_u>\tau_v$}
  \end{subfigure}
  \caption{Brake orbit (a) and two close trajectories (b) and (c). The grey arcs of parabola represent the boundaries defining the unbounded admissible subset of the configuration space in Stark's regime (see Figure~\ref{fig:admisreg} for a comparison).}
\end{figure}

From now on, we use $\enS\in J =(-\infty, \ensstar)$ as independent
variable, in place of $x_0$. Following \cite{Beletsky_1964} we can write the integrals $\tau_u, \tau_v$ as
\begin{align}
\tau_u(\enS) & = \frac{1}{\sqrt{f \xi_1}}\int_0^{\arcsin \rxi}
\frac{d\varphi}{\sqrt{1-\frac{\xi_2}{\xi_1}\sin^2\varphi}},
	\label{tauU}\\[3pt]
\tau_v(\enS) & = \frac{1}{\sqrt{f\Delta \eta}}\int_0^{\arcsin\sqrt{1-\frac{R^2}{\xiE\eta_1}}}\frac{d\varphi}{\sqrt{1-\frac{\eta_1}{\Delta \eta}\sin^2\varphi}},
\label{tauV}
\end{align}
where
\[
\rxi = \sqrt{\frac{\xiE - \xi_1}{\xiE - \xi_2}},\qquad \Delta \eta = \eta_1-\eta_2.
\]

We use the result proved below.
\begin{lemma}
  The following properties hold:
  \begin{itemize}
  \item[i)] $\tau_u$ is a strictly increasing function of $\enS$ and
    \begin{equation}
      \lim_{\enS\to \ensstar} \tau_u = +\infty,
      \label{limtauu}
    \end{equation}
  \item[ii)] $\tau_v$ is not monotone in $\enS$ and
    \begin{equation}
      \limsup_{\enS\to\ensstar} \tau_v < +\infty.
      \label{limtauv}
    \end{equation}
  \end{itemize}
\label{lemmaTauTav}
\end{lemma}

\begin{proof}
i) The derivative of  $\tau_u$ with respect to $\enS$ can be written as 
\[ \frac{d \tau_u}{d \enS} =  \frac{1}{\sqrt{f \xi_1}} \frac{1}{\Delta \xi} \left(
\frac{1}{f}\int_0^{\arcsin\rxi}
\frac{1+\frac{\xi_2}{\xi_1}\sin^2\varphi}{\left(1-\frac{\xi_2}{\xi_1}\sin^2\varphi\right)^{\frac{3}{2}}} d\varphi +  
\frac{1}{2} \sqrt{\frac{\xi_1}{\xiE}} \frac{\xiE-\xi_2}{\rxi} \ \frac{d \rxi^2}{d \enS} \right),
\]
where
\[
\Delta \xi = \xi_1-\xi_2.
\]
Therefore, to prove that $\tau_u$ is strictly increasing,
it is sufficient to show that 
\begin{equation}
\frac{d \rxi^2}{d \enS} = \frac{d}{d \enS} \left(\frac{\xiE-\xi_1}{\xiE-\xi_2}\right) > 0.
\label{argincresing}
\end{equation}
From relations
\[
\begin{split}
&\frac{d \xi_1}{d \enS}= -\frac{2\xi_1}{f\Delta \xi} <0, \qquad \frac{d \xi_2}{d \enS}= \frac{2\xi_2}{f\Delta \xi}>0,\cr
&\frac{d \xiE}{d \enS} = \frac{\xiE}{\frac{(\mu+\llK)x_T}{2x_0^2} -\frac{f}{2} \xiE - \frac{f}{2} \frac{R^2}{\xiE}} <0,\cr
\end{split}
\]
where the latter follows from \eqref{csi_star}, \eqref{llsllk}, \eqref{xiEDef} and $x_0>x_T$, we see that
\[
\frac{d (\xiE-\xi_2)}{d \enS} < 0.
\]
Moreover, we have
\begin{equation}
  \frac{d (\xiE-\xi_1)}{d \enS} > 0,
  \label{dxiExi1}
\end{equation}
in fact, relations \eqref{csi_star}, \eqref{llsllk} and \eqref{xiEDef} yield that \eqref{dxiExi1} is equivalent to
\[
R^2\xi_1 \left(\frac{1}{x_0+x_T}-\frac{x_T}{2x_0^2}\right) +
\xiE\xi_2-\xi_1\frac{{\xistar}^2}{2x_0^2}x_T> 0.
\]
which follows from $x_0>x_T$, $\xi_2=\xistar^2/\xi_1$ (see \eqref{xi12eta12} and \eqref{csi_star}), and
Remark~\ref{remarkp:regionIV}. This proves \eqref{argincresing}.

\noindent Furthermore, we have 
\[
\lim_{\enS\to \ensstar} \xi_1 = \lim_{\enS\to \ensstar} \xi_2 = \xistar
\]
and \eqref{xistar2} holds, so that 
\[
\lim_{\enS\to \ensstar} \tau_u = \frac{1}{\sqrt{f \xistar}} \int_0^{\frac{\pi}{2}}
\frac{d\varphi}{\sqrt{1-\sin^2\varphi}}  = +\infty.
\] 
This concludes the proof of i).

ii) We have
\[
\tau_v < \frac{1}{\sqrt{f\Delta \eta}}\int_0^{\pi/2}\frac{d\varphi}{\sqrt{1-\frac{\eta_1}{\Delta \eta}\sin^2\varphi}} = \frac{1}{\sqrt{f\Delta \eta}}\frac{\pi}{2M\Bigl(1,\sqrt{1-\frac{\eta_1}{\Delta \eta}}\Bigr)},
\]
so that
\[
\limsup_{\enS\to \ensstar}\tau_v \leq \frac{\pi}{2\sqrt[4]{\mu f}}
\frac{1}{2\sqrt{1-\frac{-\xistar + 2\sqrt{\mu/f}}{4\sqrt{\mu/f}}}}
<+\infty.
\]

\end{proof}

Using the previous result, to show that a solution of \eqref{sametau}
exists, we only need to find a value of the energy $\enS\in J$ such that
\[
\tau_u(\enS) < \tau_v(\enS).
\]

\begin{lemma}
  There exists $\enBar\in J$ and two functions $\tilde{\tau}_u(h_s), \tilde{\tau}_v(h_s)$ such that
  \[
  \tau_v(\enBar) > \tilde{\tau}_v(\enBar) =
  \tilde{\tau}_u(\enBar) > \tau_u(\enBar).
  \]
  \label{lemmaEnergyBar}
\end{lemma}

\begin{proof}
	
Using relation
\[
\sqrt{1-\frac{\eta_1}{\Delta \eta}\sin^2\varphi} < 1
\]
we get
\[
\tau_v > \frac{1}{\sqrt{f\Delta \eta}}
\arcsin\sqrt{1-\frac{R^2}{\xiE\eta_1}}
\]
in the interval $J$, where the function $\xiE
	\eta_1$ is decreasing. Indeed, we have
	\[
	\frac{d(\xiE\eta_1)}{d \enS} = \frac{\xiE\eta_1}{f}\Biggl(\frac{1}{\sqrt{\frac{\enS^2}{f^2}+\frac{2(\mu-\llS)}{f}}} - \frac{1}{\frac{\xiE}{2} +  \frac{R^2}{2\xiE} - \frac{(\mu+\llK)x_T}{2fx_0^2}}\Biggr) <0
	\]
which follows from \eqref{hsx0}, $x_0>x_T$ and $\llS \in [\llS^-,\llS^+]$. 	
Thus 
  \[
  \xiE \eta_1>\xiE^*\eta_1^*, \qquad
  \eta_1^* = \eta_1(\ensstar) = -\xistar + 2\sqrt{\mu/f}.
  \]
  Since $\xiE^* > \xistar$ (see Remark~\ref{remarkp:regionIV}),
  we obtain
  \[
 \frac{1}{\sqrt{f\Delta \eta}} \arcsin\sqrt{1-\frac{R^2}{\xiE\eta_1}} >\frac{1}{\sqrt{f\Delta \eta}} \arcsin\sqrt{1-\frac{R^2}{\xistar\eta_1^*}}.
  \]
 Moreover, we have
  \[
  \left(\frac{\Delta \eta}{2}\right)^{\frac{1}{2}} =  \left(\frac{\enS^2}{f^2}+2\frac{\mu-\llS}{f}\right)^{1/4}<\left(\xi_1^2+2\frac{\mu-\llS}{f}\right)^{1/4}<\left(\xi_1+\sqrt{\frac{2(\mu-\llS)}{f}}\right)^{1/2}.
  \]
  Hence, we can set   
  \[
  \tilde{\tau}_v (h_s) = \frac{1}{\sqrt{2f}\left(\xi_1+\sqrt{\frac{2(\mu-\llS)}{f}}\right)^{1/2}}\arcsin\sqrt{1-\frac{R^2}{\xistar\eta_1^*}}.
  \]
  Then, using
  \[
  	\sqrt{1-\frac{\xi_2}{\xi_1}\sin^2{\phi}}>\cos{\phi},  
  \]
  we get
  \[
  \begin{split}
    &\tau_u < \frac{1}{2} \frac{1}{\sqrt{f\xi_1}}
    \log\left(1+ \frac{2\rxi}{1-\rxi}\right)
    < \frac{1}{\sqrt{f\xi_1}}
    \left(
\frac{\rxi}{1-\rxi}
\right)\cr  
    & =
    \frac{1}{\sqrt{f\xi_1}}
    \left( \frac{1}{\sqrt{ 1 + \frac{\Delta \xi}{\xiE-\xi_1}}-1 }\right)<
    \frac{1}{\sqrt{f\xi_1}}
    \left( \frac{1}{\sqrt{ 1 + \frac{\Delta \xi}{C_2}}-1 }\right)\cr      
  \end{split}
  \]
  which follows from $0<r_{\xi}<1$ for $\enS \in J$ and Remark~\ref{remarkp:regionIV}. From relation
  \[
  \left(1+\frac{1}{\xi_1}\sqrt{\frac{2(\mu-\llS)}{f}}\right)^{1/2} < \left(1+\frac{1}{\xistar}\sqrt{\frac{2(\mu-\llS)}{f}}\right)^{1/2}
  \]
  we have
  \[
  \tilde{\tau}_u(h_s) = \frac{1}{\sqrt{f}} \frac{\left(1+\frac{1}{\xistar}\sqrt{\frac{2(\mu-\llS)}{f}}\right)^{1/2}}{\left(\xi_1+\sqrt{\frac{2(\mu-\llS)}{f}}\right)^{1/2}} \frac{1}{\sqrt{ 1 + \frac{\Delta \xi}{C_2}}-1 }.
  \]  
  Set
  \[
  K_1 =
  \sqrt{2} \left(1+\frac{1}{\xistar}\sqrt{\frac{2(\mu-\llS)}{f}}\right)^{1/2},
  \qquad K_2 = \arcsin\sqrt{1-\frac{R^2}{\xistar\eta_1^*}}.
  \]
    
  \noindent We obtain $\tilde{\tau}_u=\tilde{\tau}_v$ in 
  
  \begin{equation}
    \enBar =-f\left(\xistar^2+\frac{C_2^2}{4}\frac{K_1^2}{K_2^2}\left(2+\frac{K_1}{K_2}\right)^2\right)^{1/2}.
    \label{enbar}
  \end{equation}
  
		
\end{proof}
This concludes the proof of the existence of a family of brake periodic orbits
$\widehat{\bx}(t;\llS)$ parametrised by $\llS$.



\section{The Sun-shadow map}
\label{s:ssmap}

To study the Sun-shadow dynamics, it is useful to construct a
Poincar\'e-like map. Traditionally, the Poincar\'e maps of autonomous
two-dimensional Hamiltonian systems are defined by fixing the value of
the Hamiltonian. This is not possible for the Sun-shadow dynamics,
since the Hamiltonian is not conserved along the flow, see Proposition~\ref{HsLsprop}. However, here we can fix the value of
$\mathscr{L}_s$. Indeed, even though $\mathscr{L}_s$ is not a constant
of motion as well, it assumes the same value in Stark's regime before
and after the satellite crosses the shadow.  To define a Poincar\'e map we also need to introduce a section. The selected section
 corresponds to the upper boundary of the shadow region in the
$(x,y)$ configuration space, i.e. to $y=R$, $x\geq0$. We consider
trajectories leaving the section with $p_y>0$, in the outward
direction with respect to the shadow region.  Since the dependence on
the coordinates is decoupled in the variables $u,v$ (see \eqref{kepl_sepvariables}, \eqref{HJStark}), we
decided to use them to define the map. Thus, we define the
Sun-shadow map as
\begin{equation*}
	\begin{split}
	\mathfrak{S}: \domsigma\subset\R^2 &\to \R^2,\cr
	(u,p_u) &\mapsto (u',p'_u),\cr
	\end{split}
\end{equation*}
where the domain $\domsigma$ is discussed below,
and $(u,p_u),(u',p'_u)$ belong to the section $\Sigma$
defined as
\[
\Sigma = \{(p_u,p_v,u,v): \ |u|\geq \sqrt{R},\ uv = R, \ up_v >
\max(0,-p_uv), \ \mathscr{L}_s = \llS\}.
\]
The conditions $|u|\geq \sqrt{R}$, $uv = R$ are necessary to select the
desired section in the $(x,y)$ configuration space. The
condition $up_v>-p_uv$ is equivalent to $p_y>0$ (see
\eqref{pupv}). The additional condition, $up_v>0$, assures that every
point $(u,p_u)\in\Sigma$ corresponds to only one trajectory. Indeed,
there are points $(u,p_u)$ for which $p_y>0$ in both the cases $p_v>0$
and $p_v<0$.

\begin{proposition}
  The map $\mathfrak{S}$ is not defined in the points $(u,p_u)$ with
  \begin{equation}
    p^2_u \leq 2(\mu+\llS) + fR^2 + (f-2(\mu-\llS)/R^2)u^4.
    \label{defDomainMap}
  \end{equation}
  Moreover, in the second and fourth quadrant of the $(u,p_u)$ plane
  $\mathfrak{S}$ is not defined if $\llS > \mu - \frac{fR^2}{2}$, while if $\llS < \mu - \frac{fR^2}{2}$ it is not
  defined only in the points $(u,p_u)$ with
  \begin{equation}
    u^4\leq\frac{2(\mu+\llS)+fR^2}{2(\mu-\llS)-fR^2}R^2. 
    \label{defDomainMapP2}
  \end{equation}
  \label{dominiomappa}
\end{proposition}

\begin{proof}
  For each point $(u,p_u)$ in the domain of the map, we
  have
  \[v = \frac{R}{u},\] 
  and $p_v$ is defined by \eqref{HJStark}$_2$.  Condition
  \eqref{defDomainMap} corresponds to $p_v^2\leq 0$, that is not
  possible. In the second and fourth quadrant of
  the $(u,p_u)$ plane, the condition $up_v > \max(0,-p_uv)$ results in
  \[
  p_vu>-p_u\frac{R}{u}
  \]
  which implies 
  \[
  \bigl(2(\mu-\llS)-fR^2\bigr)u^4-2(\mu+\llS)-fR^4>0.
  \]
  If $\llS > \mu - \frac{fR^2}{2}$, we get
  \[
  u^4 < \frac{2(\mu+\llS)+fR^2}{2(\mu-\llS)-fR^2}R^2 <0,
  \]
  meaning that the map is not defined.  On the other hand, if
  $\llS < \mu - \frac{fR^2}{2}$, the previous condition is
  fulfilled for
  \[
  u^4>\frac{2(\mu+\llS)+fR^2}{2(\mu-\llS)-fR^2}R^2.
  \]
  
\end{proof}

\begin{figure}[t!]
  \begin{center}
    \includegraphics[width=1\textwidth]{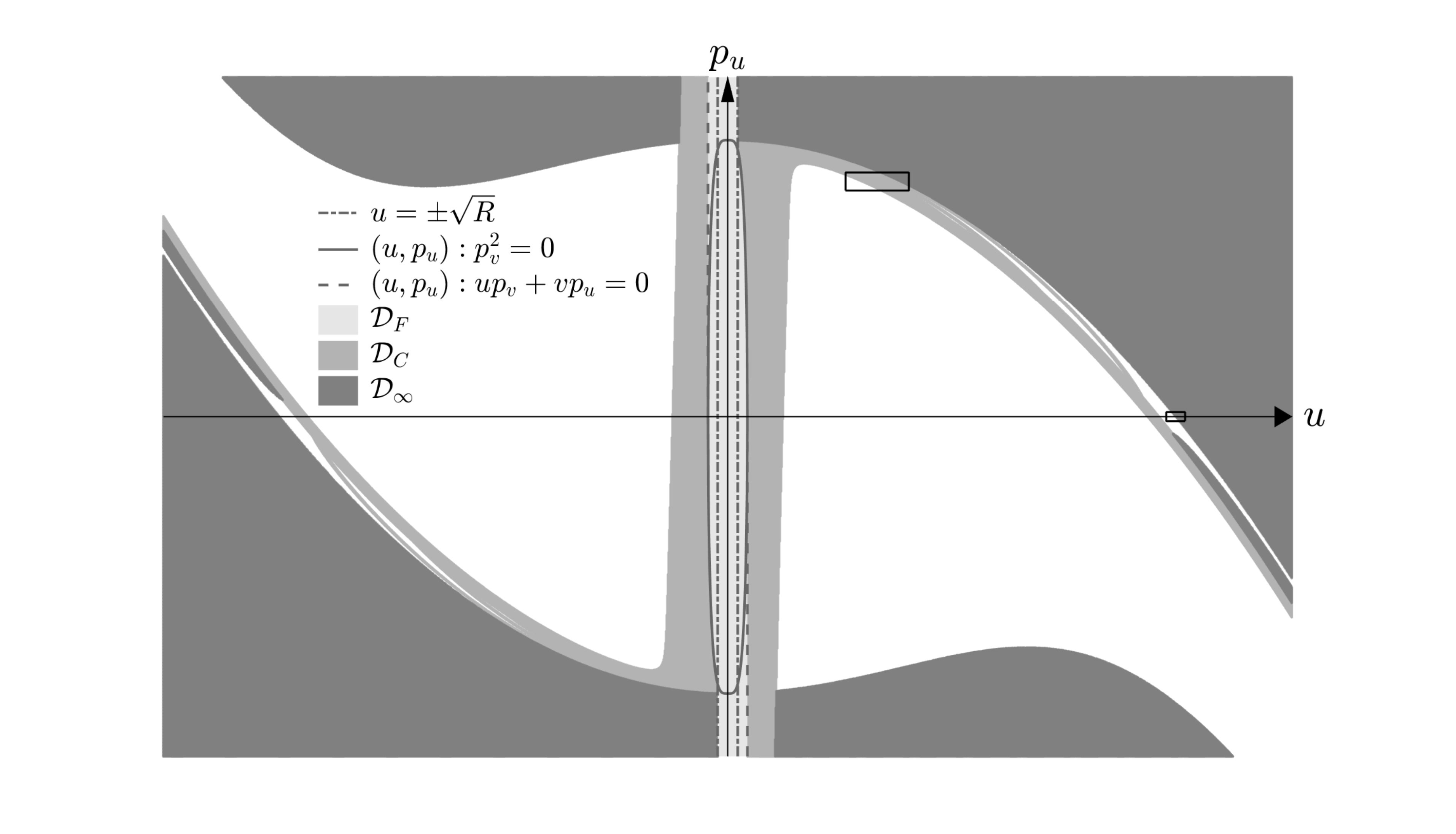}
    \caption{The domain ${\domsigma}$ of the Sun-shadow map is
      represented by the white area.
      $\kilo\cubic\meter\per\squaren\second$, $f=9.12\times 10^{-9}$
      $\kilo\meter\per\squaren\second$.  A magnification of the region
      in the small and smaller rectangles is shown in
      Figures~\ref{fig:ssdomEn} and \ref{fig:unstManConstr}
      respectively.}
    \label{fig:ssdom}
  \end{center}
\end{figure}
  
The domain $\domsigma\subset \R^2 $ does not include
the points defined in Proposition~\ref{dominiomappa}, nor the points
corresponding to trajectories which go to infinity or collide with the
Earth before going back to $\Sigma$.
In Figure~\ref{fig:ssdom}, for a specific choice of $\llS$ and $f$, the domain ${\domsigma}$ is drawn as the white area in a portion of the $(u,p_u)$ plane. The light grey region represents the set ${\cal D}_{F}$ of forbidden points in Proposition~\ref{dominiomappa}. The other two grey areas contain part of the sets ${\cal D}_{\infty}$ (darker) and ${\cal D}_{C}$ (lighter) corresponding to the trajectories which go to infinity and collide with the Earth, respectively. 



\begin{figure}[h!]
  \begin{center}
    \includegraphics[width=1\textwidth]{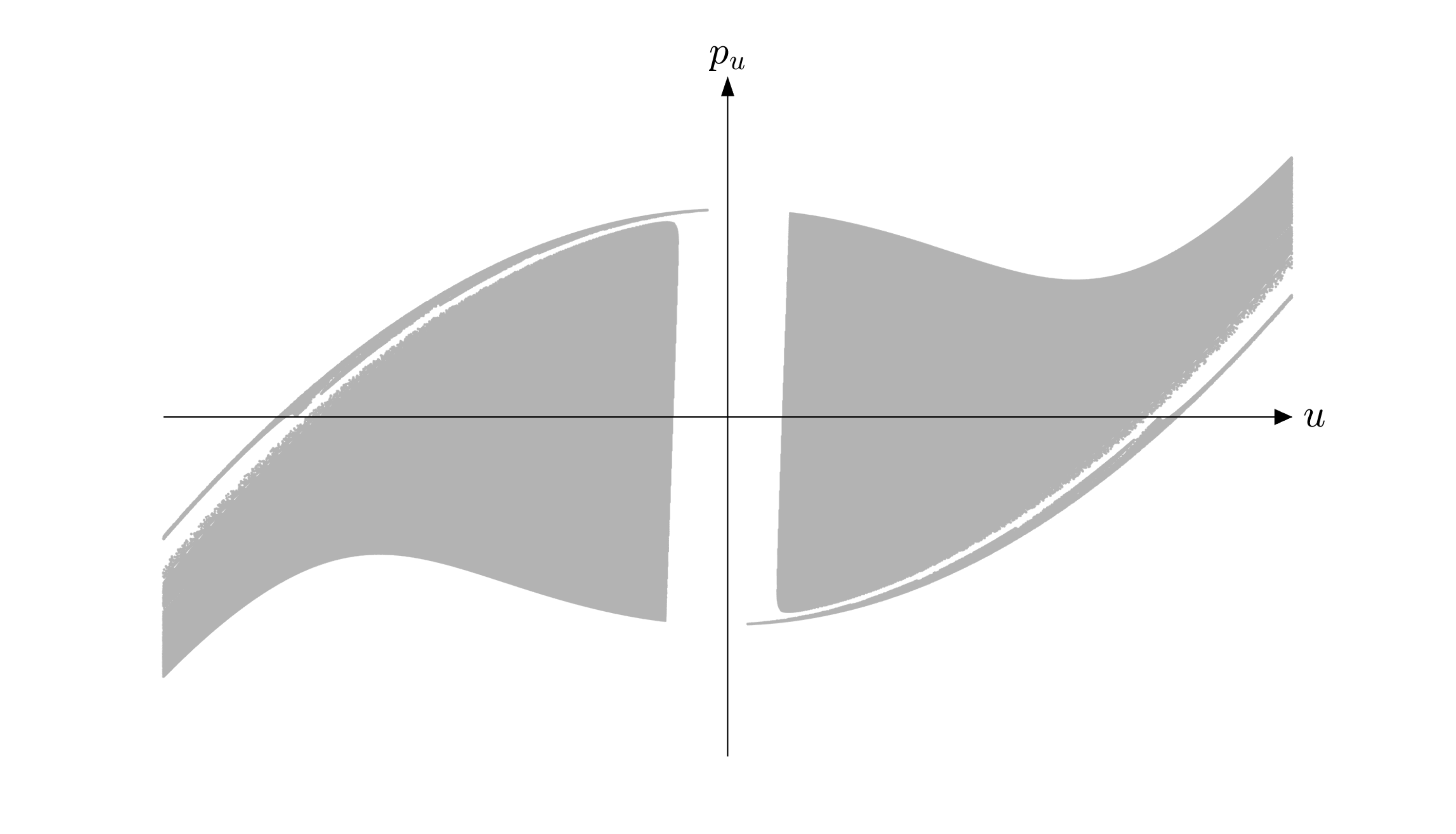}
    \caption{The image of the Sun-shadow map.}
    \label{fig:ssimage}
  \end{center}
  
\end{figure}

Figure~\ref{fig:ssimage} shows the image of the
domain $\domsigma$ under $\mathfrak{S}$ in the same portion of the $(u,p_u)$ plane represented in Figure~\ref{fig:ssdom}.

\begin{remark}
In the image of the map we may have points belonging to
${\cal{D}}_{\infty}$ or ${\cal{D}}_{C}$, so that we cannot iterate the
map again.
\end{remark}

\begin{figure}[h!]
  \begin{center}
    \begin{subfigure}{0.45\textwidth}
      \includegraphics[width=\textwidth]{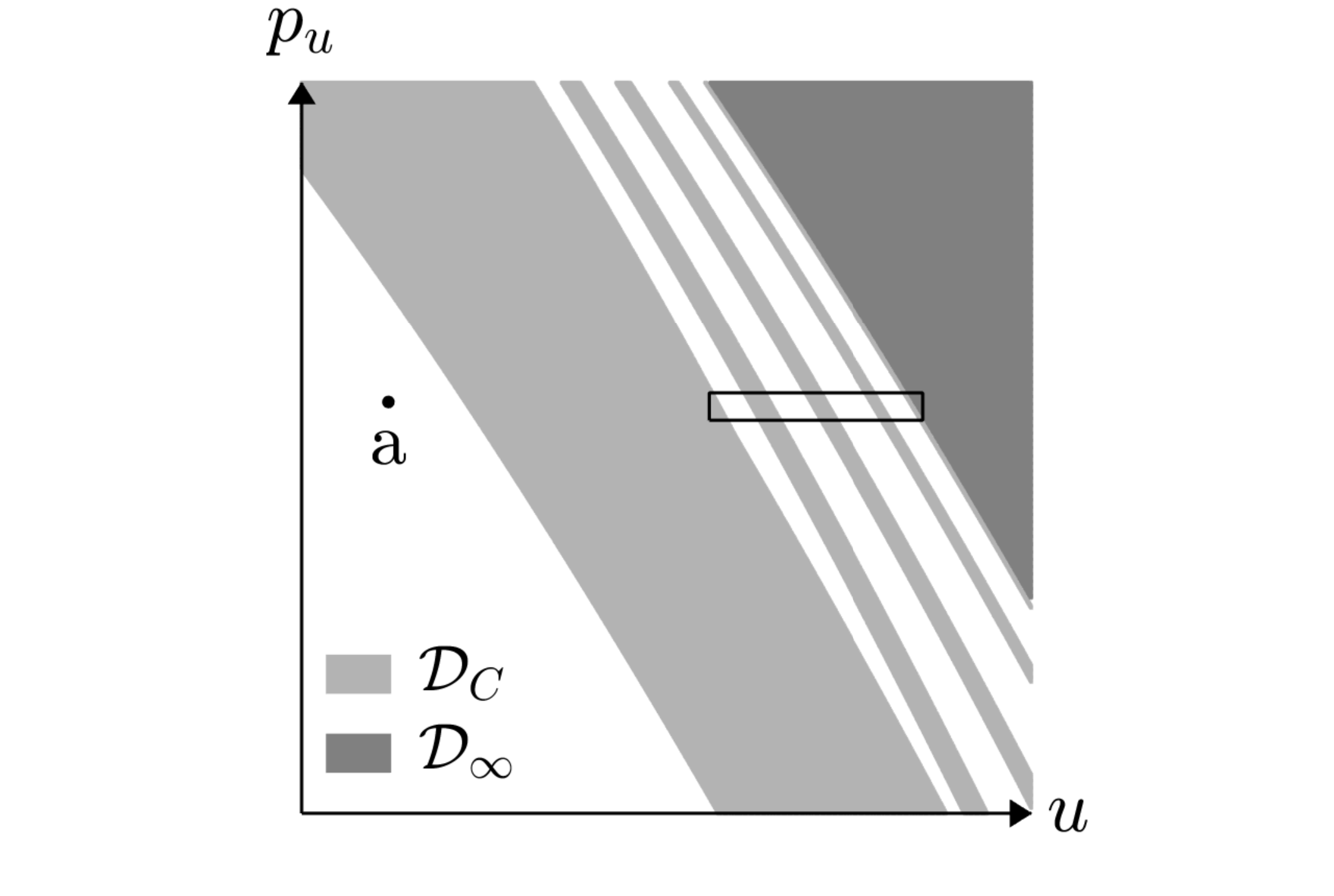}
    \end{subfigure}
    \begin{subfigure}{0.45\textwidth}
      \includegraphics[width=\textwidth]{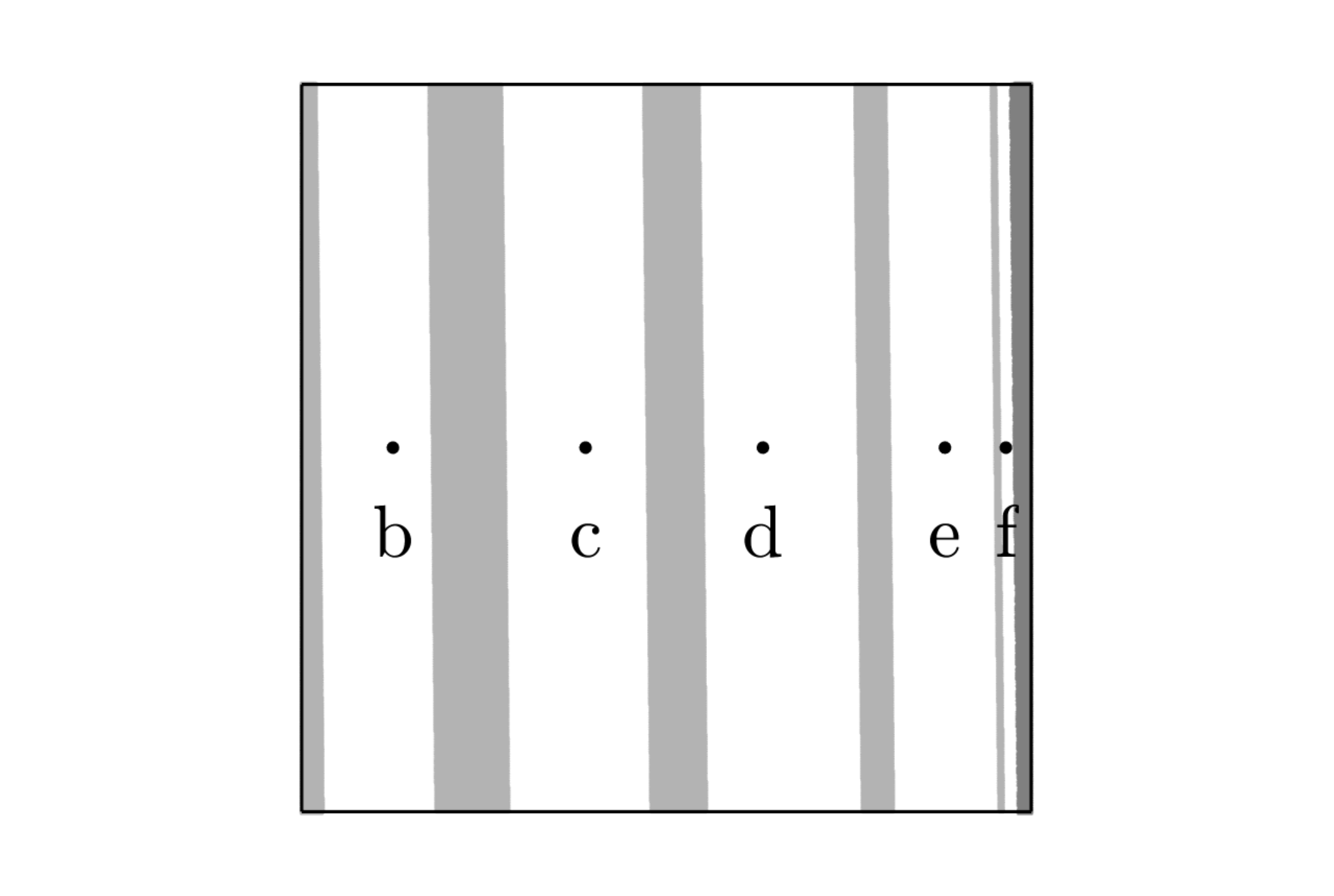}
    \end{subfigure}
    \begin{subfigure}{0.45\textwidth}
      \includegraphics[width=\textwidth]{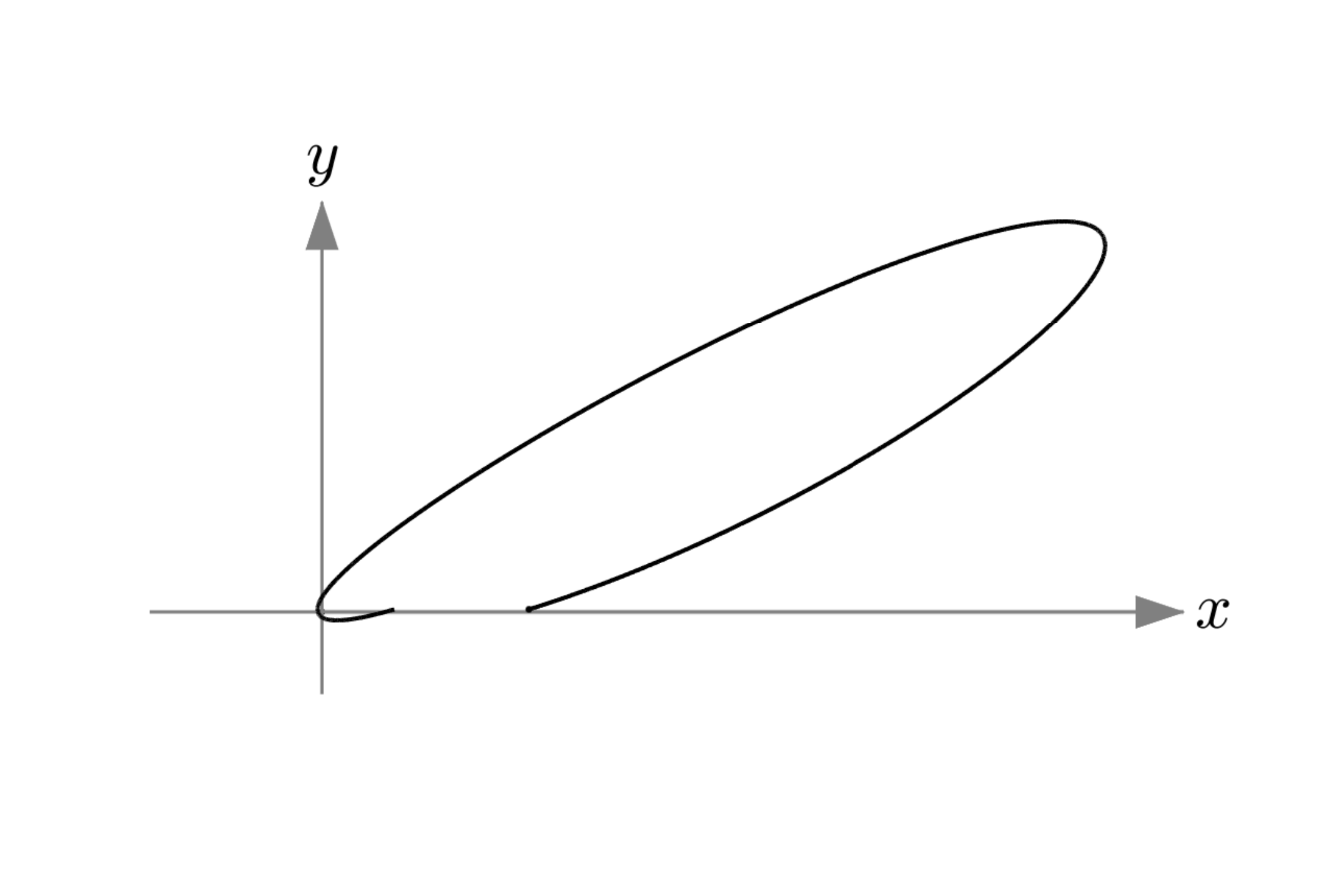}
      \caption{}
    \end{subfigure}
    \begin{subfigure}{0.45\textwidth}
      \includegraphics[width=\textwidth]{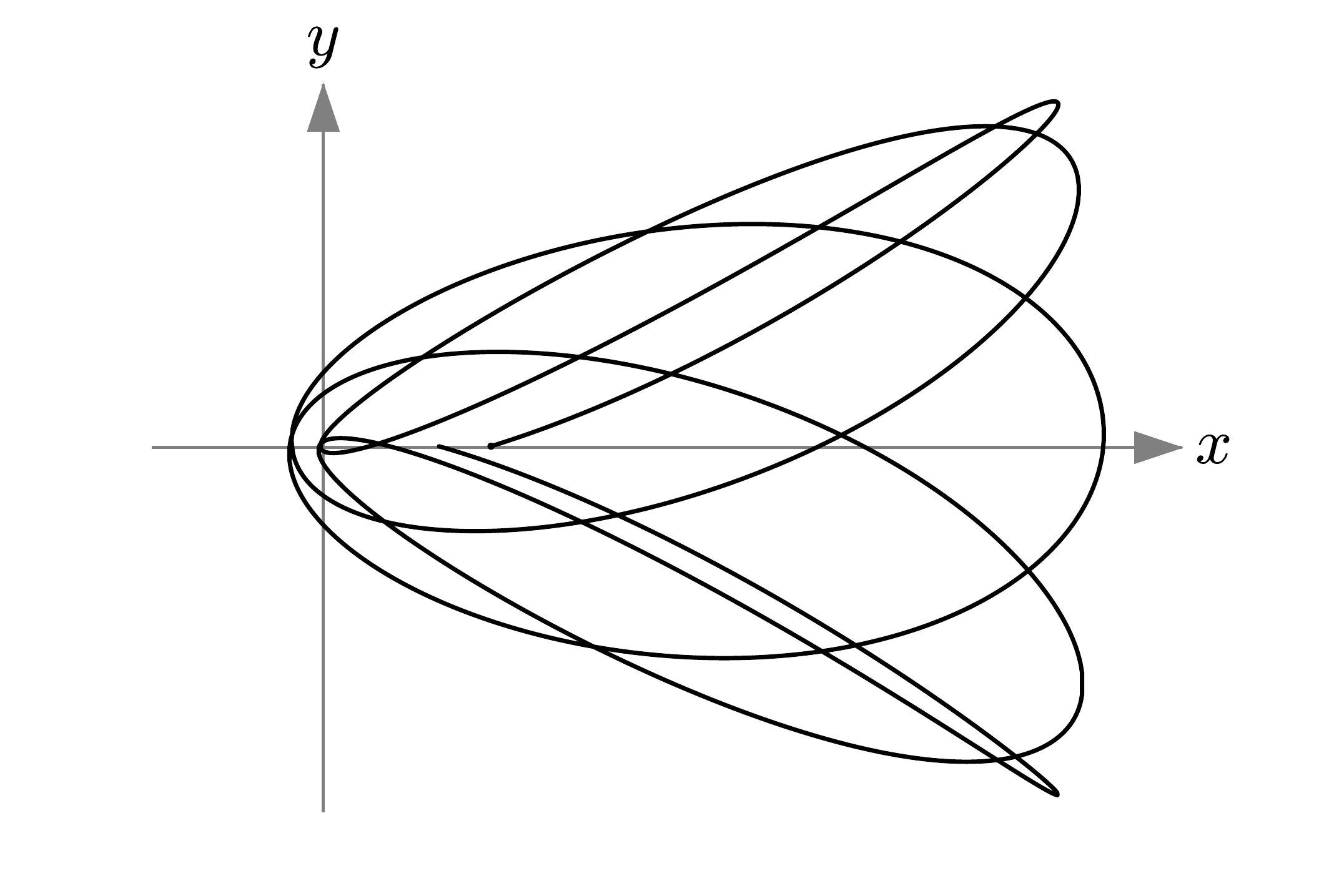}
      \caption{}
    \end{subfigure}
    \begin{subfigure}{0.45\textwidth}
      \includegraphics[width=\textwidth]{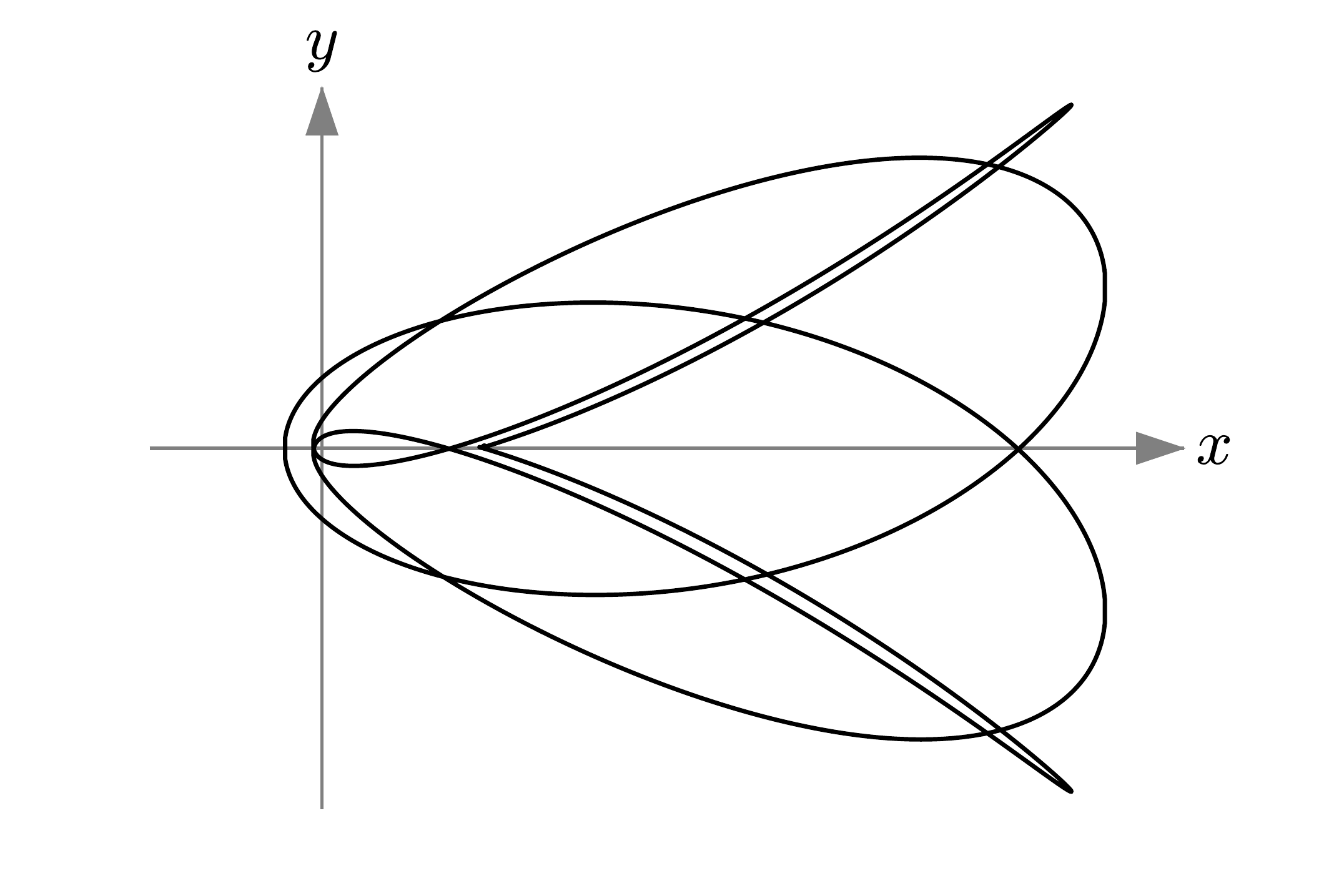}
      \caption{}
    \end{subfigure}
    \begin{subfigure}{0.45\textwidth}
      \includegraphics[width=\textwidth]{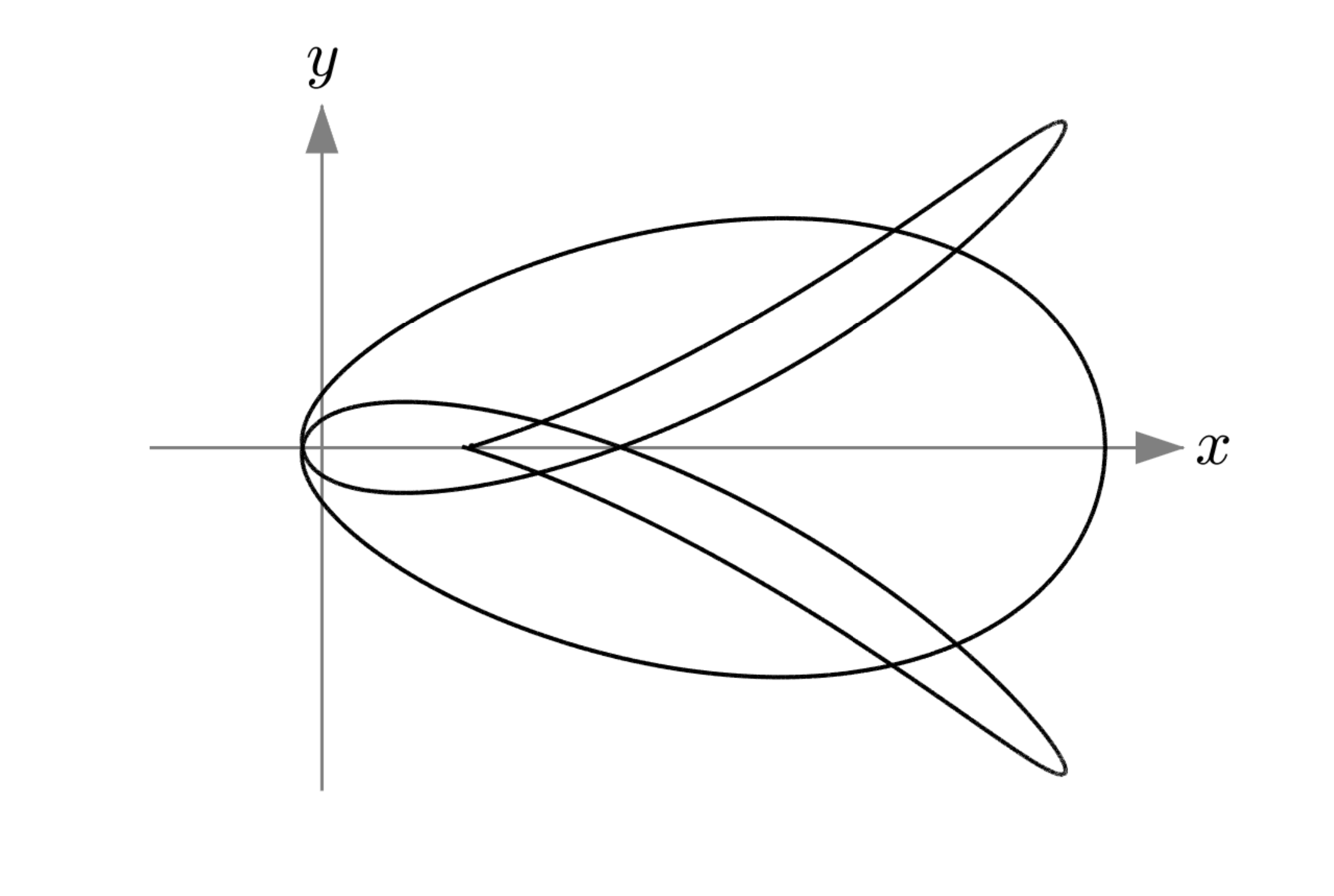}
      \caption{}
    \end{subfigure}
    \begin{subfigure}{0.45\textwidth}
      \includegraphics[width=\textwidth]{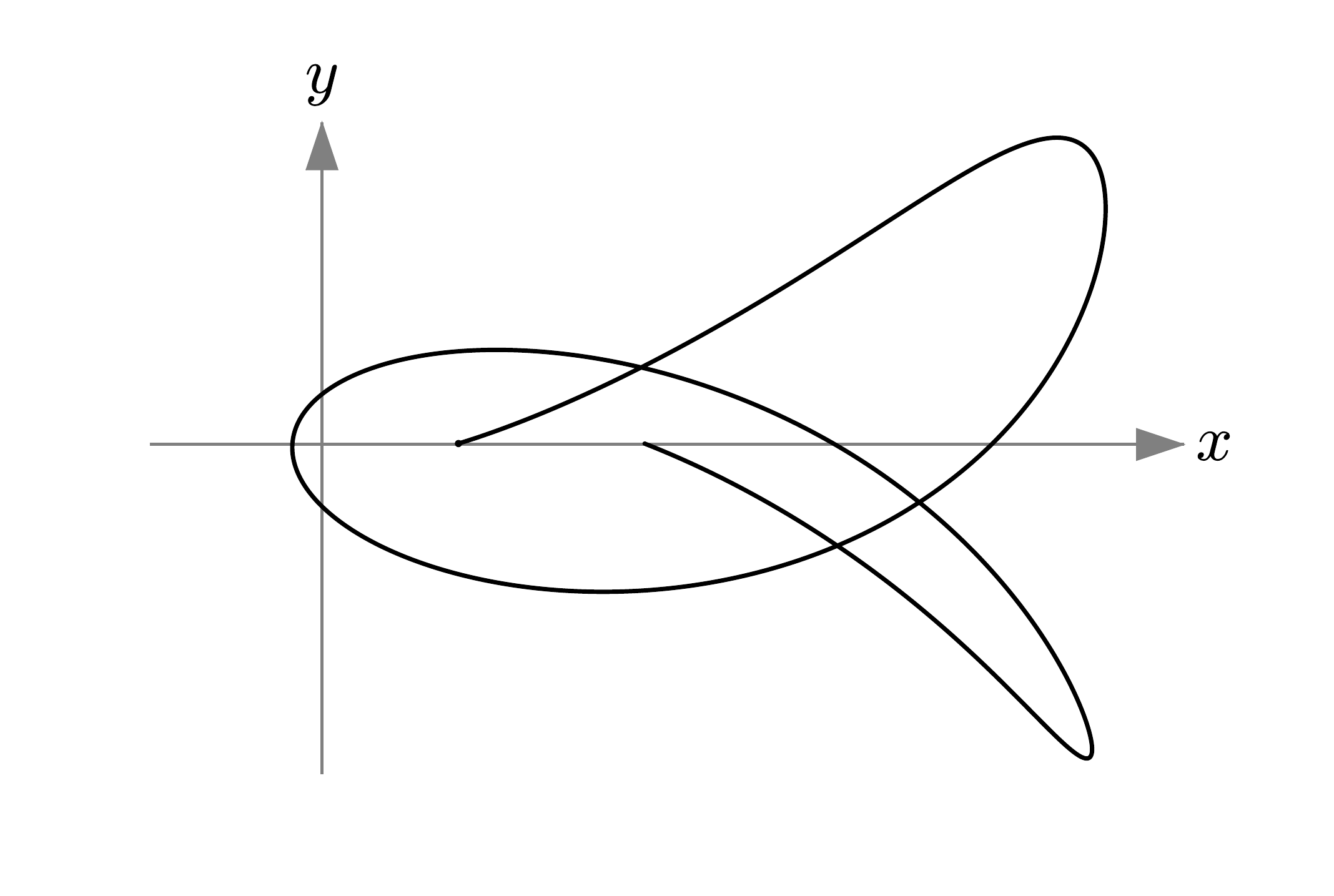}
      \caption{}
    \end{subfigure}
    \begin{subfigure}{0.45\textwidth}
      \includegraphics[width=\textwidth]{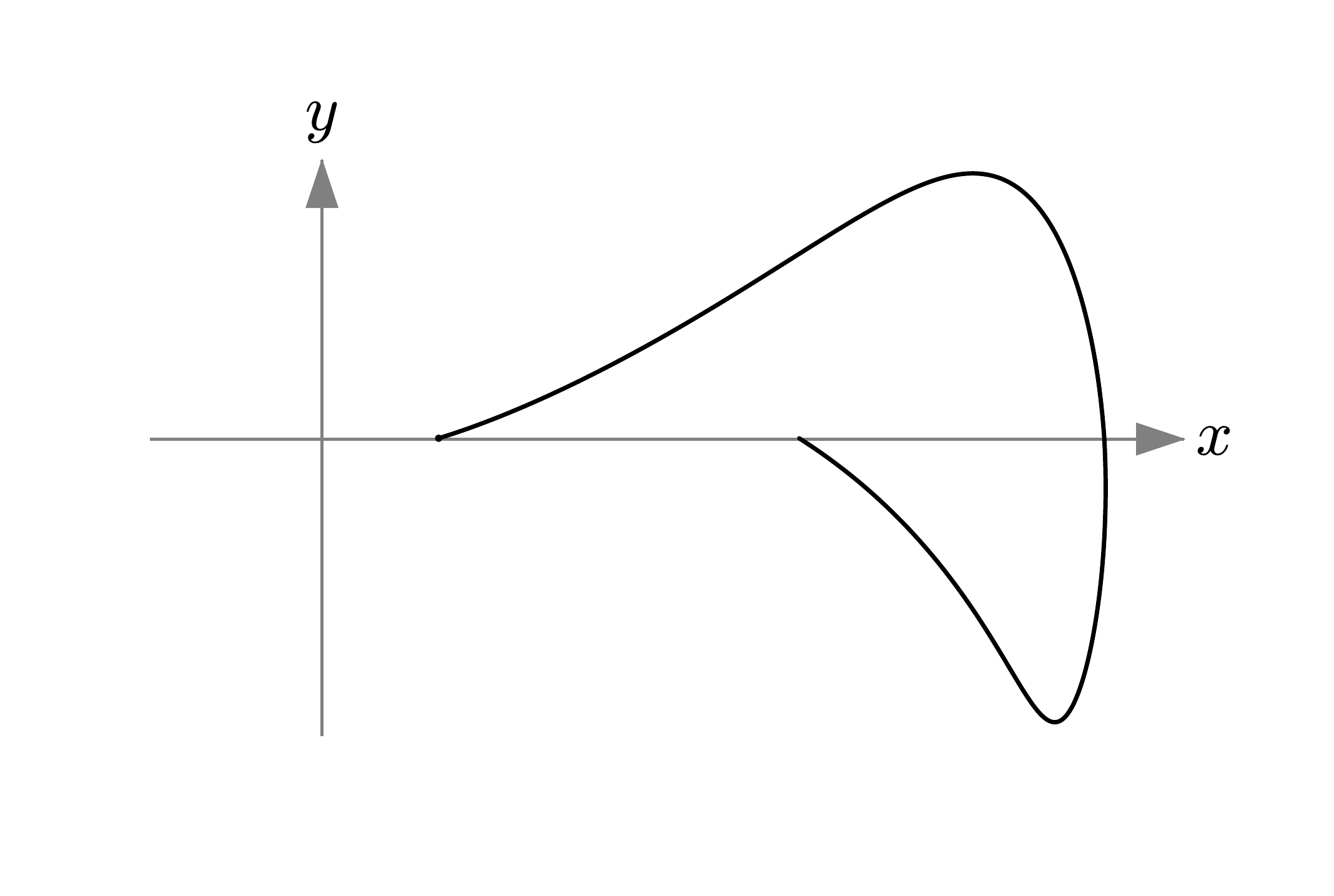}
      \caption{}
    \end{subfigure}
  \caption{Orbits with different winding numbers, corresponding to one
    iteration of $\mathfrak{S}$. The selected initial points are shown
    in the two top pictures and are labeled with a, b, c, d, e, f. On
    top right there is the magnification of the small rectangular region
    on top left.}
  \label{fig:ssdomEn}
  \end{center}
\end{figure}

In Figure~\ref{fig:ssdomEn}, the magnification of the larger rectangular region appearing in Figure~\ref{fig:ssdom} highlights the complexity of the structure of ${\domsigma}$. We have selected five points, labelled with b, c, d, e, f, in the white `{\em corridors}', and one point, labelled with a, in the larger white region on the left. In the same figure we show the portion of the trajectories corresponding to one iteration of the selected points under the map. A winding number around the origin can be associated to each trajectory by joining with a straight line their initial and final points. The values of this topological invariant are $-4$, $-3$, $-2$, $-1$, $0$, $+1$ for the cases (b), (c), (d), (e), (f), (a), respectively.  Since we get a different value 
for all these cases, the white corridors must belong to different connected components of ${\cal D}$.

\begin{proposition}
  The map $\mathfrak{S}$ is differentiable in its definition domain
  $\domsigma$.
  \label{mapDiff}
\end{proposition}
\begin{proof}

Let $\bPhi_s(\tau;{\Ubf},\tau_0)$ and $\bPhi_k(\tau;{\Ubf},\tau_0)$ be
the integral flow of Stark's and Kepler's dynamical systems
\eqref{UDynSystK}, \eqref{UDynSystS}.
Consider $(u,p_u)\in\domsigma$ and the corresponding orbit in the
Sun-shadow dynamics. Before it goes back to the section $\Sigma$, the
dynamical regime will change $n$ times, with $n$ depending on the
shape of the trajectory. The first regime will be always Stark's, the
last will be Kepler's.  Let us introduce a finite sequence of
sections $\Sigma_i$, $i=0,\ldots,n$ where the dynamics
changes, with $\Sigma_0 = \Sigma_{n}=\Sigma$. Each of them is given
by $s_i({\Ubf})=0$: for the section $\Sigma$ we have $s_i = uv-R$,
while for the intermediate sections, with $i=1,\ldots, n-1$, we have
$s_i=uv+R$ or $s_i=uv-R$ depending on the boundary of the shadow
region that is crossed when the dynamics changes. It is possible to
define the maps $\mathfrak{S}_i$
\begin{equation}
	\begin{split}
	\mathfrak{S}_i: \domsigma_i\subset\R^2 &\to \R^2,\cr
	(u^{i},p^{i}_u) &\mapsto (u^{i+1},p^{i+1}_u),\cr
	\end{split}
	\label{intmaps}
\end{equation}
where
$\domsigma_i$ is made of the points
\[
  {\Ubf}^{i}= (p^{i}_u,p^{i}_v,u^{i},v^{i})^T \in \Sigma_{i}, \qquad v^{i}=\pm R/u^{i}, \qquad p^{i}_v = p_v(p^{i}_u,u^{i}).
  \]  
The Sun-shadow map will be given by

\[\mathfrak{S} = \mathfrak{S}_{n-1} \circ ... \mathfrak{S}_{2}\circ \mathfrak{S}_{1}\circ \mathfrak{S}_{0}.\]
Thus, we have 
\begin{equation}
D\mathfrak{S}(u,p_u) = 
\frac{\partial (u',p'_u)}{\partial (u,p_u)} =
\frac{\partial(u',p'_u)}{\partial(u^{n-1},p^{n-1}_u)} \frac{\partial(u^{n-1},p^{n-1}_u)}{\partial(u^{n-2},p^{n-2}_u)}... \frac{\partial(u^2,p^2_u)}{\partial(u^{1},p^{1}_u)} \frac{\partial(u^{1},p^{1}_u)}{\partial(u,p_u)}
\label{mappJacobian}
\end{equation}
where 
\[
\frac{\partial(u^{i+1},p^{i+1}_u)}{\partial(u^{i},p^{i}_u)} =
A\Big(\bX_{i+1}({\Ubf}_{i+1})  \frac{\partial \tau}{\partial {\Ubf}_{i}}
+ \frac{\partial \bPhi_{i+1}}{\partial {\Ubf}_{i}}(\tau({\Ubf}_{i}),{\Ubf}_{i},\tau_{i})\Big) \frac{\partial {\Ubf}_{i}}{\partial (u^{i},p^{i}_u)},
\]
with 
\[A = \begin{bmatrix}
1 & 0 & 0 & 0 \\
0 & 0 & 1 & 0
\end{bmatrix}.\]
The integral flow $\bPhi_{i+1}$ (resp. $\bX_{i+1}$) is equal to either $\bPhi_s$ or $\bPhi_k$ (resp. $\bX_s$ or $\bX_k$) depending on the regime between the two sections.
The term $\partial \tau/\partial {\Ubf}_{i}$ can be computed as 
\[
\frac{\partial \tau}{\partial {\Ubf}_{i}}
= - \frac{1}{\frac{\partial s_{i+1}}{\partial {\Ubf}}({\Ubf}_{i+1}) \cdot \bX_{i+1}({\Ubf}_{i+1})} \Big(\frac{\partial s_{i+1}}{\partial {\Ubf}}({\Ubf}_{i+1})\Big)^T \ \frac{\partial \bPhi_{i+1}}{\partial {\Ubf}_{i}}(\tau({\Ubf}_{i}),{\Ubf}_{i},\tau_{i}),
\]
see \cite{Simo_1989}, while  $\partial \bPhi_{i+1}/\partial {\Ubf}_{i}$
fulfils
\[
\begin{split}
&\frac{d}{d\tau} \Big( \frac{\partial \bPhi_{i+1}}{\partial \Ubf_{i}}(\tau({\Ubf}_{i}),{\Ubf}_{i},\tau_{i})\Big) = \frac{\partial \bX_{i+1}}{\partial \Ubf}\Big(\bPhi_{i+1}(\tau({\Ubf}_{i}),{\Ubf}_{i},\tau_{i}),\tau({\Ubf}_{i})\Big) \frac{\partial \bPhi_{i+1}}{\partial \Ubf_{i}}(\tau({\Ubf}_{i}),{\Ubf}_{i},\tau_{i}),\cr
&\frac{\partial \bPhi_{i+1}}{\partial \Ubf_{i}}(\tau_{i},{\Ubf}_{i},\tau_{i})= I,\cr
\end{split}
\]
with $I$ the identity matrix and $\tau_{i}$ the value of the fictitious time at the
section $\Sigma_{i}$.









\end{proof}

\begin{proposition}
	The map $\mathfrak{S}$ is not area-preserving.
\end{proposition}
\begin{proof}
We give a numerical proof by showing that a circular region of
$\Sigma$ is mapped into a region with a different area.  In the
section $\Sigma$, we can consider a closed curve $\gamma_0$ symmetric with respect to the $u$ axis, defined by
\[
c\,p_u^2+(u-u_C)^2=r_C^2
\]
with 
\[
u_C > r_C + \Big(  \frac{2(\mu+\llS)+fR^2}{2(\mu-\llS)-fR^2}R^2 \Big)^{\frac{1}{4}}.
\]
We chose $c=1$ $\squaren\second\per\kilo\squaren\meter$ so that
$\gamma_0$ becomes a circumference of radius $r_C$ centred at
$(u,p_u)=(u_C,0)$.  We sample it with $m$ points. Each point defines a trajectory in the phase space that we propagate with a Runge-Kutta method of Gauss type, by properly switching dynamics at the boundary between Stark's and Kepler's regimes, until its next intersection
with the section $\Sigma$. In this way, we obtain the image of the initial points under the Sun-shadow map. The resulting points belong to the closed curve $\gamma_1 = \mathfrak{S}(\gamma_0)$. To compute the area $A_1$ of the region enclosed by $\gamma_1$, first we parametrise it by a variable $\theta$. In particular, for each point
on $\gamma_1$ we compute the values of the parameter $\theta$:
\[
\theta^1 = 0; \qquad \theta^j =
\theta^{j-1}+\sqrt{(p_{u}^j-p_{u}^{j-1})^2+(u^{j}-u^{j-1})^2},\qquad j
= 2,\ldots,m
\]
where $(u^j,p_{u}^j)$ are the coordinates of the points, and
$(u^m,p_{u}^m)=(u^1,p_{u}^1)$. Then, we interpolate the points
$(\theta^j,u^j),(\theta^j,p_u^j)$ by cubic splines. Finally, we
compute the area by applying the Gauss-Green formula
\[
\int_{A_1} dudp_u = \int_{\gamma} – p_udu + udp_u,
\]  
and we get
\[
A_{1} = \sum_{j=1}^{m-1}\int_{\theta_j}^{\theta_{j+1}}
-p_u(\theta)\frac{du}{d\theta}(\theta)d\theta +
u(\theta)\frac{dp_u}{d\theta}(\theta)d\theta.
\]
The resulting area $A_{1}$ is different from the area $A_0=\pi r^2_C/c$
of the region enclosed by $\gamma_0$. Indeed, assuming that
\[
\llS=348600 \ \kilo\cubic\meter\per\squaren\second, \quad
f = 9.12\times\nanod \  \kilo\meter\per\squaren\second,\quad
u_C = 1250 \  \kilo\meter ^{1/2},\quad
r_C = 250 \  \kilo\meter ^{1/2},
\]
and sampling the initial circumference with $2\times 10^5$ points, in double precision we get
\[
A_0 \approx 1.9635\times 10^5 \ \kilo\squaren\meter\per\second ,\qquad 
A_1 \approx 1.9588\times 10^5 \  \kilo\squaren\meter\per\second. 
\]

\end{proof}

\subsection{Hyperbolic fixed points}

Given $\llS \in [\llS^{-},\llS^{+}]$, the periodic orbit of brake type
$\widehat{\bx}(t;\llS)$ gives rise to two fixed points $\widehat{\bUps}_{1,2}$ of the Sun-shadow map:
\begin{equation*}
  \fpD= (\sqrt{\xiE},-\puE), \qquad  \fpS= (-\sqrt{\xiE},\puE),
\end{equation*}  
where $\xiE$ is given by equation \eqref{xiEDef} and 
\[
\puE^2 = 2\enHat\xi_{\tiny E}+2(\mu+\llS)+f\xi_{\tiny E}^2,\qquad \puE>0,
\]
with $\enHat$ the energy of $\widehat{\bx}$ in Stark's regime,
$\enHat\in[\enBar,\ensstar)$.  The point $\fpD$ lies in
  the region included into the smaller rectangle appearing in
  Figure~\ref{fig:ssdom}.
 
We can evaluate the Jacobian matrix of the map at $ \widehat{\bUps}_{j}$, $j=1,2$, using equation \eqref{mappJacobian}: it has two real eigenvalues $\lambda^j_1,\lambda^j_2$ with

\[
0<\lambda_1^{j}<1<\lambda_2^{j}.
\]
For example, by taking $\llS=348600 \ \kilo\cubic\meter\per\squaren\second$, we obtain $\lambda_1^{j}=1.54\times 10^{-4}$ and  $\lambda_2^{j}=6.48\times 10^{3}$, for $j=1,2$.
Thus, the two fixed points are hyperbolic. It follows that the periodic orbit of brake type is unstable.

\subsection{Invariant manifolds}

Here, we describe the numerical technique used for the computation of the invariant manifolds of the fixed points of the Sun-shadow map. The discussion will be focused on $\fpD$, but the procedure is the same also for $\fpS$. 

We took inspiration from the method in \cite{Hobson_1993}, thought
specifically for planar maps. This algorithm can be applied only to
two-dimensional maps which have saddle-type fixed points and whose
Jacobian matrix, evaluated at these points, has two real eigenvalues
$\lambda_1,\lambda_2$ with $0<\lambda_1<1<\lambda_2$. As
previously shown, the Sun-shadow map $\mathfrak{S}$ and its fixed
point $\fpD$ fulfil these requirements. We describe the algorithm for
the case of one branch $B$ of the unstable manifold. The stable manifold
can be constructed in a similar manner using the inverse map
$\mathfrak{S}^{-1}$. The method consists in dividing the branch of
the manifold into a sequence of primary segments. A primary segment
$V$ is a connected subset of $B$ whose last point is the image of its
first point under the map. Given an initial primary $V_0$ in a
neighbourhood of the fixed point, all the following primaries can be
obtained by iterating the map $m$ times:
\[
V_{i+1} = \mathfrak{S}(V_{i}), \qquad i = 0,\ldots,m-1.
\]
The branch will be given by the union of the computed primaries: 
\[
B = \bigcup_i V_i.
\]
The initial primary is approximated with a segment very close to $\fpD$ along an unstable eigenvector $\bm{\psi}$ of $D\mathfrak{S}(\fpD)$. Then, it is corrected by using the technique described in \cite{LegaGuzzo_2016}, based on the Modified Fast Lyapunov Indicators (MFLI). The lower is the distance of a point from the manifold, the larger is its MFLI. Thus, for each point $\bUps$ in the sample of the primary the correction is done as follows: 
\begin{enumerate}
	\item consider a small neighbourhood of $\bUps$ in the	direction orthogonal to the corresponding primary curve, and sample it uniformly;
	\item compute the MFLI of $\bUps$ and of each point of the sample;
	\item select the point with the larger MFLI. 
\end{enumerate}
There is an issue concerning the iterations of the primaries. We observed that, after a few
iterations, portions of the primaries are lost
because the corresponding trajectories never return to $\Sigma$: by consequence the primaries lose their
nature of connected sets. We decided to relax the definition of primaries given by Hobson by admitting primaries with several connected components, that we still
denote by $V_i$.

We summarize below our algorithm:
\begin{enumerate}
\item approximate the initial primary $V_0$ with a segment aligned with
  the eigenvector $\bm{\psi}$ in a small neighbourhood of $\fpD$: this segment is sampled with $n$ points distributed
  according to the exponential law
  \[
  \bUps_{i}=\bUps_{i-1}+a \ 2^i \ \bm{\psi},\qquad i=1,\ldots,n-2,
  \]
  with ${\bUps}_i=(u_i,p_{u_i})$ and $a\in\R$. The last one,
  $\bUps_{n-1}$, is the image of $\bUps_{0}$ under the map.
  We call $\tilde{V}_0$ the finite set of points approximating $V_0$;
  
\item correct $\tilde{V}_0$ by using the MFLI, as previously described; 
  
\item iterate the corrected $\tilde{V}_0$ once, and obtain the set
  \[
  \tilde{V}_1 = \{\bm{\upsilon'}_{i}=(u'_i,p'_{u_i}), \  i=0,\ldots,n-1\};
  \]
\item interpolate the points in $\tilde{V}_1$ with cubic splines and
  sample more densely the resulting curve (we still denote by
  $\tilde{V}_1$ the new sample);
  
\item correct $\tilde{V}_1$ by the MFLI, as done for $\tilde{V}_0$;
\item iterate the corrected $\tilde{V}_1$ under the map $m-1$ times.
\end{enumerate}

\begin{figure}[h!]
  \begin{center}
    \begin{subfigure}{0.47\textwidth}
      \includegraphics[width=1.1\textwidth]{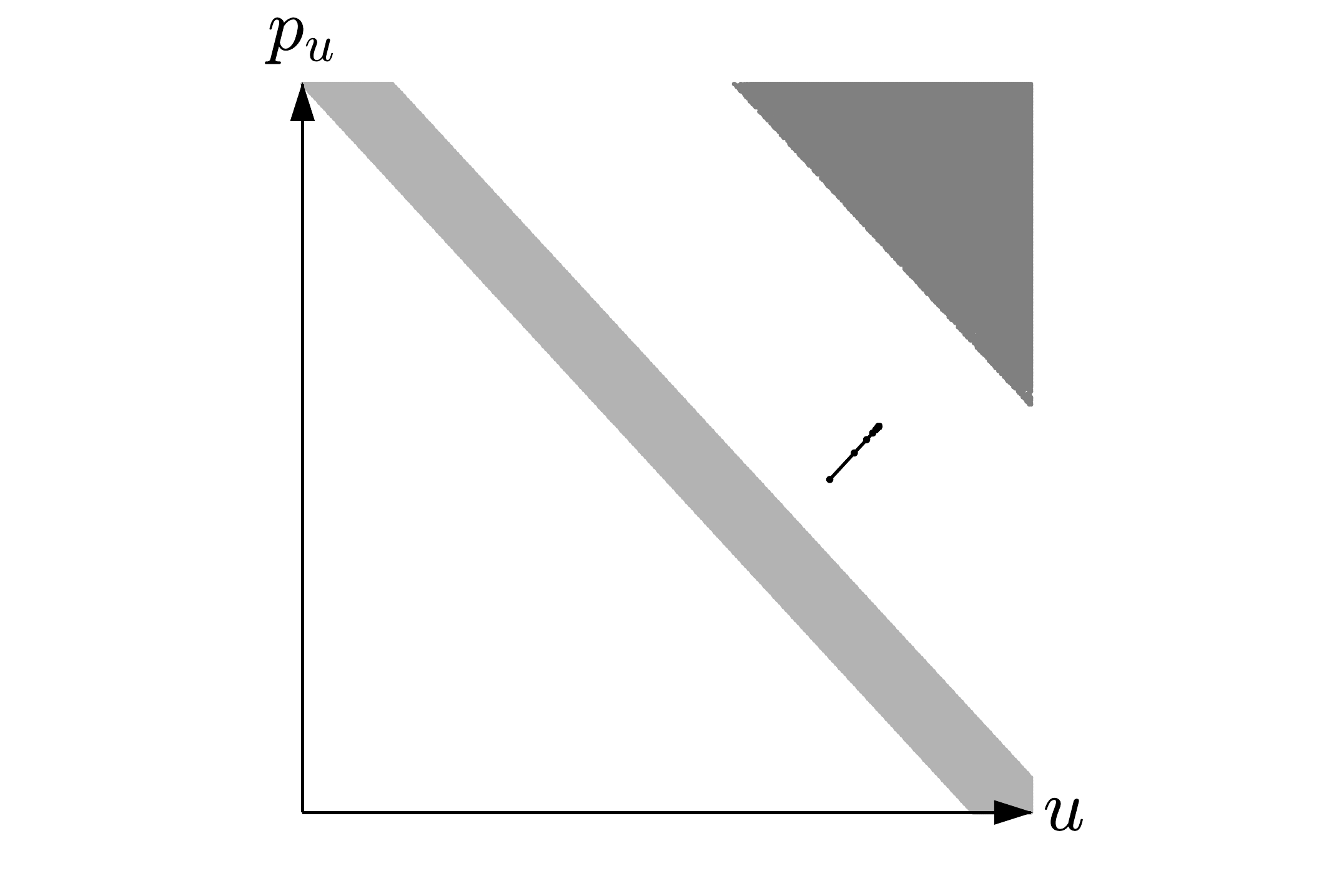}
    \end{subfigure}
    \begin{subfigure}{0.47\textwidth}
      \includegraphics[width=1.1\textwidth]{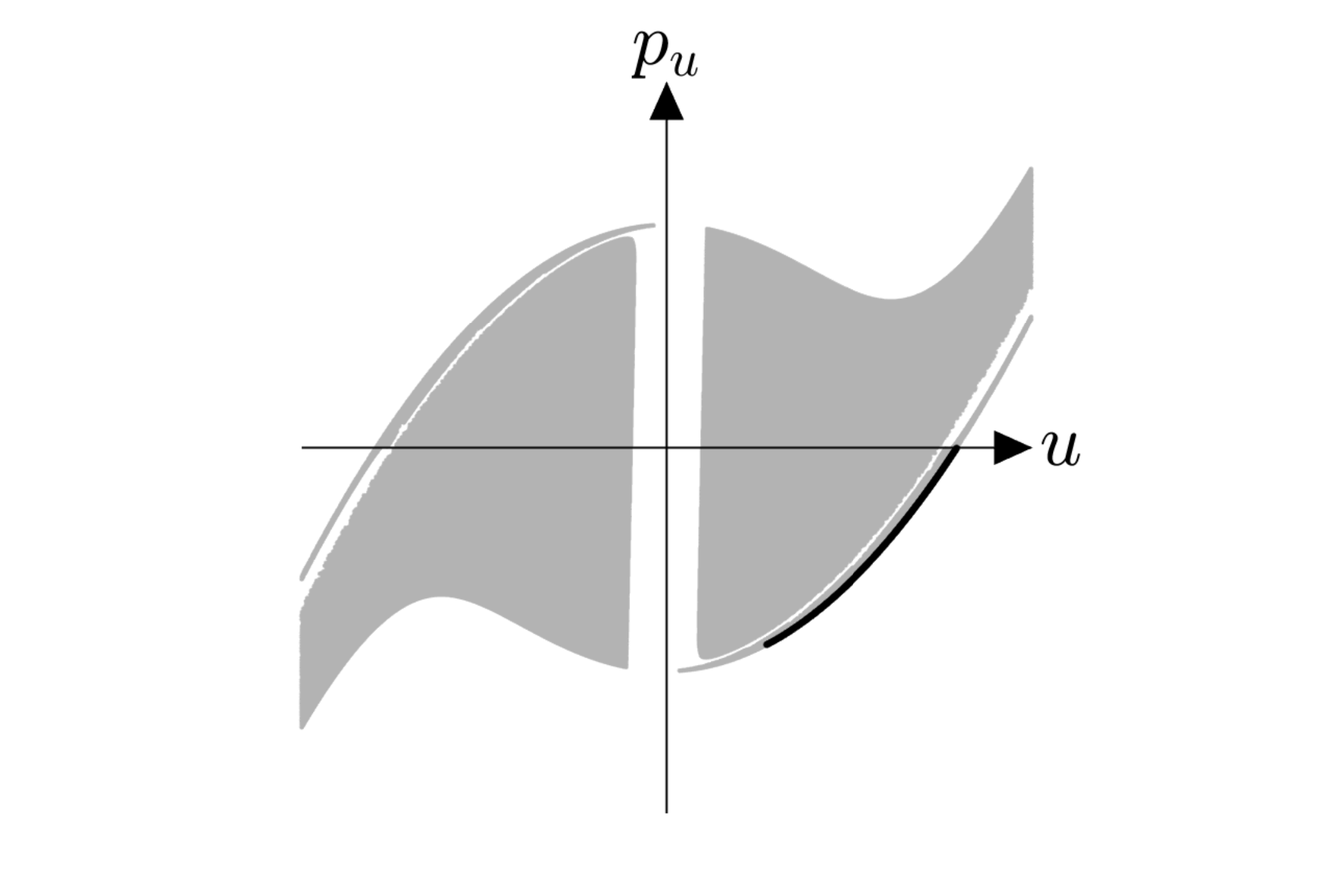}
    \end{subfigure}
    \begin{subfigure}{0.47\textwidth}
      \includegraphics[width=1.1\textwidth]{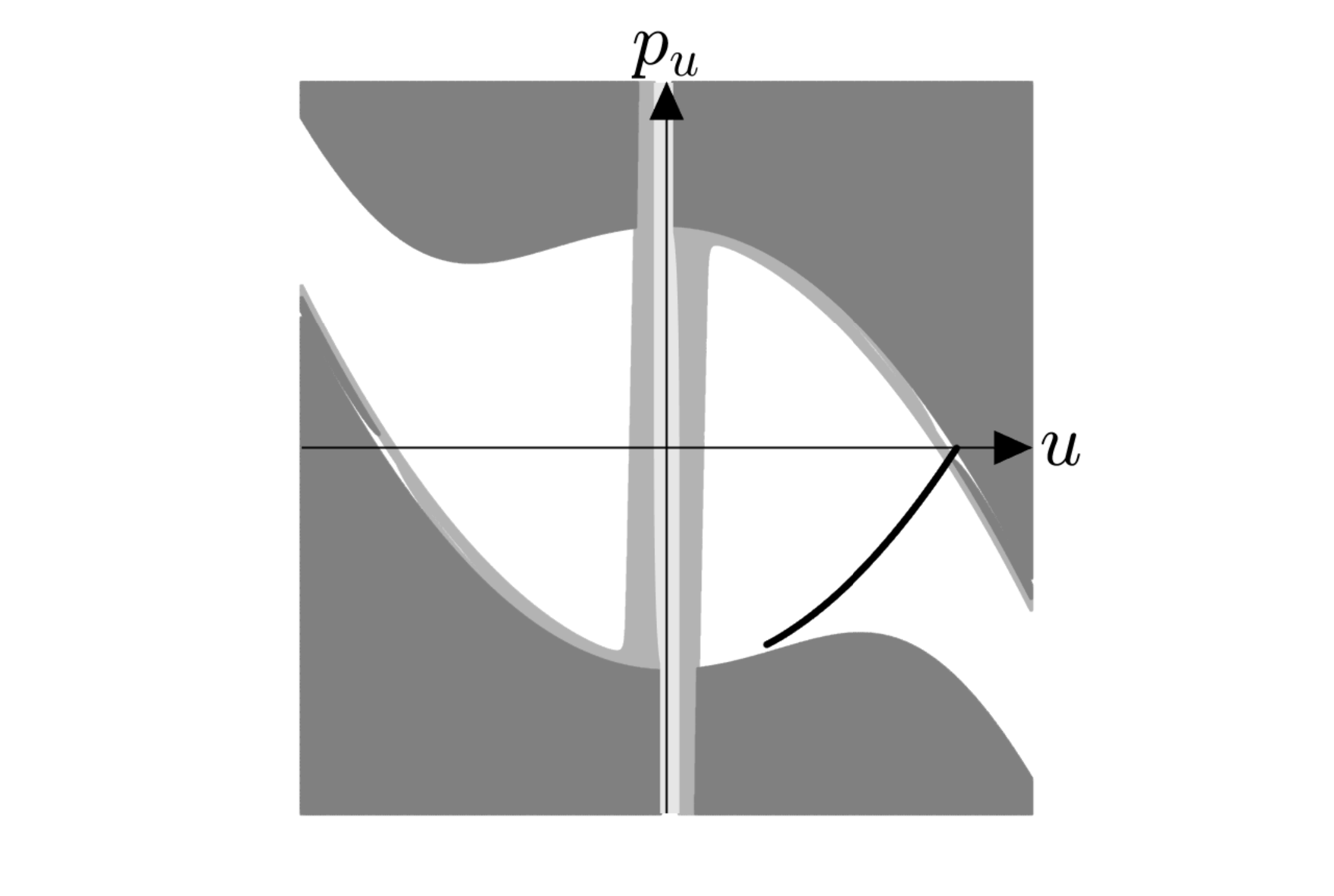}
    \end{subfigure}
    \begin{subfigure}{0.47\textwidth}
      \includegraphics[width=1.1\textwidth]{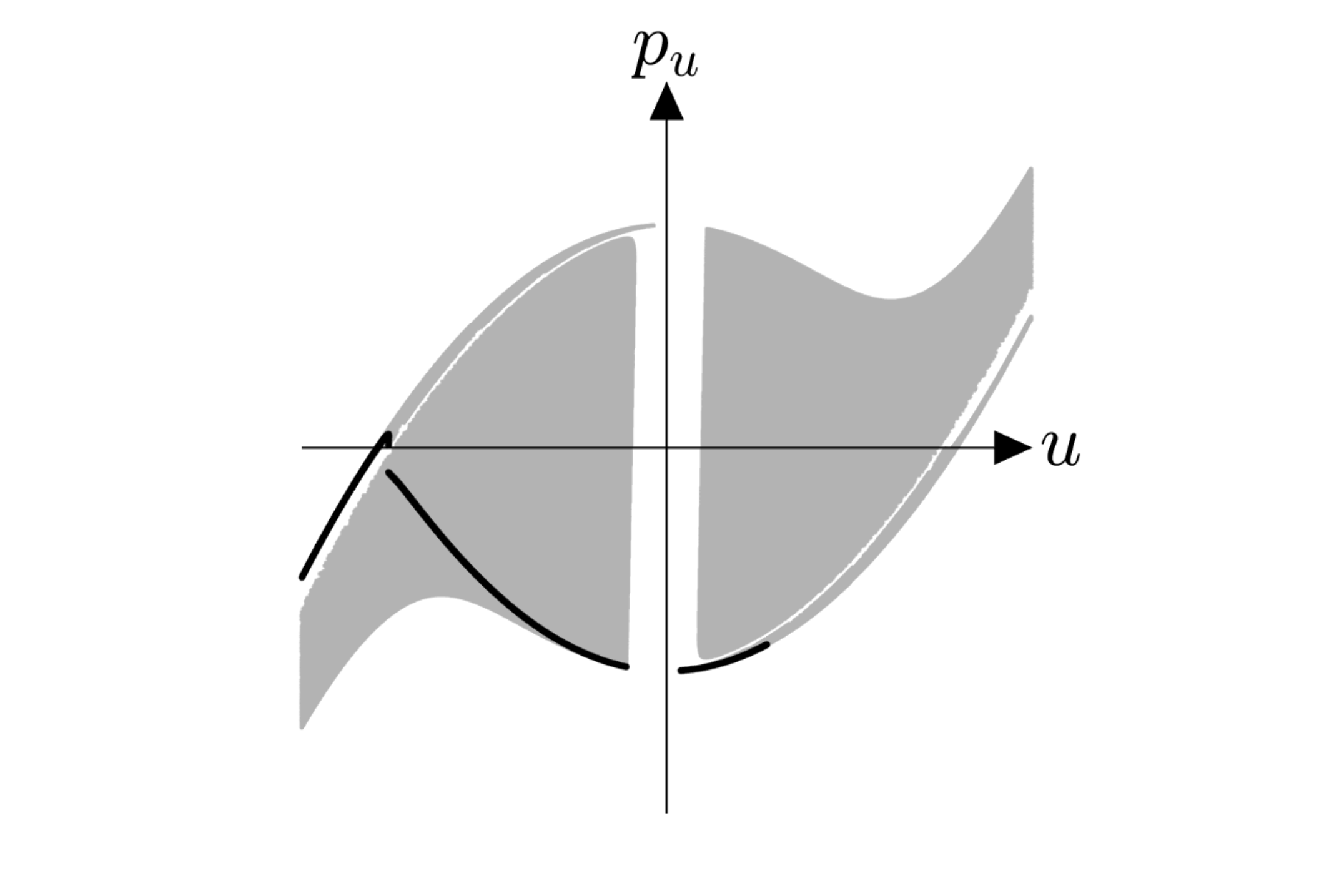}
    \end{subfigure}
    \begin{subfigure}{0.47\textwidth}
      \includegraphics[width=1.1\textwidth]{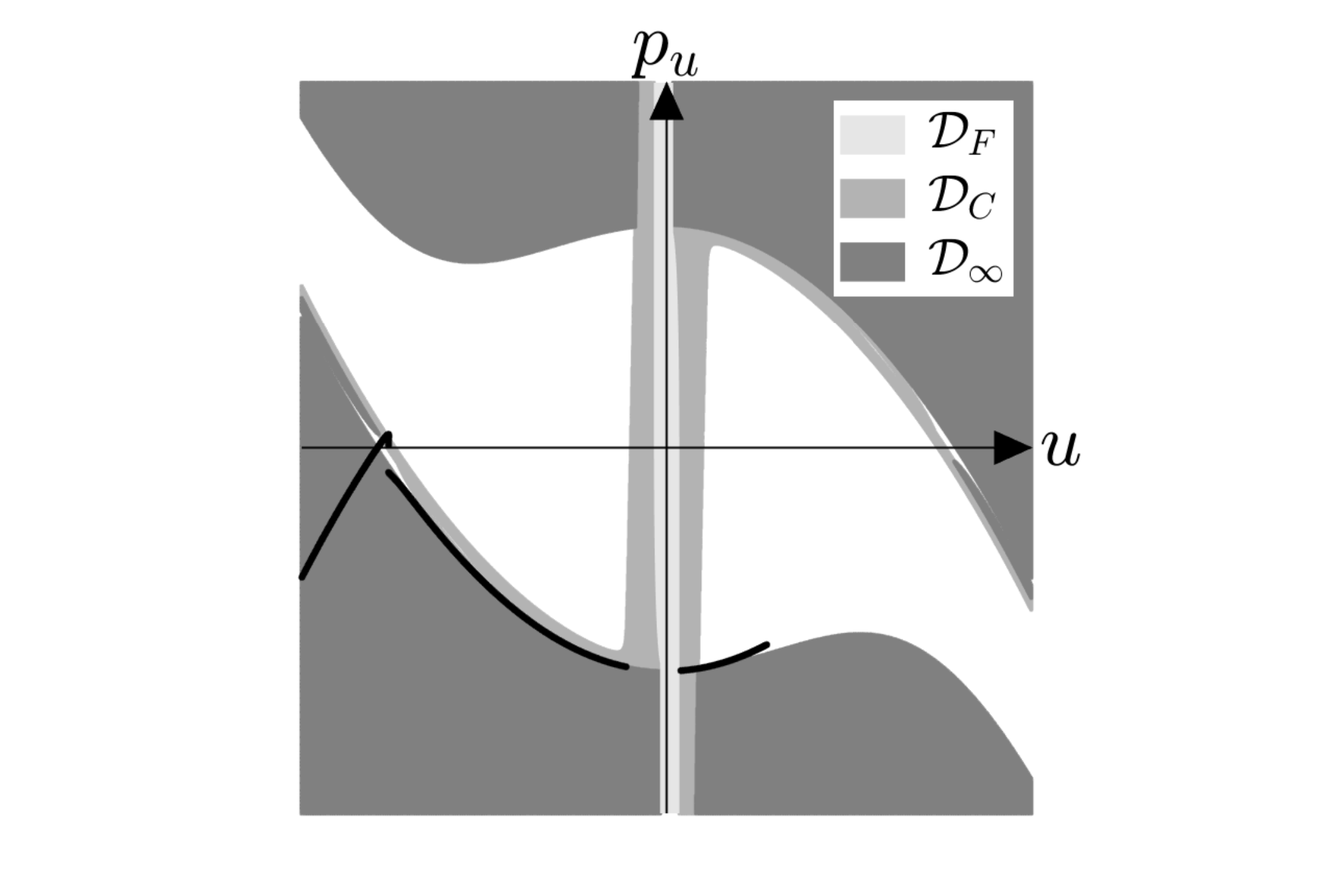}
    \end{subfigure}
    \begin{subfigure}{0.47\textwidth}
      \includegraphics[width=1.1\textwidth]{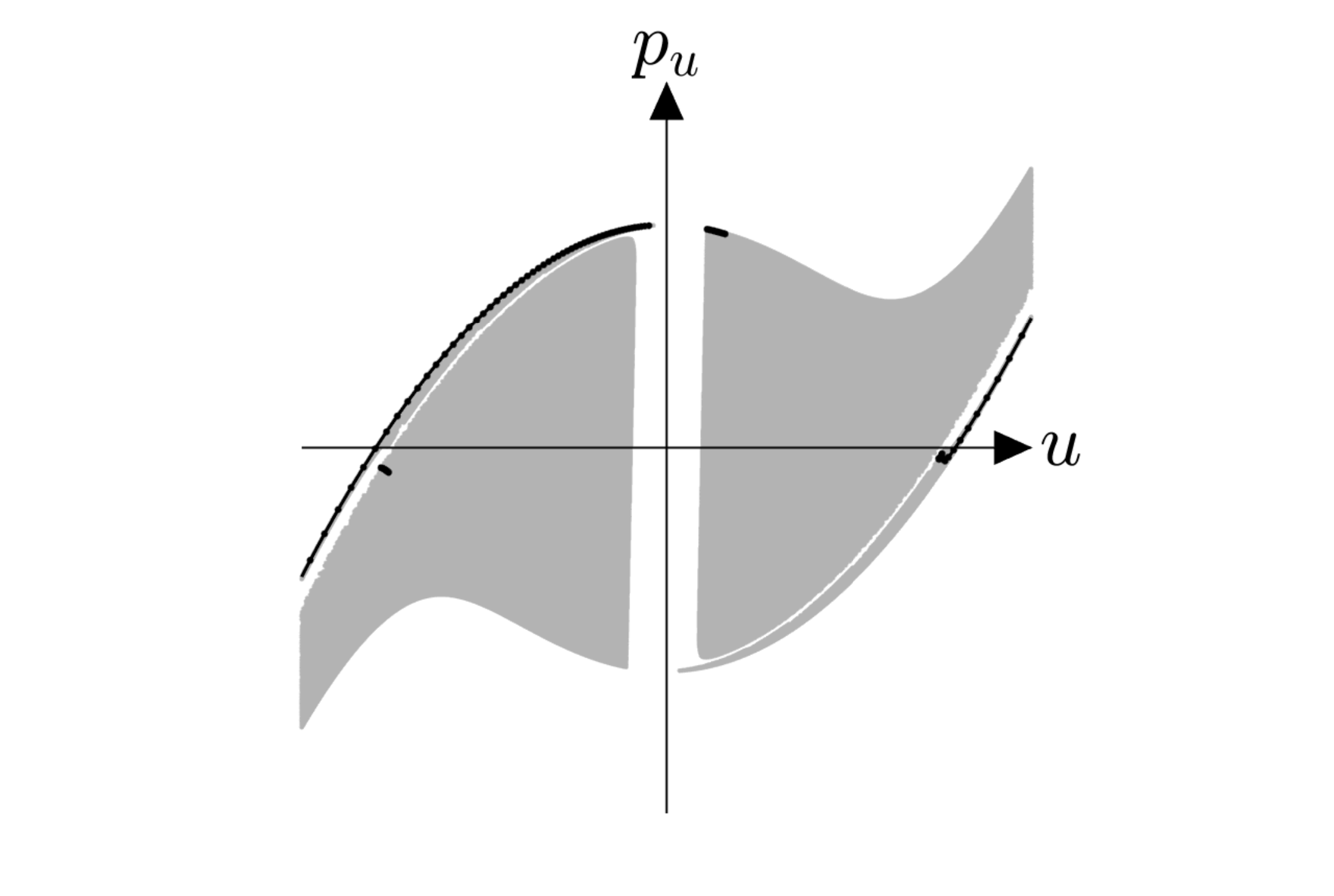}
    \end{subfigure}
  \end{center}
  \caption{Constructing one branch of the unstable manifold of
    ${\bm{\hat \upsilon}}_1$. In the background there are the domain
    ${\cal D}$ and the forbidden regions ${\cal D}_{F}, {\cal D}_{C},
    {\cal D}_\infty$ (left), the image of the map $\mathfrak{S}({\cal
      D})$ (right). The initial primary, drawn in the top left figure,
    is located in the smaller rectangle in
    Figure~\ref{fig:ssdom}. Here, $\llS=348600$
    $\kilo\cubic\meter\per\squaren\second$, $f=9.12\times 10^{-9}$
    $\kilo\meter\per\squaren\second$.}
  \label{fig:unstManConstr}
\end{figure}
	
Figure~\ref{fig:unstManConstr} shows the first steps of the
construction of a branch of the unstable manifold. The initial primary
$V_0$ (top left) is first iterated once under the map (top right). Its image
$V_1$ intersects the forbidden region ${\cal D}_{C}$ (middle left),
and at the second iteration of $\mathfrak{S}$ three components are
left (middle right), two of which lie in the $\{u<0\}$
half-plane. Again, their image $V_2$ intersects the forbidden regions
${\cal D}_{C}$, ${\cal D}_{\infty}$ (bottom left) and the number of components increases at the
successive iteration (bottom
right).

In Figure~\ref{fig:ss} we draw the four (disconnected) branches of the
stable and unstable manifold of the fixed point
$\fpD$.

\begin{figure}[h!]
  \begin{center}
    \includegraphics[width=\textwidth]{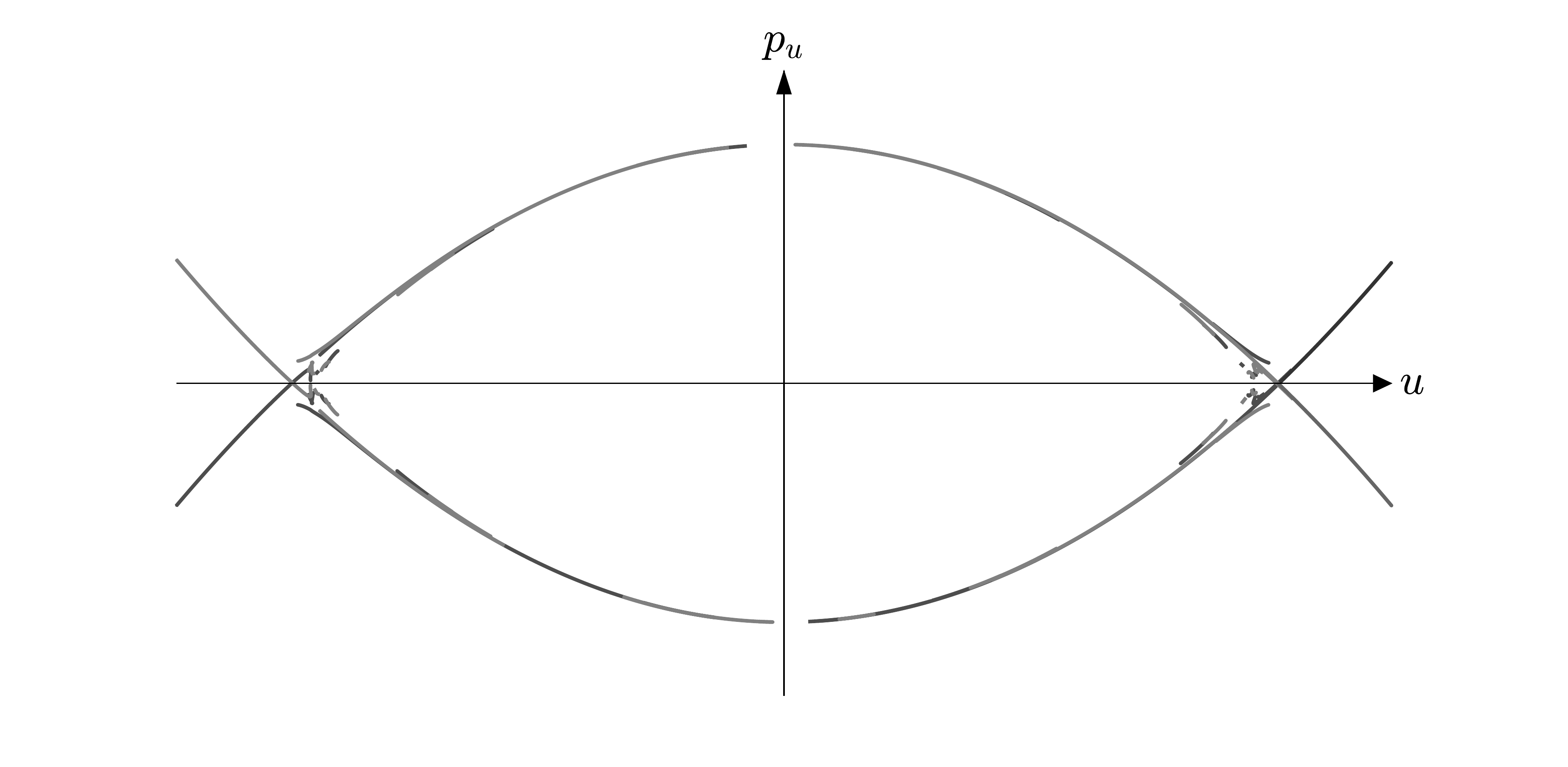}
    \caption{Stable and unstable invariant manifolds of ${\bm{\upsilon}_1}$.}
    \label{fig:ss}
  \end{center}
\end{figure}

\begin{figure}[h!]
  \begin{center}
    \begin{subfigure}{1\textwidth}
      \includegraphics[width=\textwidth]{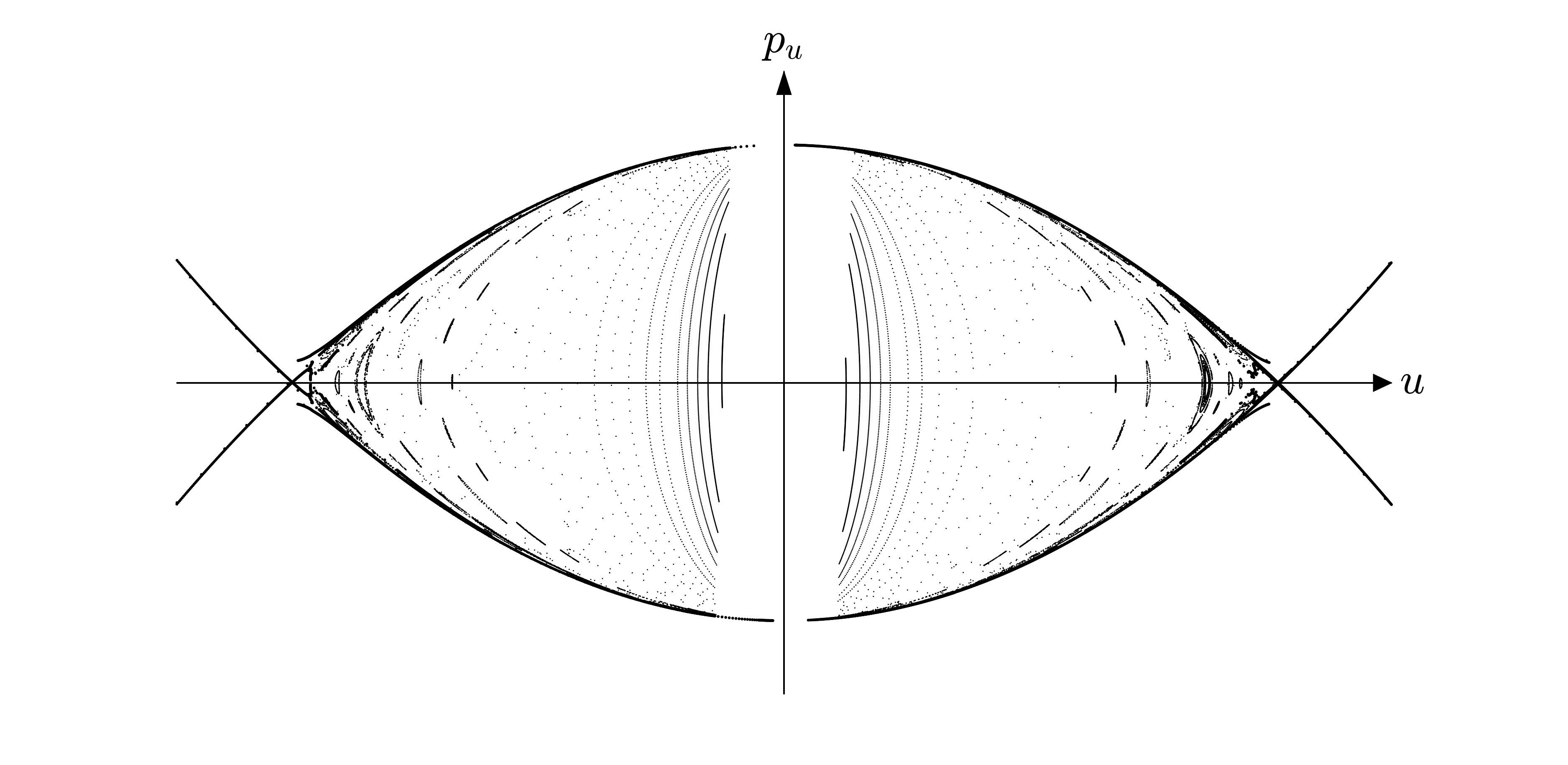}
    \end{subfigure}
    \begin{subfigure}{1\textwidth}
      \includegraphics[width=\textwidth]{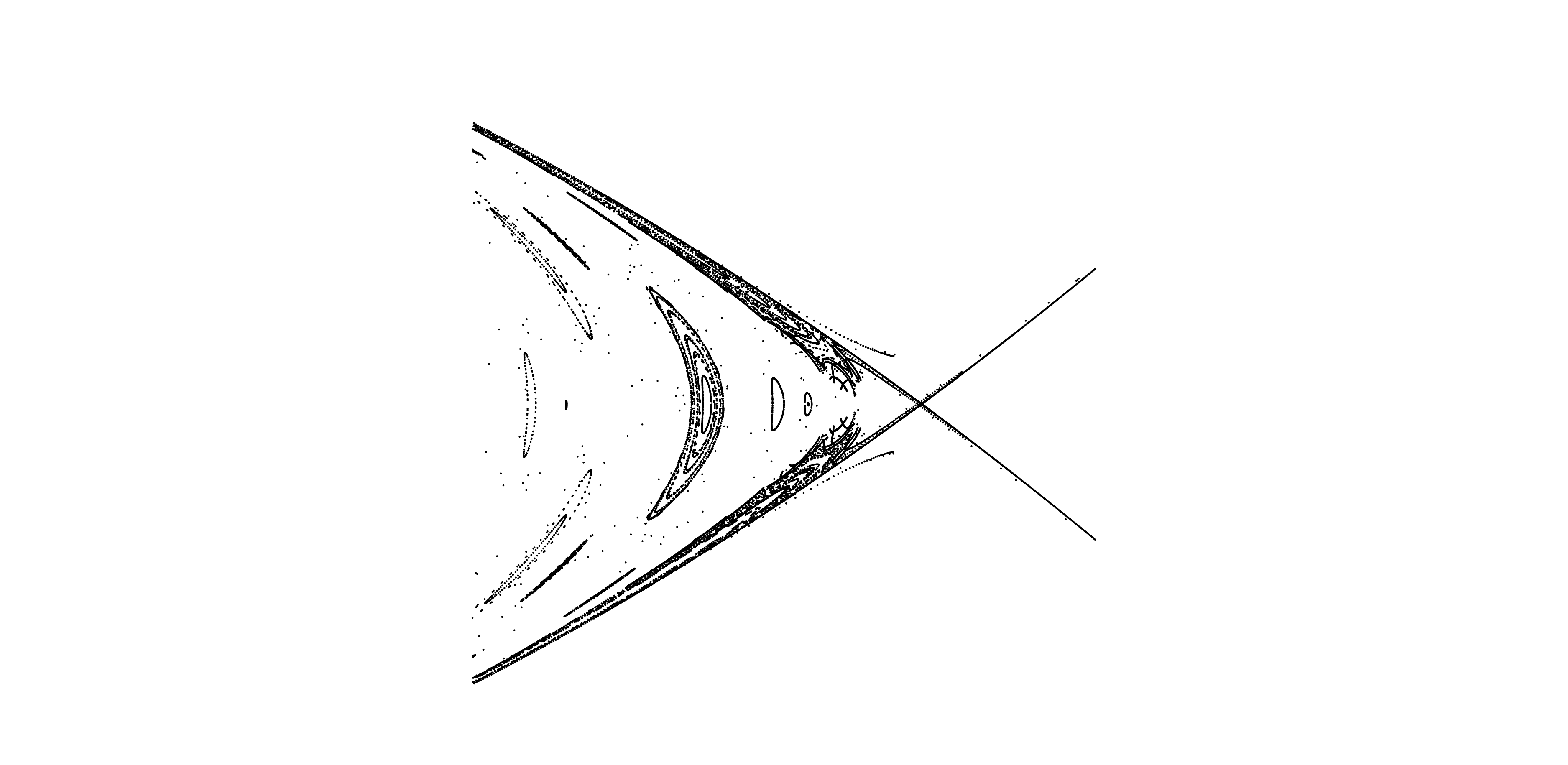}
    \end{subfigure}
    \caption{On the top we show a global picture of the Sun-shadow map, with $\llS=348600$
      $\kilo\cubic\meter\per\squaren\second$, $f=9.12\times 10^{-9}$
      $\kilo\meter\per\squaren\second$. At the bottom we zoom in the region
      close to $\fpD$.}
    \label{fig:ssen}
  \end{center}
\end{figure}

\section{Conclusions and open questions}

In this paper we have investigated the Sun-shadow dynamics, which is
defined by patching together Kepler's and Stark's dynamics. After
reviewing some relevant features of Stark's problem, we prove the
existence of a family of periodic orbits of brake type. Then, we
introduce the Sun-shadow map, by fixing a quantity which is not
conserved along the flow. This map is differentiable but is not area
preserving. Its domain shows fine structures that underlie interesting
phenomena when the map is iterated many times. There is numerical evidence that the fixed
points $\fpD$, $\fpS$ related to the periodic orbits of brake type are
hyperbolic; their invariant manifolds are constructed by an algorithm
specifically created for this purpose. A global picture of this map is
drawn in Figure~\ref{fig:ssen}, where an enhancement of the region
close to the point $\fpD$ is also displayed. We observe evidence of regular
and chaotic behaviour, with the presence of several islands: we checked
that some of them surround periodic points. In the central region,
where we have smaller values of $|u|$, the plotted points show a regular
structure, similar to the
one of the phase portrait in Figure~\ref{fig:levelCurves}. On the
other hand, the regular behaviour seems to be lost in a neighbourhood
of $\fpD$, along the stable and unstable branches of its invariant manifold.

This study opens some interesting questions about the Sun-shadow
dynamics, which deserve to be further investigated.
First, we may wonder whether the winding number associated to the
trajectories corresponding to one iteration of the map is bounded from
below, see Figure~\ref{fig:ssdomEn}.
Another interesting question is whether we can show that the islands
appearing in Figure~\ref{fig:ssen} correspond to invariant curves
around fixed or periodic points.
Moreover, we can ask ourselves whether Melnikov's method can
be adapted to prove the existence of chaotic dynamics in this case,
where the invariant manifolds of the fixed points $\fpD$, $\fpS$ are
made of several connected components (maybe infinitely many) due to
the presence of the forbidden regions ${\cal D}_\infty$, ${\cal
  D}_{C}$, see Figure~\ref{fig:unstManConstr}.
Finally, possible future developments of this work are the extension
of the Sun-shadow dynamics to the three-dimensional case, the
inclusion of other perturbations (e.g. the Earth oblateness), and
the study of the effect of the penumbra.

\section{Acknowledgements}
The authors have been partially supported by the MSCA-ITN Stardust-R,
Grant Agreement n. 813644 under the H2020 research and innovation
program.  GFG and GB also acknowledge the project MIUR-PRIN 20178CJA2B
``New frontiers of Celestial Mechanics: theory and applications'', and
the GNFM-INdAM (Gruppo Nazionale per la Fisica Matematica).

\appendix
\section{Appendix}
\label{s:app}

\begin{table}[H]
	\centering
	\caption{Stark's problem: boundaries of the four regions in the $(\llS, h_s/\sqrt{f})$ plane. The features of the trajectories in the $(x,y)$ plane are qualitatively described. $i$ is the imaginary unit.}
	\begin{tabular}{c!{\color{GrayS}\vrule}c}
		
		\rowcolor{Gray}
		\textbf{Regions $I$, $II$} & $\llS=-\mu$; $\enS/\sqrt{f} \in (0,+\infty)$\\
		\hline
		$V(v)$, $U(u)$ roots &
		${v}_1 >0$, ${v}_2 \in i\mathbb{R}$, ${u}_1=0$, ${u}_2 \in i\mathbb{R}$\\
		$v$,$u$ variable &
		$v\in[-{v}_1,{v}_1]$, $u\in(-\infty,+\infty)$ \\
		{trajectories type} &
		two types: periodic, brake, passing through the origin; \\
		& asymptotic to the periodic orbit in the future or in the past\\[3pt]		
		\rowcolor{Gray}
		\textbf{Regions $I$, $IV$ } & $\llS=-\mu$; $\enS/\sqrt{f} \in (-\infty,0)$\\
		\hline
		$V(v)$, $U(u)$ roots &
		${v}_1 >0$, ${v}_2 \in i\mathbb{R}$, ${u}_1>{u}_2=0$ \\
		$v$, $u$ variable &
		$v\in[-{v}_1,{v}_1]$, $u\in(-\infty,-u_1]\cup\{0\}\cup[u_1,+\infty)$ \\
		{trajectories type} & {two types: brake, periodic, passing through the origin; }\\ 
		& {unbounded, self-intersecting, not encircling the origin}\\[3pt]	
		\rowcolor{Gray}
		\textbf{Regions $II$, $IV$} & $\llS=(-\mu,\mu)$; $\enS/\sqrt{f}= -\sqrt{2(\mu+\llS)}$\\
		\hline
		$V(v)$, $U(u)$ roots &
		${v}_1 >0$, ${v}_2 \in i\mathbb{R}$, ${u}_1={u}_2$, $u_1,u_2>0$  \\
		$v$, $u$ variable &
		$v\in[-{v}_1,{v}_1]$, $u\in(-\infty,+\infty)$ \\ 
		{trajectories type} & two types: periodic, brake; asymptotic to the \\
		&  periodic orbit in the future or in the past\\[3pt]
		\rowcolor{Gray}
		\textbf{Regions $II$, $III$} & $\llS=\mu$; $\enS/\sqrt{f} \in (0,+\infty)$\\
		\hline
		$V(v)$, $U(u)$ roots &
		${v}_1 > {v}_2=0$, ${u}_1,{u}_2 \in \mathbb{C} \setminus \R$ \\
		$v$, $u$ variable &
		$	v\in[-{v}_1,{v}_1]$, $u\in(-\infty,+\infty)$ \\ 
		{trajectories type} & {two types: unbounded, not self-intersecting;}\\
		& {unbounded with $y=0$ and $x\geq0$}\\[3pt]
		\rowcolor{Gray}
		\textbf{Region $II$} & $\llS=\mu$; $\enS/\sqrt{f} \in(-\sqrt{2(\mu+\llS)},0)$\\
		\hline
		$V(v)$, $U(u)$ roots &
		${v}_1=0$, $v_2\in i\R$, ${u}_1,{u}_2 \in \mathbb{C}\setminus\R$ \\
		$v$, $u$ variable &
		$v=0$, $u\in(-\infty,+\infty)$ \\
		{trajectories type} & {unbounded with $y=0$ and $x\geq0$}\\[3pt]
		
		\rowcolor{Gray}
		\textbf{Region $IV$} & $\llS=\mu$; $\enS/\sqrt{f} \in(-\infty,-\sqrt{2(\mu+\llS)})$\\
		\hline
		$V(v)$, $U(u)$ roots &
		${v}_1 =0$, ${v}_2 \in i\mathbb{R}$, ${u}_1>{u}_2>0$  \\
		$v$, $u$ variable &
		$v=0$, $u\in(-\infty,-{u}_1]\cup[-{u}_2,{u}_2]\cup[{u}_1,+\infty)$ \\
		{trajectories type} & {two types: brake, periodic, passing through the origin; }\\
		& {unbounded, with $y=0$, $x> 0$}\\[3pt]
		
		\rowcolor{Gray}
		\textbf{Region $III$} & $\llS\in(\mu,+\infty)$; $\enS/\sqrt{f} = \sqrt{-2(\mu-\llS)}$\\
		\hline
		$V(v)$, $U(u)$ roots &
		${v}_1 ={v}_2$, $v_1,v_2>0$, ${u}_1,{u}_2\in \mathbb{C}\setminus\R$  \\
		$v$, $u$ variable &
		$v\in\{\pm v_1\}$, $u\in(-\infty,+\infty)$ \\
		{trajectories type} & {unbounded, not self-intersecting, parabolic}\\
	\end{tabular}
	\label{BelRegB}
\end{table}

\begin{table}[H]
	\centering
	\caption{Stark's problem: boundary points in the $(\llS, h_s/\sqrt{f})$ plane. The features of the trajectories in the $(x,y)$ plane are qualitatively described. $i$ is the imaginary unit.}
	\begin{tabular}{c!{\color{GrayS}\vrule}c}
		
		\rowcolor{Gray}
		\textbf{Regions $I$, $II$, $IV$} & $\llS=-\mu$; $\enS/\sqrt{f} = 0$\\
		\hline
		$V(v)$, $U(u)$ roots &
		${v}_1 >0$, ${v}_2=0$, ${u}_1=0$, ${u}_2 \in i\mathbb{R}$\\
		$v$,$u$ variable &
		$v\in[-{v}_1,{v}_1]$, $u\in(-\infty,+\infty)$ \\
		{trajectories type} &
		two types: periodic, brake, passing through the origin; \\ & asymptotic to the periodic orbit in the future or in the past\\[3pt]				
		\rowcolor{Gray}
		\textbf{Regions $II$, $III$} & $\llS=\mu$; $\enS/\sqrt{f}=0$\\
		\hline
		$V(v)$, $U(u)$ roots &
		${v}_1=0$, $v_2=0$, ${u}_1,{u}_2 \in \mathbb{C}\setminus\R$ \\
		$v$, $u$ variable &
		$v=0$, $u\in(-\infty,+\infty)$ \\
		{trajectories type} & {unbounded with $y=0$ and $x\geq0$}\\[3pt]
		\rowcolor{Gray}
		\textbf{Regions $II$, $IV$} & $\llS=\mu$; $\enS/\sqrt{f} = -\sqrt{2(\mu+\llS)}$\\
		\hline
		$V(v)$, $U(u)$ roots &
		${v}_1 =0$, ${v}_2 \in i\mathbb{R}$, ${u}_1= {u}_2$, $u_1,u_2>0$  \\
		$v$, $u$ variable &
		$v=0$, $u\in(-\infty,+\infty)$ \\
		{trajectories type} & fixed point; asymptotic to the fixed point in \\
		& the future and in the past with $y=0$, $x\geq 0$\\
	\end{tabular}
	\label{BelRegBP}
\end{table}

\bibliographystyle{plain}
\bibliography{refs}

\end{document}